\documentclass[11pt]{article}

\usepackage{hyperref}
\usepackage{graphicx}
\usepackage{mathrsfs}
\hypersetup{colorlinks, linkcolor=blue, urlcolor=red, filecolor=green, citecolor=blue}

\usepackage{amsmath}
\usepackage{amssymb}
\usepackage{latexsym}
\usepackage{amsthm}
\usepackage{fullpage}
\usepackage{graphicx}
\usepackage{color}

\usepackage[utf8]{inputenc}

\newcommand{\step}[1]{\medskip\noindent\textbf{Step #1. }}

 {\begin{list}{}%
         {\setlength{\leftmargin}{#1}}%
         \item[]%
 }
 {\end{list}}

\newtheorem{theorem}{Theorem}[section]
\newtheorem{proposition}[theorem]{Proposition}
\newtheorem{lemma}[theorem]{Lemma}

\newtheorem{definition}{Definition}[section]
\newtheorem{remark}[theorem]{Remark}
\newtheorem{assumption}{Assumption}[section]

\author{Laura Lauerbach\footnote{Institute of Mathematics,
University of W\"{u}rzburg, Emil-Fischer-Str.~40, 97074 W\"{u}rzburg, Germany},
Stefan Neukamm\footnote{Department of Mathematics, Technische Universit\"{a}t Dresden, 01069 Dresden, Germany},
Mathias Sch\"{a}ffner\footnote{Institute of Mathematics, University of Leipzig, 04081 Leipzig, Germany},
Anja Schl\"{o}merkemper\footnote{Institute of Mathematics,
University of W\"{u}rzburg, Emil-Fischer-Str.~40, 97074 W\"{u}rzburg, Germany}
}

\title{Mechanical behaviour of heterogeneous nanochains in the $\Gamma$-limit of stochastic particle systems}

\begin{document}
\maketitle

\begin{abstract}
Nanochains of atoms, molecules and polymers have gained recent interest in the experimental sciences. This article contributes to an advanced mathematical modeling of the mechanical properties of nanochains that allow for heterogenities, which may be impurities or a deliberately chosen composition of different kind of atoms. We consider one-dimensional systems of particles which interact through a large class of convex-concave potentials, which includes the classical Lennard-Jones potentials. 

We allow for a stochastic distribution of the material parameters and investigate the effective behaviour of the system as the distance between the particles tends to zero. The mathematical methods are based on $\Gamma$-convergence, which is a suitable notion of convergence for variational problems, and on ergodic theorems as is usual in the framework of stochastic homogenization. The allowed singular structure of the interaction potentials causes mathematical difficulties that we overcome by an approximation. We consider the case of $K$ interacting neighbours with $K\in \mathbb{N}$ arbitrary, i.e., interactions of finite range. 
\end{abstract}

\tableofcontents

\section{Introduction}

In this article we extend results on the passage from discrete to continuous systems for particle chains that show heterogeneities on the microscopic level. For instance, this can be due to fault atoms, to different bonds between the same kind of elements (e.g. $\cdots$C$\equiv$C=C$\equiv$C=C$\equiv$C$\cdots$ \cite{LaTorreetal})
or to more advanced compositions of the nanochains. One-dimensional chains of atoms find applications in carbon atom wires \cite{ReviewCAW, Nairetal,Zhangetal} or as Au-chains on substrates \cite{WagnerEtal2018}. Further, they serve as toy-models for higher dimensional systems. From the mathematical point of view, one-dimensional systems have the advantage that the particles are monotonically ordered.

In \cite{LauerbachSchaeffnerSchloemerkemper2017}, three of the current authors proved a $\Gamma$-convergence result for the passage from discrete to continuous systems in the setting of periodic heterogeneities. In the current paper, instead, we investigate the stochastic setting which provides a more general approach and thus allows for more applications, as, e.g., in the case of fault atoms or composite materials. \\

We consider a lattice model for a one-dimensional chain of $n+1$ atoms (or other particles) that interact via random potentials of Lennard-Jones type. The interactions are of finite range in the sense that the $i$th atom may interact with the atoms with labels $i+1$ up to $i+K$, $K\in\mathbb{N}$. The random interaction potentials are assumed to have a stationary and ergodic distribution. We describe configurations of the chain with help of a deformation relative to a reference configuration where the atoms are  equidistributed with lattice spacing $\lambda_n=\frac{1}{n}$, that is, $u:\lambda_n\mathbb Z\cap[0,1]\to\mathbb R$. To each such deformation we associate an ``atomistic'' energy given by the sum of all interaction potentials. In the special case of nearest-neighbour interactions, it may take the form
\begin{equation}\label{e.ste1}
  E_n(u)=\lambda_n\sum_{i=0}^{n-1}J_{i}\left(\dfrac{u^{i+1}-u^{i}}{\lambda_n}\right),
\end{equation}
see \eqref{eq:energy} for the general case.
Typically the reference configuration is not an energy minimizing state (nor an equilibrium state). Moreover, in view of spatial heterogeneity, minimizers of the energy are typically  non-trivial (in the sense that they are given by configurations of the chain with non-equidistributed atoms). \\

As a main result we prove $\Gamma$-convergence of the atomistic energy to a deterministic integral functional with a spatially homogeneous, convex potential as the number of atoms $n$ tends to infinity, see Theorem~\ref{Thm:nullteordnung}. This limit includes a passage from a discrete to a continuous model as well as a quenched (almost sure) stochastic homogenization result. \\

While we prove our mathematical results for $K$ interacting neighbours, $K\in \mathbb{N}$, and a large class of interaction potentials, cf.\ Remark~\ref{rem:Gay-Berne}, we here give a more detailed description of our result in the special case of chains with nearest-neighbour  interactions whose potentials are given by independent and identically distributed classical Lennard-Jones interactions. The latter are defined by the two-parameter family $J_{LJ}(z):=A/z^{12}-B/z^6$ with $A,B>0$. These potentials may equivalently be represented in the form  $J_{LJ}(z)=\varepsilon\left(\frac{\delta}{z}\right)^6\left[\left(\frac{\delta}{z} \right)^6-2\right]$ for suitable parameters $(\delta,\varepsilon)\in(0,\infty)^2$, see Figure~\ref{fig:LJ} for the meaning of these parameters. We consider the energy functional \eqref{e.ste1} with potentials $J_i(z)=\varepsilon_i\left(\frac{\delta_i}{z}\right)^6\left[\left(\frac{\delta_i}{z} \right)^6-2\right]$ where the parameters $(\delta_i,\varepsilon_i)$ are independent and identically distributed and bounded from above and away from $0$, say $(\delta_i,\varepsilon_i)\in (\frac1C,C)^2$ for some  constant $C>0$. 
\begin{figure}[t]
	\centering
	\includegraphics[width=0.4\textwidth]{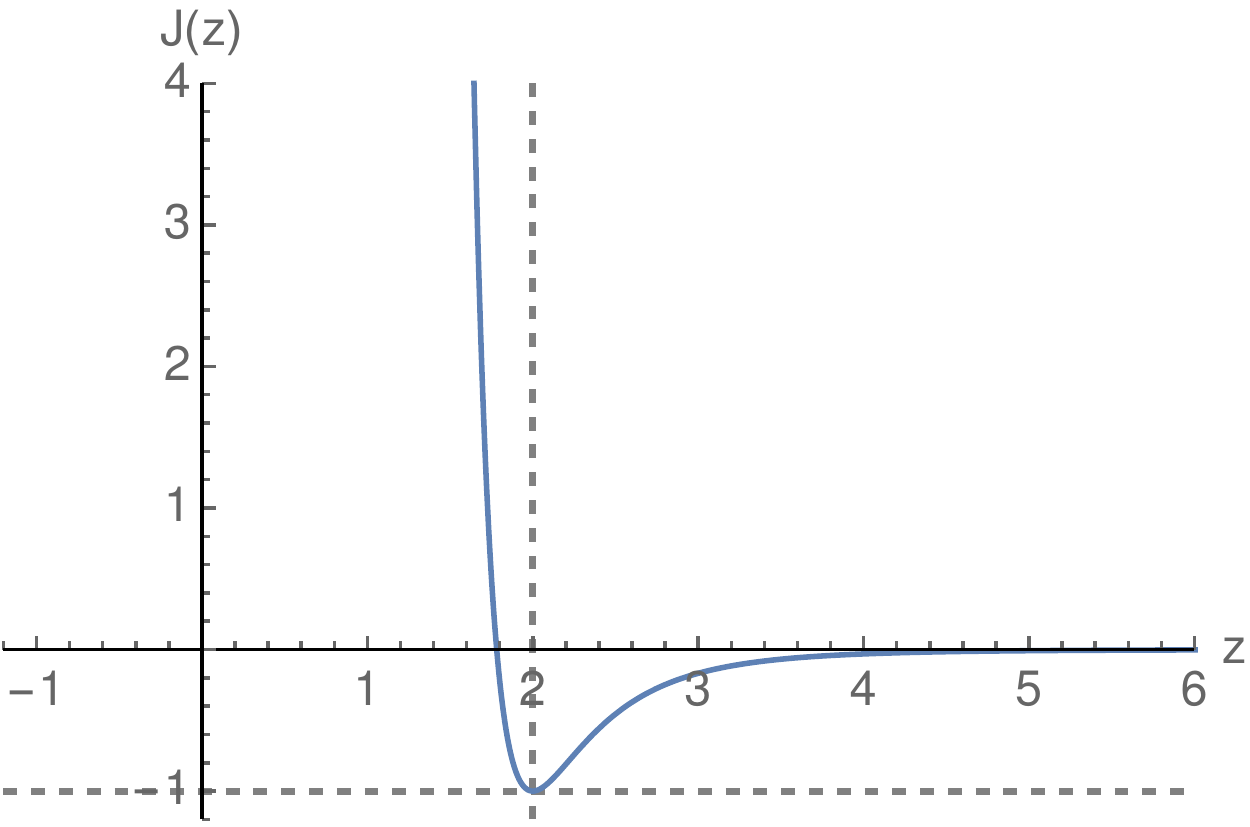}  
	\caption{Lennard-Jones potential $J_{LJ}$, with $\delta=2$ and $\varepsilon=1$.}
	\label{fig:LJ}
\end{figure}

Theorem~\ref{Thm:nullteordnung} then yields that the energy functional \eqref{e.ste1} subject to displacement boundary conditions $\Gamma$-converges as $n\to\infty$ to a deterministic continuum limit of the form
\begin{align*}
	E_{\mathrm{hom}}(u)=\int_{0}^{1}J_{\mathrm{hom}}(u'(x))\,\mathrm{d}x
\end{align*}
with the homogenized energy density 
\begin{align*}
  J_{\mathrm{hom}}(z)=\lim\limits_{N\rightarrow\infty}\dfrac{1}{N}\inf\left\{\sum_{i=0}^{N-1}J_i\left(z+\phi^{i+1}-\phi^{i}\right)\,|\, \phi^i\in\mathbb{R},\ \phi^0=\phi^N=0  \right\}, \quad z\in\mathbb{R}.
\end{align*}
It will turn out that 
\begin{equation*}
  J_{\mathrm{hom}}(z)
  \begin{cases}
    >-\mathbb{E}[\varepsilon]&\text{for }z<\mathbb{E}[\delta],\\
    =-\mathbb{E}[\varepsilon]&\text{for }z \geq \mathbb{E}[\delta],
  \end{cases}
\end{equation*}
where $\mathbb{E}$ denotes the expectation, see Figure~\ref{fig:intro2}, Propositions~\ref{Prop:Jhom} and \ref{Prop:Jhomzgross}. In particular, under compressive boundary conditions, i.e., $u(1)-u(0)<\mathbb{E}[\delta]$, minimizers are affine (in contrast to the corresponding minimizer of the associated atomistic energy). The precise form of $J_{\mathrm{hom}}$ for $z<\mathbb{E}[\delta]$ depends on the underlying distribution of the parameters $(\delta_i,\varepsilon_i)$. For illustration we consider two examples. In the first example, see Figure~\ref{fig:intro1}, we assume that
$(\delta_i,\varepsilon_i)$ is uniformly distributed in $\Omega_1:=[1,2]\times[3,4]$; in the second example we suppose that $\delta_i$ and $\varepsilon_i$ are independent and two-valued with $\mathbb P(\delta_i=1)=0.9$, $\mathbb P(\delta_i=6)=0.1$, $\mathbb P(\varepsilon_i=3)=0.9$, and $\mathbb P(\varepsilon_i=8)=0.1$. In both cases we obtain $\mathbb{E}[\delta] = 1.5$ and $\mathbb{E}[\varepsilon]=3.5$. Therefore, while $J_{\mathrm{hom}}$ coincides for $z\geq\mathbb{E}[\delta]$ in both examples, they differ for $z<\mathbb{E}[\delta]$, see Figure~\ref{fig:intro2}.\\

\begin{figure}
\begin{minipage}{0.48\textwidth}
\centering
	\includegraphics[width=0.8\textwidth]{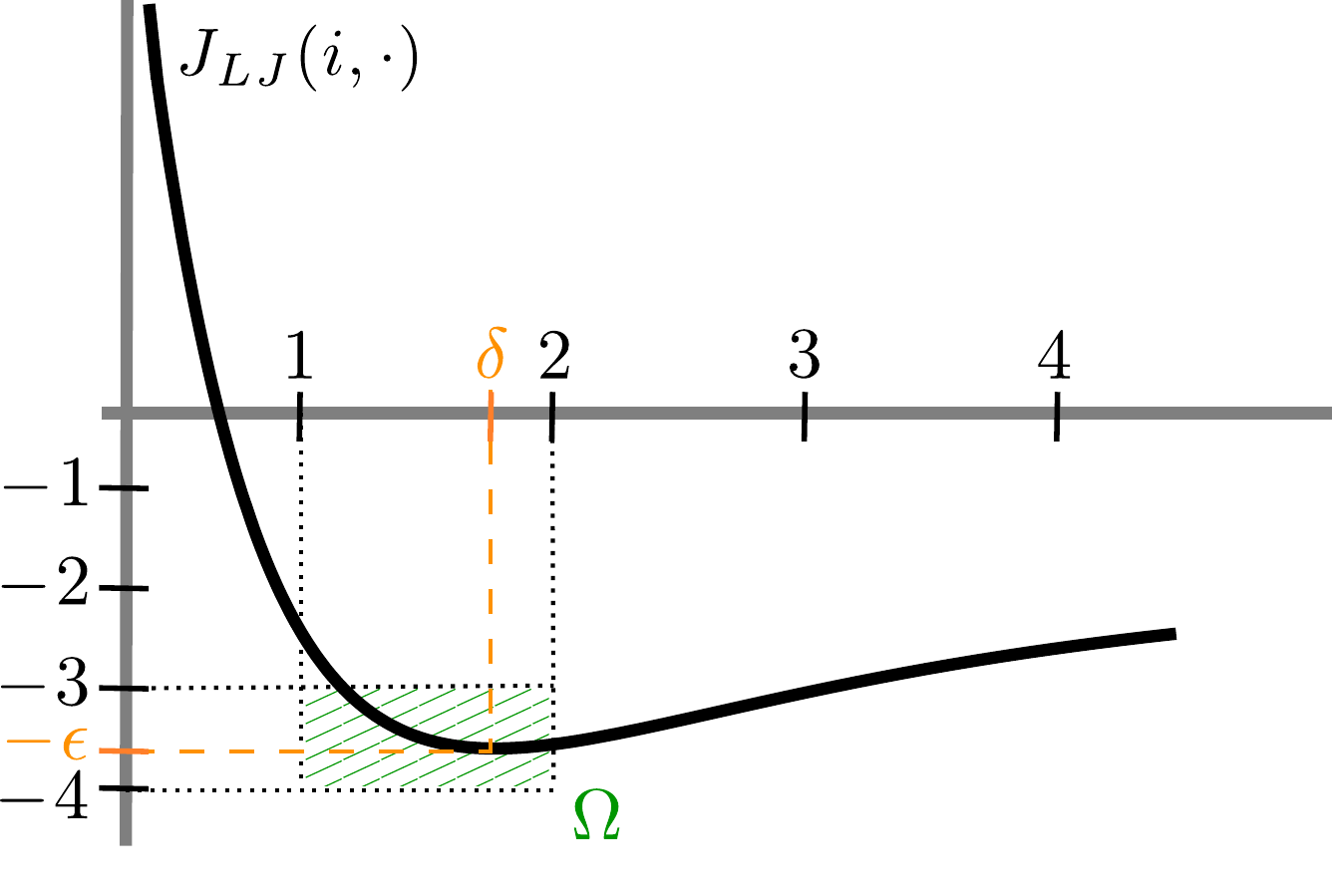}  
	\caption{A prototypical potential $J_{LJ}(i,\cdot)$ in the setting $(\Omega_1,\mathcal{F}_1,\mathbb{P}_1)$.}
	\label{fig:intro1}
\end{minipage} \hspace{0.02\textwidth}
\begin{minipage}{0.48\textwidth}
\centering
	\includegraphics[width=0.84\textwidth]{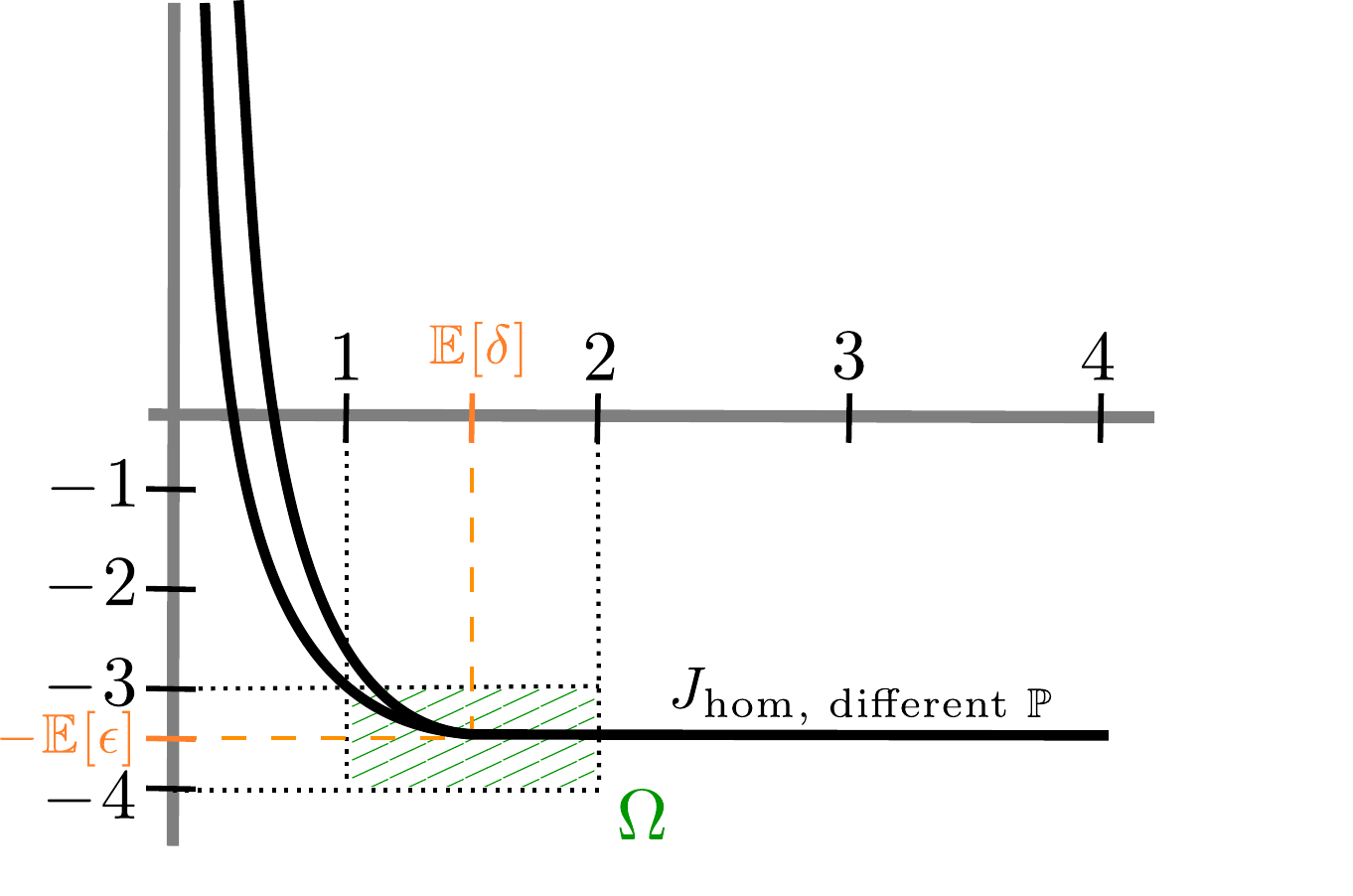}  
	\caption{Two different functions $J_{\mathrm{hom}}$ related to different probability distributions with identical expectation values $\mathbb{E}[\delta]$ and $\mathbb{E}[\varepsilon]$.}
	\label{fig:intro2}
\end{minipage}
	\end{figure}

Before we comment on related literature, we outline the strategy of our proof and the structure of the paper. In Section~\ref{sec:2} we introduce the class of all Lennard-Jones type interactions, the random setting, the energy functional $E_n^\ell$ on a suitable space of piecewise affine functions, and an infinite cell formula that is needed in the homogenized functional in the continuum limit. In Section~\ref{sec:main} we state the $\Gamma$-limit result for the functional $E_n^\ell$ with respect to the $L^1(0,1)$-topology  and properties of the homogenized energy density $J_{\mathrm{hom}}$. In the continuum limit, the system can show cracks, i.e., discontinuities of the deformation $u$. In order to gain further information on the cracks, analysis of a differently scaled energy functional is needed which takes surface energy contributions due to the formation of cracks into account. This will be the topic of a forthcoming paper, see also \cite{BraidesCicalese2007, ScardiaSchloemerkemperZanini2011, ScardiaSchloemerkemperZanini2012} and the introduction of \cite{CarioniFischerSchloemerkemper2018} for further related literature.\\

The proof of Theorem~\ref{Thm:nullteordnung}, which we provide in Section~\ref{sec:proof-thm}, requires various extensions of known homogenization results since the interaction potentials are allowed to blow up and are not convex.  To this end, we introduce a Lipschitz continuous approximation of the interaction potentials and a corresponding infinite cell formula $J_{\mathrm{hom}}^L$ in Section~\ref{sec:cellformula}. In this approximating setting, we can apply the subadditive ergodic theorem by Akcoglu and Krengel \cite{AkcogluKrengel1981}, cf.\ Theorem~\ref{Thm:AkcogluKrengel}, in the proof of Proposition~\ref{Prop:existencejhom}. In Proposition~\ref{Prop:Jhomzklein} we then show that $J_{\mathrm{hom}}$ is given as the limit of $J_{\mathrm{hom}}^L$ as $L\to\infty$ and hence exists. In Proposition~\ref{Prop:Jhom} we assert various properties of $J_{\mathrm{hom}}$ that are needed in the proof of the $\Gamma$-limit. In particular it turns out that $J_{\mathrm{hom}}$ is deterministic. 

 The proof of the liminf-inequality requires the introduction of two artificial coarser scales that help to deal with the randomness of the system and of the $K$ interacting neighbours, respectively, and makes the proof challenging from a technical point of view. The limsup-inequality is first shown for affine deformations, then for piecewise affine and finally for $W^{1,1}$-functions. By a relaxation theorem of Gelli \cite{Gelli}, cf.\ Theorem~\ref{PropGelli} in the appendix, the limsup-inequality is then true also for $BV$-functions.\\

Our work embeds into the existing literature as follows. For the related work on one-dimensional particle systems for convex-concave potentials and fracture mechanics we refer again to the introduction of \cite{CarioniFischerSchloemerkemper2018}. 
While the case of next-to-nearest neighbour interactions is quite standard, the case of $K$-interaction neighbours with $K >2$ is more involved, see \cite{BraidesLewOrtiz2006,SchaeffnerSchloemerkemper2015}. 

The periodic case of heterogeneous materials and their homogenization was investigated in \cite{BraidesGelli2006,LauerbachSchaeffnerSchloemerkemper2017}.  
Stochastic homogenization combined with passages from discrete to continuous systems has been the topic of research for other growth and coercivity conditions also in higher dimensions, see \cite{AlicandroCicaleseGloria2011,
NeukammSchaeffnerSchloemerkemper2017}.
In \cite{IosifescuLichtMichaille2001}, the authors also deal with a stochastic setting in one dimension. However, due to their growth conditions, Lennard-Jones potentials and other potentials with singular behaviour are excluded in their work, as are interactions beyond nearest neighbours. Further, in \cite{IosifescuLichtMichaille2001}  a discrete probability density is considered, while we allow the set of all interaction potentials to be infinite, even uncountable, which refers to a continuous probability density and thus a larger applicability of our results. The drawback is that our proofs are more technical and in particular need the approximation of the interaction potentials.

\section{Discrete model -- stochastic Lennard-Jones interactions} \label{sec:2}

\begin{figure}[t]
\centering
		\includegraphics[width=0.4\linewidth]{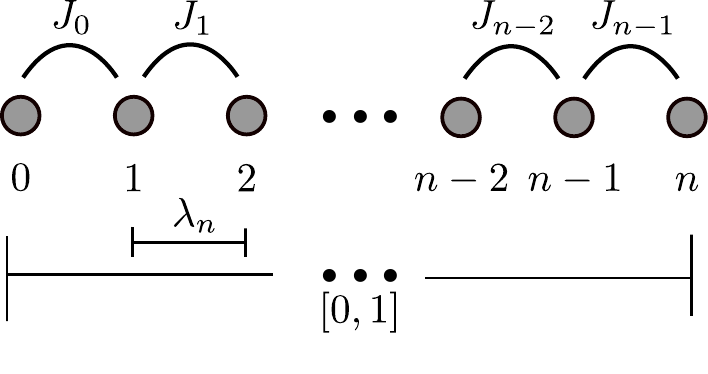}
		\caption{Chain of $n+1$ atoms with reference position $x_n^{i}=i\lambda_n$. The potentials $J_i$ describe the nearest neighbour interaction of atom $i$ and $i+1$. The characteristic length scale is $\lambda_n=\frac{1}{n}$ and the interval is $[0,1]$.}
\label{fig:kette}
\end{figure}

We consider a one dimensional lattice given by $\lambda_n\mathbb{Z}\cap[0,1]$, where $\lambda_n=\frac{1}{n}$. We regard this as a chain of $n+1$ atoms. The reference position of the $i$-th atom is referred to as $x_n^i:=i\lambda_n$. The deformation of the atoms is denoted by $u_n: \lambda_n\mathbb{Z}\cap[0,1]\rightarrow\mathbb{R}$; we write $u(x_n^i)=u^i$ for short. In the passage from discrete systems to their continuous counterparts it turns out to be useful to identify the discrete functions with their piecewise affine interpolations. We define 
\begin{align*}
	\mathcal{A}_n:=\left\lbrace u \in C([0,1]): u\ \text{is affine on}\ (i,i+1)\lambda_n,\ i\in\{0,1,...,n-1\} \right\rbrace
\end{align*}
as the set of all piecewise affine functions which are continuous. The interaction potentials of this chain are introduced in the following.

\subsection{Lennard-Jones type potentials} \label{sec:LJ}

The interaction potentials we consider belong to a large class $\mathcal{J}(\alpha,b,d,\Psi)$ of functions that includes the classical Lennard-Jones potential, which is the reason why we refer to the considered interaction potentials as being of \textit{Lennard-Jones type}. It is defined as follows.
\begin{definition}
 Fix $\alpha\in(0,1]$, $b>0$, $d\in[1,+\infty)$ and  a convex function $\Psi:\mathbb R\to[0,+\infty]$ satisfying
 \begin{equation}
 \label{growthpsi}
  \lim_{z\to0+} \Psi(z)=+\infty.
 \end{equation} 
We denote by $\mathcal J=\mathcal J(\alpha,b,d,\Psi)$ the class of functions $J:\mathbb R\to \mathbb R\cup\{+\infty\}$ which satisfy the following properties:
\begin{itemize}
	\item[(LJ1)] (Regularity and asymptotic decay) The function $J$ is lower semicontinuous, $J\in C^{0,\alpha}_{loc}(0,\infty)$ and
	\begin{equation}
	 \lim_{z\to0+} J(z)=\infty\quad \mbox{as well as}\quad J(z)=\infty\mbox{ for $z\leq0$.}
	\end{equation}
	\item[(LJ2)] (Convex bound, minimum and minimizer) $J$ has a unique minimizer $\delta$ with $\delta\in (\tfrac1d,d)$ and $J(\delta)<0$, and is strictly convex on $(0,\delta)$. Moreover, $\|J\|_{L^\infty(\delta,\infty)}<b$ and it holds	\begin{align}
	\label{LJ2abschaetzung}
		\tfrac1d\Psi(z)-d\leq J(z)\leq d\max\{\Psi(z),|z|\}\quad\text{for all}\ z\in(0,+\infty).
	\end{align}
	\item[(LJ3)] (Asymptotic behaviour) It holds
	\begin{equation}
	 \lim_{z\to\infty} J(z)=0.
	\end{equation}
\end{itemize}
\end{definition}
\begin{remark}
(i) The choice of the assumptions allows inter alia for the classical Lennard-Jones potential as well as for a potential with a hard core. The hard core is achieved by a shift of the domain from $(0,+\infty)$ to $(z_0,+\infty)$, with $z_0>0$. This can be easily done by shifting the Lennard-Jones potentials as $J(z-z_0)$, which does not affect the $\Gamma$-convergence result. More general, the result holds true for any shift of the domain from $(0,+\infty)$ to $(z_0,+\infty)$, with $z_0\in\mathbb{R}$.

(ii) The assumption of an open domain is not restrictive. Allowing also for $\mathrm{dom}J=[0,+\infty)$, the proofs get much easier, because then we have $J\in C^{0,\alpha}(0,+\infty)$, $0<\alpha\leq 1$, on its domain. This simplifies the handling of the ergodic theorems and the approximation of the potentials (introduced below) is not necessary. Therefore, $J_{\mathrm{hom}}^{L}$ can be derived directly from the ergodic theorems and the $\Gamma$-convergence result is the same.
\end{remark}

A combination of the convexity and monotonicity of $J$ in $(0,\delta)$ with the growth condition \eqref{LJ2abschaetzung} implies that $J$ is (locally) Lipschitz continuous in $(0,\delta)$. More precisely, we have
\begin{lemma}
\label{lem:lipschitz}
Fix $\alpha\in(0,1]$, $b>0$, $d\in[1,\infty)$ and  a convex function $\Psi:\mathbb R\to[0,\infty]$ satisfying \eqref{growthpsi}. There exists a function $C_{\rm Lip}:(0,d)\to[0,\infty)$ depending only on $d$ and $\Psi$ such that the following is true. Let $J\in\mathcal J(\alpha,b,d,\Psi)$ be given and let $\delta$ be its unique minimizer. Then it holds 
\begin{equation}
    \|J\|_{{\rm Lip}(\rho,\delta)}:=\sup_{\substack{x,y\in(\rho,\delta)\\ x\neq y}}\left|\frac{J(y)-J(x)}{y-x}\right|\leq C_{\rm Lip}(\rho)
\end{equation}
\end{lemma}

\begin{remark}\label{rem:Gay-Berne} By defining the class of Lennard-Jones type potentials, a wide range of interaction potentials is covered, e.g., the classical Lennard-Jones. This is of interest, because the special choice of the potential depends on the field of application, for example atomistic or molecular interactions. Besides, even the classical Lennard-Jones potential is just an approximation and not an exact measured or mathematically derived formula, therefore it is useful to have assumptions keeping the main features of the potential without fixing it in detail. Further, the Gay-Berne potential is included in this setting. This is a modified 12--6 Lennard-Jones potential where the parameters of the potential depend on the relative orientation of the interacting, e.g., ellipsoidal particles, see for example \cite{BerardiFavaZannoni1998,MorenoRazoetal2011,Orlandietal2016}. In the following chapter, we introduce a stochastic setting, which can be use to model this orientation parameter as a random variable. 
\end{remark}

\subsection{Random setting}

The randomness enters the model through the interaction potentials. On the chain of atoms described above, we consider random interactions up to order $K$, with $K\in\mathbb{N}$. An illustration is shown in Figure \ref{fig:random}. The random interaction potentials $\{J_j(\omega,i,\cdot)\}_{i\in\mathbb{Z},\,j=1,...,K}$, $J_j(\omega,i,\cdot):\mathbb{R}\to(-\infty,+\infty]$, are of Lennard-Jones type, specified in Section~\ref{sec:LJ}; they are assumed statistically homogeneous and ergodic. This is a standard way in the theory of stochastic homogenization, see, e.g., \cite{AlicandroCicaleseGloria2011}. This assumptions are phrased as follows: Let $(\Omega,\mathcal{F},\mathbb{P})$ be a probability space. This space can be discrete or continuous with uncountably many different elements in the set $\Omega$. We assume that the family $(\tau_i)_{i\in\mathbb{Z}}$ of measurable mappings $\tau_i:\Omega\rightarrow\Omega$ is an additive group action, i.e.,
\begin{itemize}
	\item (group property) $\tau_0\omega=\omega$ for all $\omega\in\Omega$ and $\tau_{i_1+i_2}=\tau_{i_1}\tau_{i_2}$ for all $i_1,i_2\in\mathbb{Z}$. 
\end{itemize}
Additionally, we assume that we have:
\begin{itemize}
	\item (stationarity) The group action is measure preserving, that is $\mathbb{P}(\tau_iB)=\mathbb{P}(B)$ for every $B\in\mathcal{F}$, $i\in\mathbb{Z}$.
	\item (ergodicity) For all $B\in\mathcal{F}$, it holds $(\tau_i(B)=B\ \forall i\in\mathbb{Z})\Rightarrow\mathbb{P}(B)=0\ \text{or}\ \mathbb{P}(B)=1$.
\end{itemize}

\begin{figure}[t]
\centering
		\includegraphics[width=0.7\linewidth]{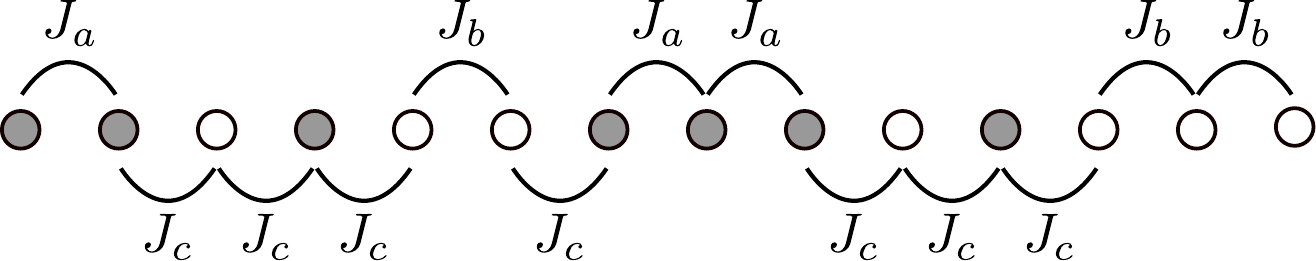}
		\caption{Randomly arranged chain of atoms. The nearest neighbour interaction potential of two grey atoms is labelled by $J_a$, that of two white atoms by $J_b$ and that between a white and a grey one by $J_b$. Since the atoms are randomly distributed, this holds for the potentials as well.}
\label{fig:random}
\end{figure}
For each $j=1,...,K$, we define $\tilde{J}_j:\Omega\to\mathcal{J}(\alpha,b,d,\Psi),\ \omega\mapsto \tilde{J}_j(\omega)(z)=:\tilde{J}_j(\omega,z)$, measurable in $\omega$. This maps the sample space into the set of Lennard-Jones potentials. Then, we define
\begin{align*}
	J_j(\omega,i,\cdot) :=\tilde{J}_j(\tau_{i}\omega,\cdot)\quad\text{ for all } i\in\mathbb{Z},\ \omega\in\Omega,\ j=1,...,K .
\end{align*}
This means that every mapping $\tau_i:\Omega\to\Omega$ of the group action is assigned to an atom of the chain and is used to relate the different atoms to different elements of the sample space and therefore to different interaction potentials. 
In the following, we denote $\tilde{J}_j$ simply by $J_j$, for better readability. It will be clear from the context which function is meant. 
We also define notation for the minimizers
\begin{align*}
\delta_j(\omega):=\mathrm{argmin}_{z\in\mathbb{R}}\left\{\tilde{J}_j(\omega,z) \right\}, \quad \delta_j(\tau_i\omega):=\mathrm{argmin}_{z\in\mathbb{R}}\left\{J_j(\omega,i,z) \right\}, \quad \text{for all}\ i\in\mathbb{Z},\ j=1,...,K.
\end{align*}

The potentials have to fulfil one more property, dealing with the H\"{o}lder estimates, where $[f]_{C^{0,\alpha}(A)}$ is the Hölder coefficient of the function $f$. This assumption is phrased in the following.

\begin{itemize}
\item[(H1)] For every $j=1,...,K$ it holds true that $\mathbb{E}\left[\left[ J_j(\omega,\cdot)\right]_{C^{0,\alpha}(\delta_j(\omega),+\infty)}\right]<\infty$.
\end{itemize}
This condition occurs with respect to the infinite set of potentials. When dealing with finitely many different potentials, this property is fulfilled automatically. Especially, (H1) is fulfilled if the H\"{o}lder coefficients on $(\delta,+\infty)$ of all functions $J\in\mathcal{J}$ are uniformly bounded.

\begin{remark}
\label{Rm:integrabilityandexpectationvalues}
(LJ2) provide a uniform bound of $\delta_j(\omega)$ and of $J_j(\omega,\delta_j(\omega))$. Therefore, the random variables $\delta_j(\omega)$ and $J_j(\omega,\delta_j(\omega))$ are integrable.

By definition of integrability, the expectation value exists for these random variables, which we denote by $\mathbb{E}[\delta_j]$ and $\mathbb{E}[J_j(\delta_j)]$. Regarding the expectation value as an ensemble mean, we can also say something about the sample average. This connection is strongly related to ergodicity and is explained in the next proposition.
\end{remark}
Define, for better readability, the random variable $C_j^H(\omega):=\left[J_j(\omega,\cdot) \right]_{C^{0,\alpha}(\delta_j(\omega),+\infty)}$, that is the Hölder coefficient of the function $J_j(\omega,\cdot)$ on $(\delta_j(\omega),\infty)$. We define some functions, which represent sample averages of the quantities $\delta_j$, $J_j(\delta_j)$ and $C_j^H$. Let $N\in\mathbb{N}$ and $A\subset \mathbb{R}$ an interval. 
\begin{align*}
\delta_j^{(N)}(\omega,A)&:=\dfrac{1}{|NA\cap\mathbb{Z}|}\sum_{i\in NA\cap\mathbb{Z}}\delta_j(\tau_i\omega),\quad C_j^{H,(N)}(\omega,A):=\dfrac{1}{|NA\cap\mathbb{Z}|}\sum_{i\in NA\cap\mathbb{Z}}C_j^H(\tau_i\omega),\\
J_j(\delta_j)^{(N)}(\omega,A)&:=\dfrac{1}{|NA\cap\mathbb{Z}|}\sum_{i\in NA\cap\mathbb{Z}}J_j(\omega,i,\delta_j(\tau_i\omega)).
\end{align*}

\begin{proposition}
	\label{Prop:averages}
	Assume that Assumption \ref{Ass:stochasticLJ} below is satisfied. Then, there exists $\Omega_1\subset\Omega$ with $\mathbb{P}(\Omega_1)=1$ such that for all $\omega\in\Omega_1$, all $j=1,...,K$ and for all $A=[a,b]$ with $a,b\in\mathbb{R}$ the limits
	\begin{align*}
		\mathbb{E}[\delta_j]&=\lim\limits_{N\to\infty}\delta_j^{(N)}(\omega,A),\quad \mathbb{E}[C_j^H]=\lim\limits_{N\to\infty}C_j^{H,(N)}(\omega,A),\\
		\mathbb{E}[J_j(\delta_j)]&=\lim\limits_{N\to\infty}J_j(\delta_j)^{(N)}(\omega,A)
	\end{align*}
	exist in $\overline{\mathbb{R}}$ and are independent of $\omega$ and the interval $A$. 
\end{proposition}

\begin{proof}
The proof follows the same argumentation as the proofs in \cite[Proposition 1]{DalMasoModica1985} or \cite[Lemma 3.9]{NeukammSchaeffnerSchloemerkemper2017}. The main difference is the formula for the approximation argument from intervals with edges in $\mathbb{Z}$ to general intervals in $\mathbb{R}$. This formula will be briefly given. For $\delta_j^{(N)}(\omega,A)$, we get for all intervals $B\subset A$ the inequality 
	\begin{align*}
		\delta_j^{(N)}(\omega,A)\leq \delta_j^{(N)}(\omega,B)+\dfrac{|N(A\setminus B)\cap\mathbb{Z}|}{|NA\cap\mathbb{Z}|}C,
	\end{align*}
	using $\delta_j(\omega)\leq d$ due to (LJ2), which can be seen by the calculation
	\begin{align*}
	\delta_j^{(N)}(\omega,A)&=\dfrac{1}{|NA\cap\mathbb{Z}|}\sum_{i\in NA\cap\mathbb{Z}}\delta_j(\tau_i\omega)\leq \dfrac{1}{|NA\cap\mathbb{Z}|}\sum_{i\in NB\cap\mathbb{Z}}\delta_j(\tau_i\omega)+\dfrac{1}{|NA\cap\mathbb{Z}|}\sum_{i\in NA\cap\mathbb{Z}}C\\
	&\leq \dfrac{1}{|NB\cap\mathbb{Z}|}\sum_{i\in NB\cap\mathbb{Z}}\delta_j(\tau_i\omega)+\dfrac{|N(A\setminus B)\cap\mathbb{Z}|}{|NA\cap\mathbb{Z}|}C.
	\end{align*}
The formula for $J_j(\delta_j)^{(N)}(\omega,A)$ is exactly the same, since it also holds true that $J_j(\delta_j)(\omega)$ are bounded.\\

The formula for $C_j^{H,(N)}(\omega,A)$ uses $C_j^H(\omega)>0$, as follows:
\begin{align*}
C_j^{H,(N)}(\omega,A)&=\dfrac{1}{|NA\cap\mathbb{Z}|}\sum_{i\in NA\cap\mathbb{Z}}C_j^H(\tau_i\omega)\geq \dfrac{1}{|NA\cap\mathbb{Z}|}\sum_{i\in NB\cap\mathbb{Z}}C_j^H(\tau_i\omega)=\dfrac{|NB\cup\mathbb{Z}|}{|NA\cup\mathbb{Z}|} C_j^{H,(N)}(\omega,B).
\end{align*}
\end{proof}

\subsection{Energy}

The assumptions on the stochastic setting of the chain with Lennard-Jones type interaction potentials are summarized in 
\begin{assumption}
\label{Ass:stochasticLJ}
The following is satisfied:
\begin{itemize}
\item $(\Omega,\mathcal{F},\mathbb{P})$ is a probability space. 
\item The family $(\tau_i)_{i\in\mathbb{Z}}$ of measurable mappings is an additive group action which is stationary and ergodic.
\item Assumptions (H1) holds true.
\item The potentials $J_j$ are of Lennard-Jones type, i.e., $J_j\in\mathcal{J}(\alpha,b,d,\Psi)$.
\end{itemize} 
\end{assumption}

Let $u\in\mathcal{A}_n$ be a given deformation. Then we define the energy of the chain of atoms for $K$ interacting neighbours by
\begin{align} \label{eq:energy}
	E_n(\omega,u):=\sum_{j=1}^{K}\sum_{i=0}^{n-j}\lambda_nJ_j\left(\omega,i,\dfrac{u^{i+j}-u^{i}}{j\lambda_n}\right).
\end{align}
For a given $\ell>0$, we take the boundary conditions into account by considering the functional $E_n^{\ell}:\Omega\times L^1(0,1)\rightarrow(-\infty,+\infty]$ defined by 
\begin{align} 
\label{Enl}
E_n^{\ell}(\omega,u):=\begin{cases}
E_n(\omega,u)\quad&\text{if}\ u\in\mathcal{A}_n\ \text{and}\ 	u(0)=0,\ u(1)=\ell,\\+\infty\quad&\text{else.}
\end{cases}
\end{align}
Its $\Gamma$-limit will involve the function $J_{\mathrm{hom}}:\mathbb{R}\to(-\infty,+\infty]$, which is defined by $J_{\rm hom}(z):=\lim_{N\to\infty}\mathbb E\left[J_{\rm hom}^{(N)}(\cdot,z)\right]$ and turns out to be equal to $J_{\mathrm{hom}}(z)=\lim\limits_{N\rightarrow\infty}J_{\mathrm{hom}}^{(N)}(\omega,z)$, with
\begin{align}
\label{JhomNDef}
\begin{split}
	J_{\mathrm{hom}}^{(N)}(\omega,z):=\dfrac{1}{N}\inf\left\{\sum_{j=1}^{K}\sum_{i=0}^{N-j}J_j\left(\omega,i,z+\dfrac{\phi^{i+j}-\phi^{i}}{j}\right),\ \phi^i\in\mathbb{R},\ \phi^s=\phi^{N-s}=0\right.\\ \text{for}\ s=0,...,K-1  \Bigg\},
	\end{split}
\end{align}
Note that in the stochastic setting this infinite cell formula can not be reduced to a finite cell formula as is the case in the periodic or convex setting. However, the infinite cell formula can equivalently be written as:
\begin{align*}
\begin{split}
	J_{\mathrm{hom}}^{(N)}(\omega,z)=\dfrac{1}{N}\inf\left\{\sum_{j=1}^{K}\sum_{i=0}^{N-j}J_j\left(\omega,i,\frac{1}{j}\sum_{k=i}^{i+j-1}z^{k}\right),\ z^i\in\mathbb{R}, \sum_{i=0}^{N-1}z^{i}=Nz,\ z^s=z^{N-s-1}=z\right.\\   \text{for}\ s=0,...,K-2   \Bigg\}.
	\end{split}
\end{align*}

\section{Continuum limit -- the main result} \label{sec:main}

We provide a $\Gamma$-convergence result for the sequence $(E_n^\ell)$ given in \eqref{Enl}. To this end, we first recall suitable function spaces from \cite{BraidesDalMasoGarroni1999,ScardiaSchloemerkemperZanini2011} which capture the Dirichlet boundary conditions in $u_n(0)=0$ and $u_n(1)=\ell$ in the limit. For given $\ell>0$, we denote by $BV^\ell(0,1)$ the space of bounded variations in $(0,1)$ and in order to measure jumps at the boundary, we set $u(0-)=0$ and $u(1+)=\ell$. For $u\in BV^\ell(0,1)$, we set $S_u:=\{x\in[0,1]\,|\, [u](x)\neq0\}$ with $[u](x):=u(x+)-u(x-)$ where $u(x+)$ for $x\in[0,1)$ and $u(x-)$ for $x\in(0,1]$ are the right and left essential limits at $x$. For $u\in BV(0,1)$, we label by $D^au$ the absolutely continuous part and by $D^su$ the singular part of the measure $Du$ with respect to the Lebesgue measure $\mathcal{L}^1$. Further, the density of $D^au$ is denoted by $u'\in L^1(0,1)$, i.e.~$D^au=u'\mathcal{L}^1$. For $u\in BV^\ell(0,1)$ we extend $D^su$ to $[0,1]$ by 
$$D^su:=\sum_{x\in S_u} [u](x)\delta_x + D^cu,$$
where $D^cu$ denotes the diffusive part of the measure $D^su$. 

\begin{theorem}\label{Thm:nullteordnung}
Assume that Assumption \ref{Ass:stochasticLJ} is satisfied. Let $\ell>0$. Then, there exists a set $\Omega'\subset\Omega$ with $\mathbb{P}(\Omega')=1$ such that for all $\omega\in\Omega'$ the $\Gamma$-limit of $E_n^{\ell}(\omega,\cdot)$ with respect to the $L^1(0,1)$-topology is $E_{\mathrm{hom}}^{\ell}:L^1(0,1)\rightarrow(-\infty,+\infty]$, given by
\begin{align*}
E_{\mathrm{hom}}^\ell(u)=\begin{cases}
	\displaystyle\int_{0}^{1}J_{\mathrm{hom}}(u'(x))\,\mathrm{d}x\quad&\text{if}\ u\in BV^{\ell}(0,1),\ \mathrm{D}^su\geq0,\\
	+\infty\quad&\text{else.}
	\end{cases}
\end{align*}
with 
\begin{align}\label{def:Jhomtheorem}
\begin{split}
	J_{\mathrm{hom}}(z)=\lim\limits_{N\rightarrow\infty}\dfrac{1}{N}\inf\left\{\sum_{j=1}^{K}\sum_{i=0}^{N-j}J_j\left(\omega,i,z+\dfrac{\phi^{i+j}-\phi^{i}}{j}\right),\ \phi^i\in\mathbb{R},\right.\\ \phi^s=\phi^{N-s}=0\ \text{for}\ s=0,...,K-1  \Bigg\}.
\end{split}
\end{align}
Moreover, the minimum values of $E_{n}^{\ell}(\omega,\cdot)$ and $E_{\mathrm{hom}}^\ell$ satisfy
\begin{align*}
 \lim\limits_{n\to\infty}\inf_u E_{n}^{\ell}(\omega,u)=\min_u E_{\mathrm{hom}}^{\ell}(u)=J_{\mathrm{hom}}(\ell).
\end{align*}
\end{theorem}
(The proof of Theorem~\ref{Thm:nullteordnung} can be found in Section~\ref{sec:proof-thm})

\smallskip

The possible blow-up of the interaction potentials combined with their non-convexity prevents to use well-established homogenization methods directly to prove Theorem~\ref{Thm:nullteordnung}. In fact a main preliminary result for the Theorem~\ref{Thm:nullteordnung} is to show that the asymptotic cell formula \eqref{def:Jhomtheorem} is well-defined. For given $z\in\mathbb R$ the limit  \eqref{def:Jhomtheorem} exists by the subadditive ergodic theorem of Akcoglu and Krengel \cite{AkcogluKrengel1981}. However, the exceptional set depends on $z$ in general. Assuming polynomial growth from above on the interaction potentials this problem can be solved by a continuity argument, see e.g.\ \cite{DalMasoModica1985,NeukammSchaeffnerSchloemerkemper2017}. Due to the uncontrolled growth of the Lennard-Jones type interaction, we cannot apply this argument directly and we introduce an additional approximation procedure, see Section~\ref{sec:cellformula} below for more details. 

Next, we state the result regarding the limit \eqref{def:Jhomtheorem} and list some of the main properties of $J_{\rm hom}$. For given $A=[a,b)$, $a,b\in\mathbb{R}$, we set
\begin{align}\label{def:Jhomlocal}
J_{\mathrm{hom}}^{(N)}(\omega,z,A):=\dfrac{1}{|NA\cap\mathbb{Z}|}\inf\left\{\sum_{j=1}^{K}\sum_{\substack{i\in \mathbb Z \cap NA\\ i+j-1\in NA }}J_j\left(\omega,i,z+\dfrac{\phi^{i+j}-\phi^{i}}{j}\right),\,\phi\in \mathcal A_{N,K}^0(A)\right\},
\end{align}
where
\begin{equation}\label{def:ANK0}
 \mathcal A_{N,K}^0(A):=\left\{\phi:\mathbb Z\to\mathbb R,\;\,\mbox{$\phi^i=0$ for $\left| \min\limits_{j\in AN\cap\mathbb{Z}}\{j\}-i \right|<K$ or $\left| \max\limits_{j\in AN\cap\mathbb{Z}}\{j\}+1-i \right|<K$}\right\}
\end{equation}

\begin{proposition}\label{Prop:Jhom}
Assume that Assumption \ref{Ass:stochasticLJ} is satisfied. There exists $\Omega_0\subset\Omega$ with $\mathbb{P}(\Omega_0)=1$ such that the following is true: For all $\omega\in\Omega_0$, $z\in\mathbb{R}$ and $A:=[a,b)$ with $a,b\in\mathbb R$ it holds
\begin{equation}\label{eq:defJhom}
 J_{\rm hom}(z):=\lim_{N\to\infty}\mathbb E\left[J_{\rm hom}^{(N)}(\cdot,z,[0,1))\right]=\lim_{N\to\infty} J_{\rm hom}^{(N)}(\omega,z,A).
\end{equation}	
The map $z\mapsto J_{\mathrm{hom}}(z)$ is convex, lower semicontinuous, monotonically decreasing and satisfies
\begin{align}\label{superlinwachstum}
 \lim\limits_{z\rightarrow-\infty}\dfrac{J_{\mathrm{hom}}(z)}{|z|}=+\infty.
\end{align}
Moreover, it holds for every $\omega\in\Omega_0$ and $A:=[a,b)$, $a,b\in\mathbb R$
\begin{equation}\label{eq:gammalimjhom}
 \Gamma\text{-}\lim_{N\to\infty}J_{\rm hom}^{(N)}(\omega,\cdot,A)=J_{\rm hom}.
\end{equation}	
\end{proposition}

In the case of only nearest neighbour interactions, that is $K=1$, we get a finer result, illustrated in Figure \ref{fig:jhom}.

\begin{proposition}
\label{Prop:Jhomzgross}
	Assume that Assumption \ref{Ass:stochasticLJ} is satisfied and set $K=1$. There exists $\Omega_0\subset\Omega$ with $\mathbb{P}(\Omega_0)=1$ such that the following is true: For all $\omega\in\Omega_0$, $$J_{\mathrm{hom}}(z)=\mathbb{E}[J_1(\delta)] \quad \mbox{ for all } z\geq \mathbb{E}[\delta_1]. $$
\end{proposition}

\begin{figure}
\begin{minipage}{0.48\textwidth}
\centering
	\includegraphics[width=0.9\linewidth]{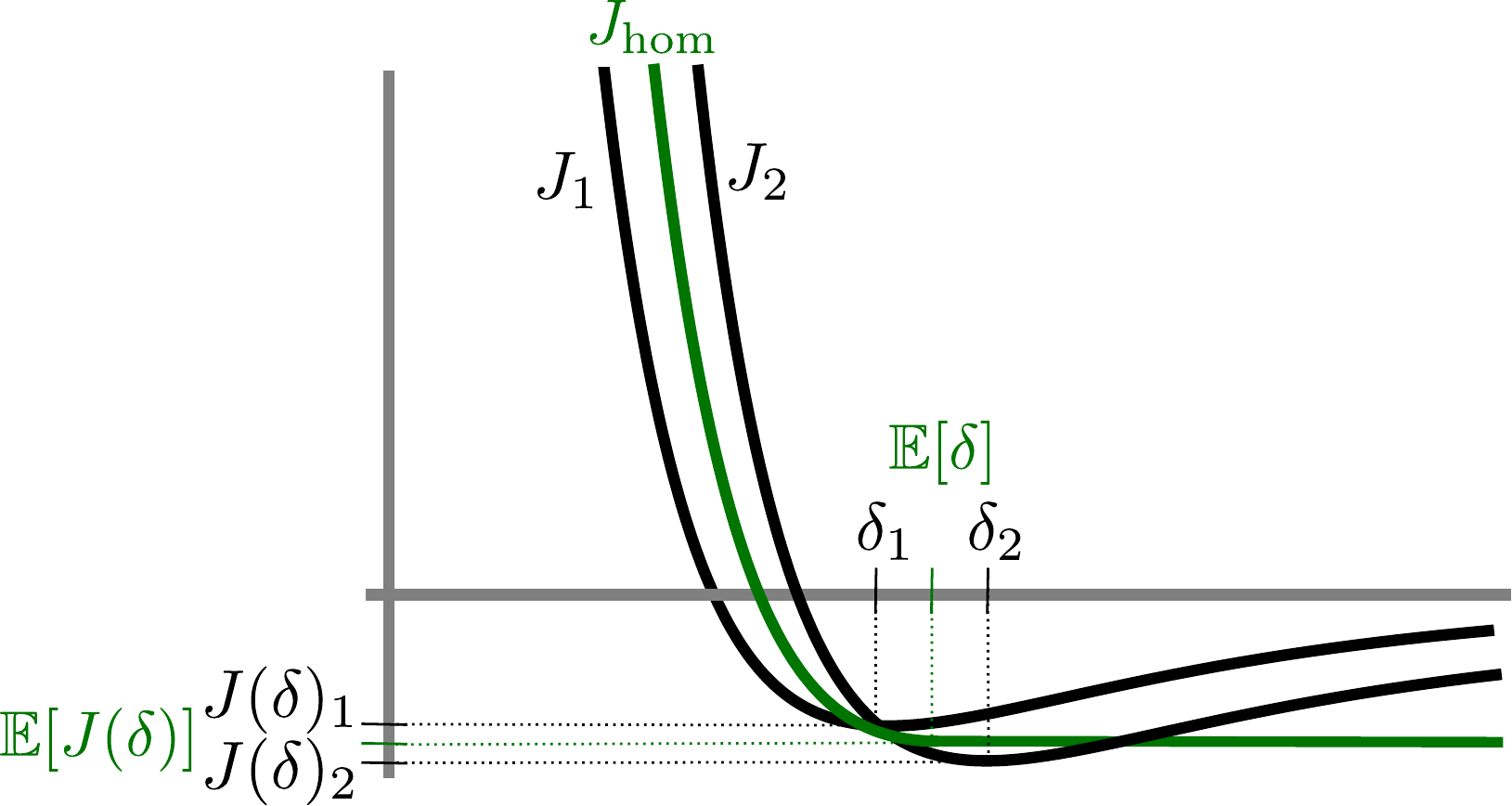}
	\caption{The function $J_{\mathrm{hom}}$ in the case $K=1$ and with two different potentials $J_1$ and $J_2$, which are equidistributed. Therefore the expectation values are equal to the arithmetic mean.}
\label{fig:jhom}
\end{minipage}\hspace{0.02\textwidth}
\begin{minipage}{0.48\textwidth}
\centering
	\includegraphics[width=0.78\linewidth]{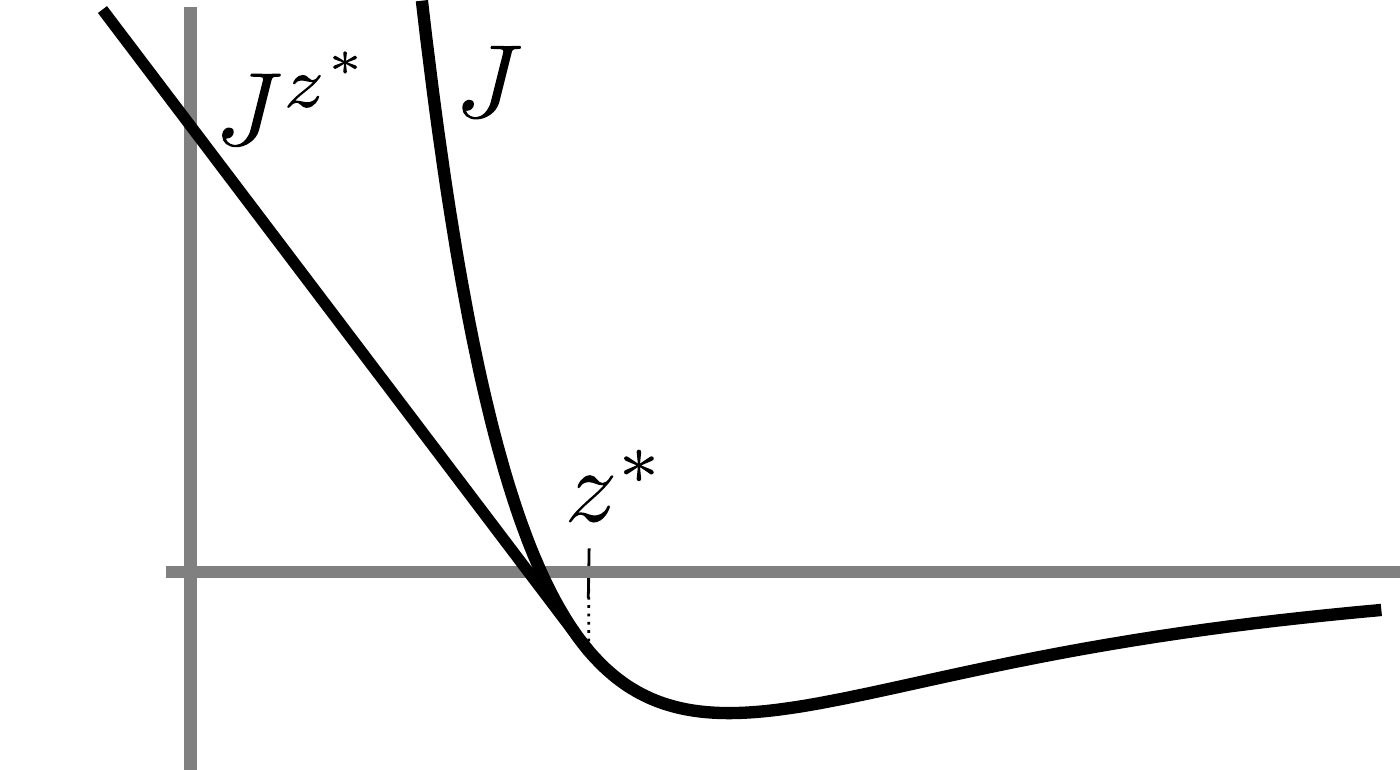}
		\caption{The function $J$ is a typical representative of a Lennard-Jones type potential and $J^{z^*}$ its $z^*$-approximation function.}
\label{Fig:cutoff}
\end{minipage}
	\end{figure}

\section{Asymptotic cell-formula}\label{sec:cellformula}

As mentioned above, the key ingredient to establish the limit \eqref{eq:defJhom} is the subadditive ergodic theorem. The difficulty here is that the Lennard-Jones type interaction potentials might blow up and in order to prove that \eqref{eq:defJhom} is valid for every $z\in\mathbb R$ and every $\omega\in\Omega'$ we need some additional arguments compared to previous results in stochastic homogenization, see e.g.~\cite{DalMasoModica1985}. We overcome the problem by a suitable approximation:

\begin{definition}
\label{Def:cutoff}
Consider a Lennard-Jones type potential $J_j(\omega,\cdot)$. By (LJ1), $J_j(\omega,z)$ has a blow up at $z\rightarrow0^+$. For having a linear growth at $z\to0^+$, we define a approximation as follows: Consider $0<z^*<\frac{1}{d}$, see (LJ2). Then, we denote the \textit{$z^*$-approximation} of $J_j(\omega,\cdot)$ by $J_j^{z^*}(\omega,\cdot)$, which is defined as
\begin{align*}
	J_j^{z^*}(\omega,z):=\begin{cases}
	m_j^{z^*}(\omega)(z-z^*)+J_j(\omega,z^*)&\quad\text{for}\ z<z^*,\\
	J_j(\omega,z)&\quad\text{for}\ z\geq z^*,
	\end{cases}
\end{align*}
with $m_j^{z^*}(\omega)\in\partial J_j(\omega,z^*)$, where $\partial J_j(\omega,z^*)$ is the subdifferential of $J_j(\omega,\cdot)$ at $z^*$. For uniqueness, we choose the smallest element of the subdifferential.
\end{definition}

Since $J_j(\omega,\cdot)$ is convex in $(0,\delta_j(\omega))$, this subdifferential is nonempty and if $J_j(\omega,\cdot)$ is differentiable in $z^*$, the derivative coincides with the subdifferential. By definition, the approximating function is continuous on its domain, H\"older-continuous on $(z^*,+\infty)$ and Lipschitz-continuous on $(-\infty,z^*)$. A Lennard-Jones type potential, together with one of its approximating functions is shown in Figure \ref{Fig:cutoff}.

For notational convenience, we set, for all $L\in\mathbb{N}$,
\begin{align*}
	J_j^L(\omega,z):=J_j^{z_L}(\omega,z),\quad\text{for all}\ z\in\mathbb{R},
\end{align*}
where $(z_L)_{L\in\mathbb{N}}\subset\mathbb{R}$ is a monotonously decreasing sequence with $z_L<\frac{1}{d}$  and $z_L\rightarrow 0$ for $L\rightarrow\infty$.

\begin{remark}
(i) It holds true that $J_j^L(\omega,z)\leq J_j(\omega,z)$ for every $z\in\mathbb{R}$ and every $L\in\mathbb{N}$. This follows directly from the definition of the subdifferential of a convex function.\\
(ii) For the approximation $J_j^L$, the estimates in \eqref{LJ2abschaetzung} in (LJ2) do not hold true any more. But, we have
\begin{align}
\label{LJschrankecutoff}
-d\leq J_j^L(\omega,z)\leq d\max\{\Psi(z),|z|\}\quad\text{for all}\ z\in\mathbb{R},\ j=1,...,K,\ \omega\in\Omega
\end{align}
by construction.
\end{remark}

In contrast to $J_j$ the interaction potentials $J_j^L$ are uniformly continuous for every $L\in\mathbb N$. Thus, we can combine the subadditiv ergodic theorem with standard techniques in stochastic homogenization (in particular) to obtain
\begin{proposition}\label{Prop:existencejhom}
Assume that Assumption \ref{Ass:stochasticLJ} is satisfied. There exists $\Omega_0\subset\Omega$ with $\mathbb{P}(\Omega_0)=1$ such that the following is true: For all $\omega\in\Omega_0$, $z\in\mathbb{R}$ and $A:=[a,b)$ with $a,b\in\mathbb R$ it holds
\begin{equation}\label{lim:JhomLN}
 J_{\rm hom}^L(z):=\lim_{N\to\infty}\mathbb E\left[J_{\rm hom}^{L,(N)}(\cdot,z,[0,1))\right]=\lim_{N\to\infty} J_{\rm hom}^{L,(N)}(\omega,z,A),
\end{equation}
where
\begin{equation}\label{def:JhomLlocal}
J_{\mathrm{hom}}^{L,(N)}(\omega,z,A):=\dfrac{1}{|NA\cap\mathbb{Z}|}\inf\left\{\sum_{j=1}^{K}\sum_{\substack{i\in \mathbb Z \cap NA\\ i+j-1\in NA }}J_j^L\left(\omega,i,z+\dfrac{\phi^{i+j}-\phi^{i}}{j}\right),\,\phi\in \mathcal A_{N,K}^0(A)\right\},
\end{equation}
and $\mathcal A_{N,K}^0(A)$ is defined in \eqref{def:ANK0}.
\end{proposition}
Further, some usefull properties of $J_{\rm hom}^L$ are established in
\begin{proposition}\label{Prop:jhomlcontinuous}
Assume that Assumption \ref{Ass:stochasticLJ} is satisfied. The map $z\mapsto J_{\rm hom}^L(z)$ is continuous and convex. Moreover, there exists $\Omega_0\subset\Omega$ with $\mathbb{P}(\Omega_0)=1$ such that the following is true: For all $\omega\in\Omega_0$, and every $A=[a,b)$, $a,b\in\mathbb R$ it holds
\begin{align*}
		\Gamma\text{-}\lim_{N\to\infty} J_{\mathrm{hom}}^{L,(N)}(\omega,\cdot,A)= J_{\mathrm{hom}}^L.
		\end{align*}
\end{proposition}
Finally, the following proposition justifies the approximation of the potentials and is the key to the proof of Proposition~\ref{Prop:Jhom}.
\begin{proposition}\label{Prop:Jhomzklein}
Assume that Assumption \ref{Ass:stochasticLJ} is satisfied. There exists $\Omega_0\subset\Omega$ with $\mathbb{P}(\Omega_0)=1$ such that the following is true: For all $\omega\in\Omega_0$, $z\in\mathbb{R}$ and $A:=[a,b)$ with $a,b\in\mathbb R$ the limit $\lim\limits_{N\rightarrow\infty}J_{\mathrm{hom}}^{(N)}(\omega,z,A)$ exists in $\overline{\mathbb{R}}$ and is independent of $\omega$ and $A$. Moreover,
$$\lim\limits_{N\rightarrow\infty}J_{\mathrm{hom}}^{(N)}(\omega,z,A)=\lim\limits_{L\rightarrow\infty}J_{\mathrm{hom}}^L(z).$$
\end{proposition}

\section{Proofs}

\subsection{Proof of Theorem~\ref{Thm:nullteordnung}} \label{sec:proof-thm}

We are now in a position to prove our main result, the $\Gamma$-convergence result in Theorem~\ref{Thm:nullteordnung}. We first show the compactness result, secondly the liminf-inequality and thirdly the limsup-inequality. While the compactness proof is straight forward, the liminf-inequality is rather technical due to the introduction of two coarser additional scales that allow to deal with the randomness of the system. The limsup-inequality is first shown for affine deformations, then for piecewise affine and finally for $W^{1,1}$-functions.

\begin{proof}[Proof of Theorem \ref{Thm:nullteordnung}]
The existence of $J_{\mathrm{hom}}$ and some properties of this function is shown in Proposition \ref{Prop:Jhom}.

\step 1 'Compactness'. 

Let $(u_n)\subset L_1$ be a sequence with $\sup_n E_n^\ell(\omega,u_n)<\infty$. By the superlinear growth at $-\infty$, the common lower bound from (LJ2) and the boundary conditions $u_n(0)=0$, $u_n(1)=\ell$ for every $n\in\mathbb{N}$, we obtain $\sup_n\lVert u_n\rVert_{W^{1,1}(0,1)}<C<\infty$. Since $\lVert u_n\rVert_{W^{1,1}(0,1)}$ is equibounded, we can extract a subsequence (not relabelled) $(u_n)$ which weakly$^*$ converges in $BV(0,1)$ to $u\in BV(0,1)$ \cite[Thm.~3.23]{AmbrosioFuscoPallara}. By definition, we also have $u\in BV^{\ell}(0,1)$.

Because we need it in the following as a technical result, we also consider a given partition $I_k=[c,d]$, $k=0,1,...m$ and $c,d\in[0,1]$, assuming $(\mathbb{Z}\cap n I_j)\cap(\mathbb{Z}\cap n I_{k})=\emptyset$ for $j\neq k$ and for all $n\in\mathbb{N}$. With the same argumentation as above, we get 
\begin{align}
\label{compactheitmitrandwerten}
 \lVert(u_n')\rVert_{L^{1}(I_k^-)}\leq 2C|I_k|+u_n(b)-u_n(a),
\end{align}
with $a:=\min\{x:x\in I_k\cap \frac{1}{n}\mathbb{Z}\}$ and $b:=\max\{x:x\in I_k\cap \frac{1}{n}\mathbb{Z}\}$ and $I_k^-:=\lambda_n\big[\min\{i:i\in nI_k\cap \mathbb{Z}\},\max\{i:i\in nI_k\cap \mathbb{Z}\}\big)$.\\

\step 2 'Liminf inequality'.

	Let $(u_n)\subset L^1(0,1)$ be a sequence with $u_n\rightarrow u$ in $L^1(0,1)$ and  with $\sup_nE_n^{\ell}(\omega,u_n)<\infty$. From the compactness result, we know that $u_n\rightharpoonup^* u$ in $BV(0,1)$ and $u$ fulfils the boundary conditions. We regard $u$ as a good representative (cf.~\cite[Thm.~3.28]{AmbrosioFuscoPallara}). The aim is to show
	\begin{align*}
	\liminf_{n\rightarrow\infty}E_n^{\ell}(\omega,u_n)\geq\int_{0}^{1}J_{\mathrm{hom}}(u'(x))\,\mathrm{d}x.
	\end{align*} 	
	We pass to $\hat{u}_n$ instead of $u_n$, which is the sequence $(\hat{u}_n)$ of piecewise constant functions defined by $\hat{u}_n(i/n)=u_n(i/n)$. By Theorem \ref{Thm:interpolationConstant}, $\hat{u}_n$  also weakly$^*$ converges to $u$ in $BV(0,1)$. Further, it has the same discrete difference quotients as $u_n$. Now, we pass to a subsequence $\hat{u}_{n}$ (not relabelled) with $\liminf\limits_{n\rightarrow\infty}E_n^{\ell}(\omega,\hat{u}_n)=\lim\limits_{n\rightarrow\infty}E_{n}^{\ell}(\omega,\hat{u}_{n})$ and then to a further subsequence $(\hat{u}_{n})$ (not relabelled) such that $\hat{u}_{n}\rightarrow u$ pointwise almost everywhere, which is possible because of the convergence $\hat{u}_{n}\rightarrow u$ in $L^1(0,1)$.\\
	
	\textit{Introduction of the first artificial scale.}
	
	\begin{figure}[t]
		\centering
				\includegraphics[width=0.5\linewidth]{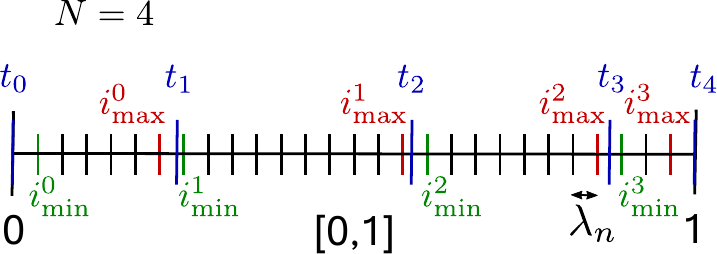}
				\caption{Illustration of the definitions $i_{min}^j$ and $i_{max}^j$ for $N=4$.}
		\label{fig:iminmax}
	\end{figure}
		
While in the periodic problem the periodicity length functions as a coarser scale, we here have to introduce an artificial coarser scale. We define the coarser grid as follows: For a fixed $\delta>0$, small enough, there always exists $M\in\mathbb{N}$ and $t_0,...,t_M\in[0,1]$ such that $t_0=0$, $t_M=1$, $\delta<t_{m+1}-t_m<2\delta$, $t_m$ is not in the jump set of $u$ and $\hat{u}_{n}(t_m)\rightarrow \hat{u}(t_m)$ pointwise for $n\rightarrow\infty$ and for every $m=0,...,M$. With the definition (illustrated in Figure \ref{fig:iminmax})
\begin{align*}
i_{min}^m&:=\min\left\{i:i\in n[t_m,t_{m+1}) \right\},\quad i_{max}^m:=\max\left\{i:i\in n[t_m,t_{m+1}) \right\},
\end{align*}
and the lower bound $J_j(\omega,\cdot)\geq -d$ for every $j\in\{1,\dots,K\}$, see (L2), we estimate 
\begin{align}\label{liminf1}
\begin{split}
E_n^{\ell}(\omega,u_n)=&\sum_{j=1}^{K}\sum_{i=0}^{n-j}\lambda_nJ_j\left(\omega,i,\dfrac{\hat{u}_n^{i+j}-\hat{u}_n^{i}}{j\lambda_n}\right)\\
\geq&\sum_{j=1}^{K}\sum_{m=0}^{M-1}\lambda_n\sum_{i=i_{min}^m}^{i_{max}^m+1-j}J_j\left(\omega,i,\dfrac{\hat{u}_n^{i+j}-\hat{u}_n^{i}}{j\lambda_n}\right)-\lambda_ndK^2M.
\end{split}
\end{align}
	
\textit{Introduction of the second artificial scale.}

We want to continue with the first term of the right hand side of \eqref{liminf1}. 
The two remainders, which will show up in \eqref{liminfjhom} are the reason for introducing the second artificial scale $\epsilon$. In the case of next-to-nearest neighbours, these remainders do not appear at all. Therefore, the second scale is not necessary in this case. It is just useful in the case $K\geq2$.\\
For this, we introduce an additional length scale $\epsilon>0$ that is much smaller than $\delta$. Because of the pointwise convergence almost everywhere of $\hat{u}_n$, we can find for every $m=0,...,M$ values $a_m\in\mathbb{R}$ and $b_m\in\mathbb{R}$ (explicitly, $a_m$ and $b_m$ depend also on $M$, $n$ and $\epsilon$, but we do not denote this for better readability) which are not in the jump set of $u$, with $t_m<a_m<b_m<t_{m+1}$, $\epsilon<a_m-t_m<2\epsilon$, $\epsilon<t_{m+1}-b_m<2\epsilon$ and $\hat{u}_n(a_m)\to u(a_m)$ and $\hat{u}_n(b_m)\to u(b_m)$ pointwise for $n\to\infty$. With that, we define $h_n^{a_m}\in\mathbb{N}$ and $h_n^{b_m}\in\mathbb{N}$, with $0\leq h_n^{a_m}\leq h_n^{b_m}\leq n$ such that $a_m\in\lambda_n[h_n^{a_m},h_n^{a_m}+1)$ and $b_m\in\lambda_n[h_n^{b_m},h_n^{b_m}+1)$. Note that for $n$ large enough it always holds true that $i_{min}^m+K<<h_n^{a_m}$ and $h_n^{b_m}<<i_{max}^m-K$. 
	
		\begin{figure}[t]
		\centering
				\includegraphics[width=0.5\linewidth]{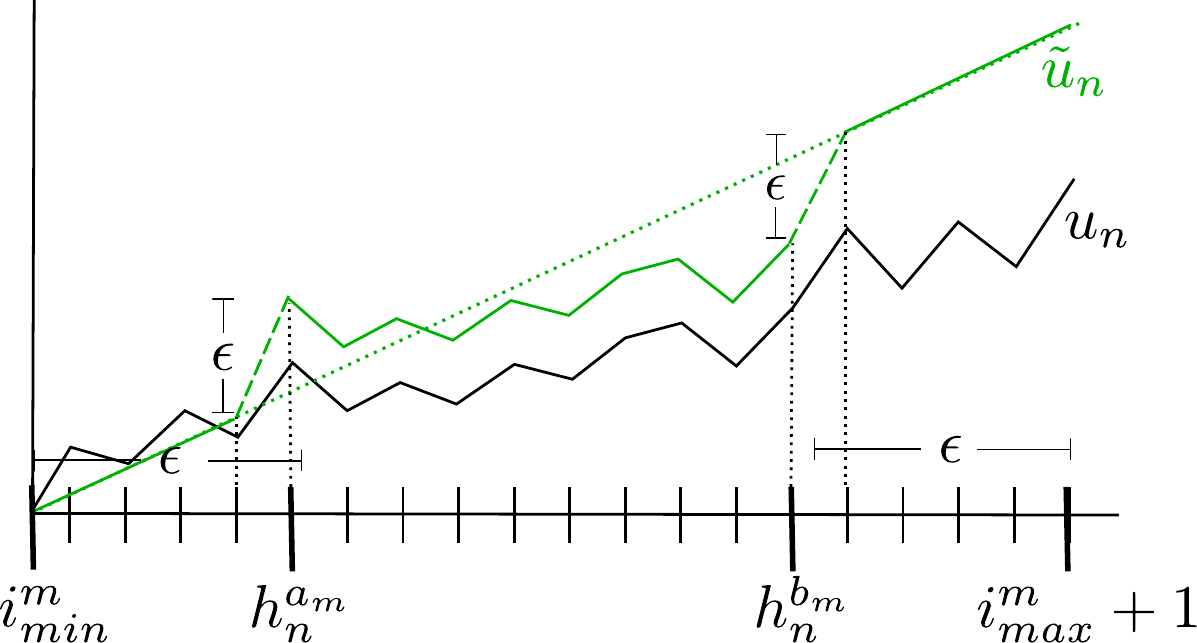}
				\caption{Illustration of the definition of $\tilde{u}_n$.}
		\label{fig:liminf}
		\end{figure}

	We further need a modified version $\tilde{u}_n$ of the function $\hat{u}_n$, because $\hat{u}_n$ does not fulfil the boundary constraint of the infimum problem of $J_{hom}^{L,(n)}$. Therefore we change it a little bit, such that $\tilde{u}_n$ becomes a competitor for the infimum problem. Remind, that it holds true that the discrete difference quotients of $u_n$ and $\hat{u}_n$ are the same, by construction, and can therefore be used equivalently. Now, define as an abbreviation
	\begin{align*}
	z_{n,m}^{\epsilon}:=\dfrac{u_n^{h_n^{b_m}}-u_n^{h_n^{a_m}}+2\epsilon}{\lambda_n(h_n^{b_m}-h_n^{a_m}+2)},
	\end{align*}
	which will be the average slope of $\tilde{u}_n$ on the interval $\lambda_n[i_{min}^m,i_{max}^m+1]$. Since $\hat{u}$ is piecewise constant and by the definition of $a_m$ and $b_m$, we get for $n\rightarrow\infty$
		\begin{align}
			\label{differenzen}
			z_{n,m}^{\epsilon}=\dfrac{u_n^{h_n^{b_m}}-u_n^{h_n^{a_m}}+2\epsilon}{\lambda_n(h_n^{b_m}-h_n^{a_m}+2)}=\dfrac{\hat{u}_n(b_m)-\hat{u}_n(a_m)+2\epsilon}{\lambda_n(h_n^{b_m}-h_n^{a_m}+2)}\rightarrow\dfrac{u(b_m)-u(a_m)+2\epsilon}{b_m-a_m}.
		\end{align}
	With this, we define $\tilde{u}_n$ as the continuous and piecewise affine function with $\tilde{u}_n^0=0$ and
	\begin{align*}
	\dfrac{\tilde{u}_n^{i+1}-\tilde{u}_n^{i}}{\lambda_n}&=z_{n,m}^\epsilon\ &\text{for}&\ i_{min}^m\leq i\leq h_n^{a_m}-2\ \text{and}\ h_n^{b_m}+1\leq i\leq i_{max}^m,\\
	\tilde{u}_n^i&=\tilde{u}_n^{i-1}+\epsilon\ &\text{for}&\ i= h_n^{a_m}\ \text{and}\ i=h_n^{b_m}+1,\\
	\dfrac{\tilde{u}_n^{i+1}-\tilde{u}_n^{i}}{\lambda_n}&=\dfrac{u_n^{i+1}-u_n^i}{\lambda_n} \ &\text{for}&\ h_n^{a_m}\leq i\leq h_n^{b_m}-1.
	\end{align*}
	A sketch of this construction can be found in Figure \ref{fig:liminf}. Note that the boundary constraints of the infimum problem of $J_{hom}^{L,(n)}$ are fulfilled, by definition. Further note that the slopes of $u_n$ and $\tilde{u}_n$ are the same on the interval $h_n^{a_m}\leq i\leq h_n^{b_m}-1$. The two jumps which we included in the definition of $\tilde{u}_n$ are of technical reasons. They are designed in such a way that the remainders, which show up in the following, can easily be estimated. This can be seen in \eqref{grundsprung}, where the presence of the jump ensures that the discrete gradients can be bounded from below by a positive value converging to $+\infty$.

	With all these definitions, and by definition of $J_{\mathrm{hom}}^{(n)}$, we can estimate the first term of the right hand side of \eqref{liminf1}:
	\begin{equation}
		\label{liminfjhom}
		\begin{split}
		&\sum_{j=1}^{K}\sum_{m=0}^{M-1}\lambda_n\sum_{i=i_{min}^m}^{i_{max}^m+1-j}J_j\left(\omega,i,\dfrac{\hat{u}_n^{i+j}-\hat{u}_n^{i}}{j\lambda_n}\right)\\
		\geq& \sum_{m=0}^{M-1}\lambda_n\left\lvert i_{max}^m-i_{min}^m+1 \right\rvert J_{\mathrm{hom}}^{(n)}\left(\omega,z_{n,m}^{\epsilon},[t_m,t_{m+1}) \right)\\
		&+\sum_{m=0}^{M-1}\sum_{j=1}^{K}\sum_{i=i_{min}^m}^{h_n^{a_m}-1}\lambda_n\left(J_j\left(\omega,i,\dfrac{\hat{u}_n^{i+j}-\hat{u}_n^{i}}{j\lambda_n}\right)-J_j\left(\omega,i,\dfrac{\tilde{u}_n^{i+j}-\tilde{u}_n^{i}}{j\lambda_n}\right)\right)\\ &+\sum_{m=0}^{M-1}\sum_{j=1}^{K}\sum_{i=h_n^{b_n}-j+1}^{i_{max}^m+1-j}\lambda_n\left(J_j\left(\omega,i,\dfrac{\hat{u}_n^{i+j}-\hat{u}_n^{i}}{j\lambda_n}\right)-J_j\left(\omega,i,\dfrac{\tilde{u}_n^{i+j}-\tilde{u}_n^{i}}{j\lambda_n}\right)\right).
		\end{split}
	\end{equation}
	
	\textit{Vanishing remainders.}
	
	Later on, we continue with the first term of \eqref{liminfjhom} in \eqref{liminf2}. Before, we do so, we first consider the second and third term of \eqref{liminfjhom} and show that they vanish in the limit $(\liminf_{\epsilon\to 0}\lim\limits_{n\to\infty})$. Because the calculation and the arguments are the same, we just show them for the term
	\begin{align}
	\label{restterm1}
	\begin{split}
	&\sum_{m=0}^{M-1}\sum_{j=1}^{K}\sum_{i=i_{min}^m}^{h_n^{a_m}-1}\lambda_n\left(J_j\left(\omega,i,\dfrac{\hat{u}_n^{i+j}-\hat{u}_n^{i}}{j\lambda_n}\right)-J_j\left(\omega,i,\dfrac{\tilde{u}_n^{i+j}-\tilde{u}_n^{i}}{j\lambda_n}\right)\right)
	\end{split}
	\end{align}
	and not for the other one. The first part of \eqref{restterm1} can be estimated by \eqref{LJschrankecutoff} as
	\begin{align*}
	&\sum_{m=0}^{M-1}\sum_{j=1}^{K}\sum_{i=i_{min}^m}^{h_n^{a_m}-1}\lambda_nJ_j\left(\omega,i,\dfrac{\hat{u}_n^{i+j}-\hat{u}_n^{i}}{j\lambda_n}\right)\geq\sum_{m=0}^{M-1}\sum_{j=1}^{K}\sum_{i=i_{min}^m}^{h_n^{a_m}-1}\lambda_n(-d)\\\geq&-MK\left(h_n^{a_m}-i_{min}^m \right)\lambda_nd\stackrel{n\to\infty}{\longrightarrow} -MKd(a_m-t_m)\geq -2\epsilon MKd,
	\end{align*}
	which converges to $0$ for $\epsilon\to0$. The second part of \eqref{restterm1} is
	\begin{align*}
	\sum_{m=0}^{M-1}\sum_{j=1}^{K}\sum_{i=i_{min}^m}^{h_n^{a_m}-1}-\lambda_nJ_j\left(\omega,i,\dfrac{\tilde{u}_n^{i+j}-\tilde{u}_n^{i}}{j\lambda_n}\right).
	\end{align*}
	By construction of $\tilde{u}_n$ and since we consider $i_{min}^m\leq i\leq h_n^{a_m}-1$, it holds true that 
	\begin{align*}
		\dfrac{\tilde{u}_n^{i+j}-\tilde{u}_n^{i}}{j\lambda_n}&=\left(\dfrac{\tilde{u}_n^{i+j}-\tilde{u}_n^{i+j-1}}{j\lambda_n}+...+\dfrac{\tilde{u}_n^{i+1}-\tilde{u}_n^{i}}{j\lambda_n} \right)=\dfrac{1}{j}\left(xz_{n,m}^{\epsilon}+\sum_{k=p}^{q}\dfrac{u_n^{k+1}-u_n^{k}}{\lambda_n}+y\dfrac{\epsilon}{\lambda_n}\right),
	\end{align*}
	where $x\in\{0,...,j\}$, $p\geq h_n^{a_m}$, $q\leq h_n^{a_m}-1+j$, $q-p+1\leq j$, $y\in\{0,1\}$ and from $y=0$ follows $q<p$. Further, we know from \eqref{differenzen} that $z_{n,m}^{\epsilon}$ converges and is therefore bounded by a constant $\hat{C}$. Due to $\sup_n E_n^\ell(\omega,u_n)<\infty$, we have $\dfrac{u_n^{k+1}-u_n^k}{\lambda_n}\geq 0$  for every $k=0,...,N-1$. Putting all these information together yields that it holds one of the following two cases, namely either Case 1
	\begin{align*}
	\dfrac{\tilde{u}_n^{i+j}-\tilde{u}_n^{i}}{j\lambda_n}&=\frac{1}{j}\left(jz_{n,m}^\epsilon\right)=z_{n,m}^\epsilon
	\end{align*}
	or Case 2
	\begin{align}
	\label{grundsprung}
		\dfrac{\tilde{u}_n^{i+j}-\tilde{u}_n^{i}}{j\lambda_n}&\geq\frac{1}{j}\left(-j\hat{C}+\dfrac{\epsilon}{\lambda_n}\right)=\dfrac{-\lambda_n\hat{C}+\epsilon/j}{\lambda_n}\geq\dfrac{C}{\lambda_n},
	\end{align}
	for $n$ large enough. In Case 1, we get
	\begin{align*}
		&\sum_{m=0}^{M-1}\sum_{j=1}^{K}\sum_{i=i_{min}^m}^{h_n^{a_m}-1}-\lambda_nJ_j\left(\omega,i,\dfrac{\tilde{u}_n^{i+j}-\tilde{u}_n^{i}}{j\lambda_n}\right)\geq\sum_{m=0}^{M-1}\sum_{j=1}^{K}\sum_{i=i_{min}^m}^{h_n^{a_m}-1}-\lambda_nd\max\left\{\Psi\left(z_{n,m}^{\epsilon}\right),\left|z_{n,m}^{\epsilon} \right| \right\}\\ \geq&\sum_{m=0}^{M-1}\sum_{j=1}^{K}\sum_{i=i_{min}^m}^{h_n^{a_m}-1}-\lambda_nd\max\left\{\max_{|z|\leq\hat{C}}\Psi\left(z\right),\max_{|z|\leq\hat{C}}\left|z \right| \right\}\\
		=&\sum_{m=0}^{M-1}\sum_{j=1}^{K}\sum_{i=i_{min}^m}^{h_n^{a_m}-1}-\lambda_nC\geq-MKC\lambda_n\left(h_n^{a_m}-i_{min}^m \right)\stackrel{n\to\infty}{\longrightarrow}-MKC(a_m-t_m)\geq -2\epsilon MKC,
	\end{align*}
	with $0<C<\infty$. In the limit $\epsilon\to0$, this converges to zero, as desired. In this calculation, we assumed that $z_{n,m}^{\epsilon}$ lies within the domain of $\Psi$. This is indeed true because $z_{n,m}^{\epsilon}$ is a linear combination of the discrete gradients of $u_n$, which lie in the domain of $J_j$ because of $\sup_nE_n^{\ell}(\omega,u_n)<\infty$.
	
In Case 2, we have $\dfrac{\tilde{u}_n^{i+j}-\tilde{u}_n^{i}}{j\lambda_n}\geq d$ for $n$ large enough, with $d$ from (LJ2). This yields for $n$ large enough with (LJ2)
\begin{align*}
&\sum_{m=0}^{M-1}\sum_{j=1}^{K}\sum_{i=i_{min}^m}^{h_n^{a_m}-1}-\lambda_nJ_j\left(\omega,i,\dfrac{\tilde{u}_n^{i+j}-\tilde{u}_n^{i}}{j\lambda_n}\right)\geq\sum_{m=0}^{M-1}\sum_{j=1}^{K}\sum_{i=i_{min}^m}^{h_n^{a_m}-1}-\lambda_nb\\\geq& -MKb\lambda_n(h_n^{a_m}-i_{min}^m)\stackrel{n\to\infty}{\longrightarrow}-MKb(a_m-t_m)\geq -2MKb\epsilon.
\end{align*}
In the limit $\epsilon\to0$, this also converges to zero, as desired.\\

	\textit{Conclusion and removal of the two artificial scales.}
	
	By passing to the limit $\liminf_{n\rightarrow\infty}$ in \eqref{liminf1} and with Proposition \ref{Prop:Jhom}, as well as estimates \eqref{differenzen}, \eqref{liminfjhom} and \eqref{restterm1}, we obtain
	\begin{align}
	\label{liminf2}
	\begin{split}
	\liminf_{n\rightarrow\infty} E_n^{\ell}(\omega,u_n)&\geq \liminf_{n\to\infty}\sum_{m=0}^{M-1}\lambda_n\left\lvert i_{max}^m-i_{min}^m+1 \right\rvert J_{\mathrm{hom}}^{(n)}\left(\omega,z_{n,m}^{\epsilon},[t_m,t_{m+1}) \right)\\
	&\geq\sum_{m=0}^{M-1}\left\lvert t_{m+1}-t_{m}\right\rvert J_{\mathrm{hom}}\left(\dfrac{u\left(b_{m}\right)-u\left(a_m\right)+2\epsilon}{b_{m}-a_{m}}\right).
	\end{split}
	\end{align}
	For $\liminf_{\epsilon\to0}$, we then get
	\begin{align*}
		\liminf_{\epsilon\to0}\dfrac{u\left(b_{m}\right)-u\left(a_m\right)+2\epsilon}{b_{m}-a_{m}}=\dfrac{u(t_{m+1})-u(t_m)}{t_{m+1}-t_m},
	\end{align*}
	because there is no jump in $a_m$, $b_m$ and $t_m$ and therefore $u$ is absolutely continuous. Therefore we can continue with \eqref{liminf2} by	
	\begin{align}
	\label{liminf2.2}
	\begin{split}
	\liminf_{n\rightarrow\infty} E_n^{\ell}(\omega,u_n)&\geq\liminf_{\epsilon\to0}\sum_{m=0}^{M-1}\left\lvert t_{m+1}-t_{m}\right\rvert J_{\mathrm{hom}}\left(\dfrac{u\left(b_{m}\right)-u\left(a_m\right)+2\epsilon}{b_{m}-a_{m}}\right)\\
	&\geq \sum_{m=0}^{M-1}\left\lvert t_{m+1}-t_{m}\right\rvert J_{\mathrm{hom}}\left(\dfrac{u\left(t_{m+1}\right)-u\left(t_m\right)}{t_{m+1}-t_{m}}\right),
	\end{split}
	\end{align}
	because $J_{\mathrm{hom}}$ is lower semicontinuous due to Proposition~\ref{Prop:Jhom}. We now want to define $(w_M)$ as the piecewise affine interpolation of $u$ with grid points $t_m$ as in Theorem~\ref{thm:interpolationaffine}. We continue with estimating \eqref{liminf2.2} as follows:
\begin{align}\label{liminf3}
\liminf_{n\rightarrow\infty}E_n^{\ell}(\omega,u_n)&\geq\sum_{m=0}^{M}\left\lvert t_{m+1}-t_{m}\right\rvert J_{\mathrm{hom}}\left(\dfrac{w_M\left(t_{m+1}\right)-w_M\left(t_m \right)}{t_{m+1}-t_{m}}\right)=\int_{0}^{1}J_{\mathrm{hom}}\left(w_M'(x) \right)\,\mathrm{d}x.
\end{align}
Note that $J_{\mathrm{hom}}$ fulfils all assumptions of Theorem~\ref{PropGelli} (see Propositions \ref{Prop:Jhom}) and  $w_M\rightharpoonup^*u$ in $BV(0,1)$, according to Theorem \ref{thm:interpolationaffine}. Therefore, we finally get by taking the limit $\liminf_{M\rightarrow\infty}$ (which corresponds to $\delta\rightarrow0$) on both sides in \eqref{liminf3}
\begin{align}
\label{liminf}
\liminf_{n\rightarrow\infty}E_n^{\ell}(\omega,u_n)&\geq\liminf_{M\rightarrow\infty}\int_{0}^{1}J_{\mathrm{hom}}\left(w_M'(x) \right)\,\mathrm{d}x\geq\int_{0}^{1}J_{\mathrm{hom}}(u'(x))\,\mathrm{d}x.
\end{align}
Recalling $\infty>\liminf_{n\rightarrow\infty}E_n^{\ell}(\omega,u_n)$, Theorem~\ref{PropGelli} we obtain $D^su\geq0$ on $(0,1) $. An extension to $[0,1]$ can be done in the same way as in \cite[Thm.~4.2]{BraidesDalMasoGarroni1999}.\\

\step 3 'Limsup inequality'.

We need to show that for every $u\in BV^{\ell}(0,1)$ with $D^su\geq0$ there exists a sequence $(u_n)$ such that 
\begin{align}
\label{limsup}
\limsup_{n\rightarrow\infty}E_n^{\ell}(\omega,u_n)\leq E_{\mathrm{hom}}^{\ell}(u),
\end{align}
with $E_{\mathrm{hom}}^{\ell}(u)$. By Theorem~\ref{PropGelli}, it is sufficient to show \eqref{limsup} for $u\in W^{1,1}(0,1)$, instead of $u\in BV(0,1)$. This can be seen as follows: We know from Theorem~\ref{PropGelli} that the lower semicontinuous envelope of 
\begin{align*}
\mathcal{E}(u):=\begin{cases}
\int_{0}^{1}J_{\mathrm{hom}}(u'(x))\,\mathrm{d}x&\quad\text{for}\ u\in W^{1,1}(0,1),\\
+\infty&\quad\text{else.}
\end{cases}
\end{align*}
is $E_{\mathrm{hom}}^{\ell}(u)$, that is $\mathrm{sc}\,\mathcal{E}\equiv E_{\mathrm{hom}}^{\ell}$ with respect to the weak$^*$ convergence in $BV(0,1)$. Further, we know that the lower semicontinuous envelope with respect to the strong convergence in $L^1(0,1)$ can be even smaller, i.e.~$\mathrm{sc}_{L^1(0,1)}\mathcal{E}\leq\mathrm{sc}_{BV(0,1)}\mathcal{E}\equiv E_{\mathrm{hom}}^{\ell}$. That means that if we have shown \eqref{limsup} for $\mathcal{E}$, which means that we have
\begin{align*}
\Gamma\text{-}\limsup_{n\rightarrow\infty}E_n^{\ell}(\omega,u)\leq \mathcal{E}(u),
\end{align*}
then, with the definition of the lower semicontinuous envelope as $\mathrm{sc}\,f(x):=\sup\{\,g(x):g\ \text{l.s.c,}\ g\leq f \} $, we get
\begin{align*}
\Gamma\text{-}\limsup_{n\rightarrow\infty}E_n^{\ell}(\omega,u)\leq\mathrm{sc}_{L^1(0,1)}\mathcal{E}(u)\leq\mathrm{sc}_{BV(0,1)}\mathcal{E}(u)= E_{\mathrm{hom}}^{\ell}(u),
\end{align*}
because $\Gamma$-$\limsup$ is always lower semicontinuous. Therefore, we need to show \eqref{limsup} only for $u\in W^{1,1}(0,1)$. We prove this without taking into account the boundary values. For indicating this, we leave out the superscript $\ell$. The Dirichlet boundary conditions can then be added in the same way as in \cite[Thm.~4.2]{BraidesDalMasoGarroni1999}.\\

\textit{1) Affine functions.}

We start with constructing a recovery sequence for affine functions with $u'(x)=z$, $z\in(0,+\infty)$. For $z\notin(0,+\infty)$, the limsup-inequality is trivial because then we have $E_{\mathrm{hom}}(u)=\infty$, for $u(x)=zx$. With proposition \ref{Prop:Jhomzklein}, we get the existence of $\Omega_0\subset\Omega$ with $\mathbb{P}(\Omega')=1$ such that for all $z\in\mathbb{R}$ and all $A=[a,b)$, $a,b\in\mathbb{R}$ it holds true that
\begin{align}\label{minproblem}
\lim\limits_{n\rightarrow\infty}\dfrac{1}{|nA\cap\mathbb{Z}|}\inf\left\{\sum_{j=1}^{K}\sum_{\substack{i\in \mathbb Z \cap nA\\ i+j-1\in nA }}J_j\left(\omega,i,z+\dfrac{\phi^{i+j}-\phi^{i}}{j}\right),\,\phi\in \mathcal A_{N,K}^0(A)\right\}=J_{\mathrm{hom}}(z).
\end{align}
In the following, we will use the definitions 
\begin{align}\label{def:iminmax}
i_{min}^A:=\min\{i,i\in nA\cap\mathbb{Z}\}\quad \text{and}\quad i_{max}^A:=\max\{i,i\in nA\cap\mathbb{Z}\}.
\end{align}
Let us now consider an affine function $u(x):=zx$ for $z\in\mathbb{R}$. Let $\eta>0$ be a coarser scale. For simplicity, we assume $1/\eta\in\mathbb{N}$, such that the interval $[0,1]$ can be split equidistantly. The partition of the interval is labelled by $I_{k}^{\eta}:=[k\eta,(k+1)\eta)$ with $k=0,...,\frac{1}{\eta}-1$. An illustration of the two length scales $\lambda_n$ and $\eta$ is shown in Figure \ref{fig:zweiskalen}.
	\begin{figure}[t]
	\centering
			\includegraphics[width=0.5\linewidth]{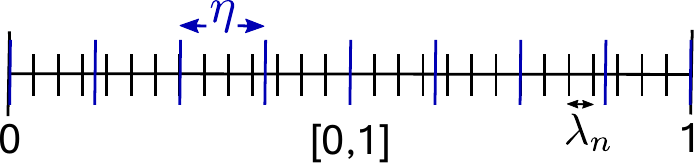}
			\caption{The two length scales $\lambda_n$ and $\eta$ involved in the proof of the limsup-inequality.}
	\label{fig:zweiskalen}
	\end{figure}

 Now, let $\eta$ be fixed. Then, for every $n\in\mathbb{N}$ there exists a minimizer $\phi_{n,I_k^{\eta}}:\{i:i\in \mathbb{Z}\cap nI_k^{\eta}\}\to(-\infty,+\infty]$ of the minimum problem in \eqref{minproblem} with $A=I_k^{\eta}$ for every $k=0,...,\frac{1}{\eta}-1$, which is interpolated to a piecewise affine function. Further, we define $\varphi_{n,I_k^{\eta}}(x):=\lambda_n\phi_{n,I_k^{\eta}}\left(\frac{x}{\lambda_n}\right)$ and
\begin{align}
\label{recovery}
u_{n,\eta}(x):=zx+\sum_{k=0}^{\frac{1}{\eta}-1}\varphi_{n,I_k^{\eta}}(x)\chi_{I_k^{\eta}}(x),
\end{align}
where $\chi_{I}$ is the characteristic function of the interval $I$. This is not yet the recovery sequence, but close by. By definition, it holds $u_{n,\eta}(0)=0$ and $u_{n,\eta}(1)=z:=\ell$ for every $n\in\mathbb{N}$. First, we show
\begin{align}
\label{attouchenergie0}
\limsup_{n\rightarrow\infty}E_n(\omega,u_{n,\eta})\leq J_{\mathrm{hom}}(z).
\end{align}
With the definition $E_n(\omega,u,I):=\sum_{j=1}^{K}\sum_{i=i_{min}^I}^{i_{max}^I+1-j}\lambda_nJ_j(\omega,i,\frac{u^{i+j}-u^{i}}{j\lambda_n})$ and $\phi_{n,I_1^{\eta}}^i:=\phi_{n,I_1^{\eta}}(i)$ for shorthand and by \eqref{minproblem}, it holds true that
\begin{align*}
E_n(\omega,u_{n,\eta},I_k^{\eta})&=\lambda_n\sum_{j=1}^{K}\sum_{i=i_{min}^{I_k^\eta}}^{i_{max}^{I_k^\eta}+1-j}J_j\left(\omega,i,z+\dfrac{\varphi_{n,I_k^{\eta}}^{i+j}-\varphi_{n,I_k^{\eta}}^{i}}{j\lambda_n}\right)\\
&=\lambda_n\sum_{j=1}^{K}\sum_{i=i_{min}^{I_k^\eta}}^{i_{max}^{I_k^\eta}+1-j}J_j\left(\omega,i,z+\dfrac{\phi_{n,I_k^{\eta}}^{i+j}-\phi_{n,I_k^{\eta}}^{i}}{j}\right)\rightarrow |I_k^{\eta}|J_{\mathrm{hom}}(z)\quad\text{for}\ n\rightarrow\infty.
\end{align*}
Since by construction
\begin{align*}
E_n(\omega,u_{n,\eta})=\sum_{k=0}^{\frac{1}{\eta}-1}E_n(\omega,u_{n,\eta},I_k^{\eta})+\sum_{k=0}^{\frac{1}{\eta}-2}\sum_{j=2}^{K}\sum_{s=0}^{j-2}\lambda_nJ_j\left(\omega,i_{max}^{I_k^\eta}-s, \dfrac{u_{n,\eta}^{i_{max}^{I_k^\eta}-s+j}-u_{n,\eta}^{i_{max}^{I_k^\eta}-s}}{j\lambda_n}\right),
\end{align*}
we have for the first part
\begin{align*}
\begin{split}
\lim\limits_{n\rightarrow\infty}\sum_{k=0}^{\frac{1}{\eta}-1}E_n(\omega,u_{n,\eta},I_k^{\eta})=\sum_{k=0}^{\frac{1}{\eta}-1} |I_k^{\eta}|J_{\mathrm{hom}}(z)=\sum_{k=0}^{\frac{1}{\eta}-1} \eta J_{\mathrm{hom}}(z)=J_{\mathrm{hom}}(z).
\end{split}
\end{align*}
The second part yields, noting that it holds true that $-s+j\leq K$ and $s\leq K-1$,
\begin{align*}
\begin{split}
&\sum_{k=0}^{\frac{1}{\eta}-2}\sum_{j=2}^{K}\sum_{s=0}^{j-2}\lambda_nJ_j\left(\omega,i_{max}^{I_k^\eta}-s, \dfrac{u_{n,\eta}^{i_{max}^{I_k^\eta}-s+j}-u_{n,\eta}^{i_{max}^{I_k^\eta}-s}}{j\lambda_n}\right)\\
=&\sum_{k=0}^{\frac{1}{\eta}-2}\sum_{j=2}^{K}\sum_{s=0}^{j-2}\lambda_nJ_j\left(\omega,i_{max}^{I_k^\eta}-s, z\right)\stackrel{(LJ2)}{\leq}\sum_{k=0}^{\frac{1}{\eta}-2}\sum_{j=2}^{K}\sum_{s=0}^{j-2}\lambda_nd\max\{\Psi(z),|z|\}\\
\leq& \lambda_nd\max\{\Psi(z),|z|\}\left(\frac{1}{\eta}-1\right)\frac12(K+1)K\to0\quad\text{for}\ n\to\infty.
\end{split}
\end{align*}

Together, this shows \eqref{attouchenergie0}. For later references, observe that this result is independent of $\eta$. Next, we show
\begin{align}
\label{attoucheta0}
\lim\limits_{\eta\to 0}\lim\limits_{n\to\infty}\lVert u_{n,\eta}-u\lVert_{L^1(0,1)}=0.
\end{align}
Since we know, that the energy of the recovery sequence has to be equi-bounded, we get from the compactness result \eqref{compactheitmitrandwerten} for all $k\in\{0,...,\frac{1}{\eta}-1\}$
\begin{align}
\label{ablschranke}
\lVert u_{n,\eta}'\rVert_{L^1(I_k^{\eta})}\leq (C+|z|)|I_k^{\eta}|,
\end{align}
because we have $u_{n,\eta}(b)-u_{n,\eta}(a)=z|I_k^{\eta}|+\varphi_{n,I_k^{\eta}}(b)-\varphi_{n,I_k^{\eta}}(a)=z|I_k^{\eta}|+0$, where $a:=\inf\{x:x\in I_k^{\eta}\}$ and $b:=\sup\{x:x\in I_k^{\eta}\}$. It follows $\lVert \varphi_{n,I_k^{\eta}}'\rVert_{L^1(I_k^{\eta})}\leq \tilde{C}\eta$, as 
\begin{align*}
\int_{I_k^{\eta}}|\varphi_{n,I_k^{\eta}}'(x)|\,\mathrm{d}x&=\int_{I_k^{\eta}}|\varphi_{n,I_k^{\eta}}'(x)+z-z|\,\mathrm{d}x\\
&\leq\int_{I_k^{\eta}}|u_{n,\eta}'(x)|\,\mathrm{d}x+\int_{I_k^{\eta}}|z|\,\mathrm{d}x\leq C|I_k^{\eta}|+2|z||I_k^{\eta}|=\tilde{C}|I_k^{\eta}|=\tilde{C}\eta.
\end{align*}
Recall that it holds true $|I_k^{\eta}|=\eta$ by definition. With this result, we get 
\begin{align*}
\int_{I_k^{\eta}}|\varphi_{n,I_k^{\eta}}(x)|\,\mathrm{d}x&=\int_{I_k^{\eta}}\left|\int_{j\eta}^{x}\varphi_{n,I_k^{\eta}}'(s)\,\mathrm{d}s\right|\,\mathrm{d}x\leq\int_{I_k^{\eta}}\int_{j\eta}^{x}\left|\varphi_{n,I_k^{\eta}}'(s)\right|\,\mathrm{d}s\,\mathrm{d}x\\
&\leq\int_{I_k^{\eta}}\int_{I_k^{\eta}}\left|\varphi_{n,I_k^{\eta}}'(s)\right|\,\mathrm{d}s\,\mathrm{d}x=|I_k^{\eta}|\int_{I_k^{\eta}}\left|\varphi_{n,I_k^{\eta}}'(s)\right|\,\mathrm{d}s\leq \tilde{C}\eta^2.
\end{align*}
This leads us to
\begin{align*}
\begin{split}
\lVert u_{n,\eta}-u\lVert_{L^1(0,1)}&=\int_{0}^{1}\left|\sum_{k=0}^{\frac{1}{\eta}-1}\varphi_{n,I_k^{\eta}}(x)\chi_{I_k^{\eta}}(x)\right|\,\mathrm{d}x\\
&\leq \sum_{k=0}^{\frac{1}{\eta}-1}\int_{I_k^{\eta}}\left|\varphi_{n,I_k^{\eta}}(x)\right|\,\mathrm{d}x\leq \sum_{k=0}^{\frac{1}{\eta}-1}\tilde{C}\eta^2=\frac{1}{\eta}\tilde{C}\eta^2=\tilde{C}\eta,
\end{split}
\end{align*}
which proves \eqref{attoucheta0} to be true. Since our aim is to construct a recovery sequence, which is only dependent on $n$, we have to pass to an appropriate subsequence. This is done with the help of the Attouch Lemma. Combined, \eqref{attouchenergie0} and \eqref{attoucheta0} yield that
\begin{align*}
	\limsup\limits_{\eta\rightarrow0}\limsup\limits_{n\rightarrow\infty}\left(|E_n(\omega,u_{n,\eta})-J_{\mathrm{hom}}(z)|+\lVert u_{n,\eta}-u\lVert_{L^1(0,1)} \right)=0.
\end{align*}
Using this result with the Attouch Lemma (Theorem \ref{attouch}), we therefore get the existence of a subsequence $\eta_n$ with $\eta_n\to0$ for $n\to\infty$ and
\begin{align*}
0&\leq\limsup_{n\to\infty}\left(|E_n(\omega,u_{n,\eta_n})-J_{\mathrm{hom}}(z)|+\lVert u_{n,\eta_n}-u\lVert_{L^1(0,1)} \right)\\&\leq \limsup_{\eta\to 0}\limsup_{n\to\infty}\left(|E_n(\omega,u_{n,\eta})-J_{\mathrm{hom}}(z)|+\lVert u_{n,\eta}-u\lVert_{L^1(0,1)} \right)=0
\end{align*}
Finally, this shows that for $u_{n,\eta_n}$ it holds true that $E_n(\omega,u_{n,\eta_n})\to J_{\mathrm{hom}}(z)$ and $u_{n,\eta_n}\to u$ in $L^1(0,1)$ for $n\rightarrow\infty$. Therefore $(u_{n,\eta_n})$ is the recovery sequence for the affine function $u(x)=zx$ with $z\in\mathbb{R}$. Moreover, we also have $u_{n,\eta_n}\to u$ weakly$^*$ in $BV(0,1)$, since from \eqref{ablschranke} we have $\limsup_{n\to\infty} \|u_{n,\eta_n}'\|_{L^1(0,1)}<\infty$.\\

Note that the same construction can be applied on any interval $(a,b)$ instead of $[0,1]$.\\

\textit{2) Piecewise affine functions.}

With this construction of a recovery sequence for affine functions, we can construct a recovery sequence for piecewise affine functions by dividing the interval $[0,1]$ into parts where the function is affine and repeating the above construction. The only difficulty lies in gluing the different parts together. We show this by considering a function $u$ with
\begin{align*}
u(x):=\begin{cases}
z_1x\quad&\text{for}\ x\in[0,a),\\
z_1a+z_2(x-a)\quad&\text{for}\ x\in[a,1],
\end{cases}
\end{align*}
for $0<a<1$. This function is piecewise affine with $u'(x)=z_1$ on $(0,a)$ and $u'(x)=z_2$ on $(a,1)$. Let $(u_n^1)$ be the recovery sequence for $u(x)=z_1x$ on $(0,a)$ and $(u_n^2)$ the recovery sequence for $u(x)=z_2x$ on $(a,1)$ constructed in Step 1. Without relabelling it, we extend $u_n^1$ continuously with constant slope $z_1$ on $(i_{max}^{[0,a)},a)$, because it is not defined there yet. The same we do for $u_N^2$ on $(a,i_{max}^{[a,1)})$ with slope $z_2$. Then, we claim that 
\begin{align}
\label{recoverysequencepwa}
u_n(x):=u_n^1(x)\chi_{[0,a)}+\left(z_1a+u_n^2(x-a) \right)\chi_{[a,1]}
\end{align} 
is a recovery sequence for $u$. Indeed, it holds true that
\begin{align*}
u_n(x)&=u_n^1(x)\chi_{[0,a)}+\left(z_1a+u_n^2(x-a) \right)\chi_{[a,1]}\\
&\rightarrow z_1 x\chi_{[0,a)}+\left(z_1a+z_2(x-a) \right)\chi_{[a,1]}=u(x)
\end{align*}
in $L^1(0,1)$ for $n\rightarrow\infty$, since both sequences are recovery sequences. Further, it is
\begin{align*}
E_n(\omega,u_{n})= &E_n(\omega,u_{n}^1,[0,a))+E_n(\omega,u_{n}^2,[a,1))\\&+\sum_{j=2}^{K}\sum_{s=0}^{j-2}\lambda_nJ_j\left(\omega,i_{max}^{[0,a)}-s, \dfrac{u_n^{i_{max}^{[0,a)}-s+j}-u_n^{i_{max}^{[0,a)}-s}}{j\lambda_n}\right).
\end{align*}
By construction, we have that
\begin{align*}
\lim\limits_{n\to\infty}\left(E_n(\omega,u_{n}^1,[0,a))+E_n(\omega,u_{n}^2,[a,1))\right)&=\int_{0}^{a}J_{\mathrm{hom}}(z_1)\,\mathrm{d}x+\int_{a}^{1}J_{\mathrm{hom}}(z_2)\,\mathrm{d}x\\&=\int_{0}^{1}J_{\mathrm{hom}}(u'(x))\,\mathrm{d}x.
\end{align*}
For the given values of $s$ and $j$, we get
\begin{align*}
\dfrac{u_n^{i_{max}^{[0,a)}-s+j}-u_n^{i_{max}^{[0,a)}-s}}{j\lambda_n}&=\dfrac{z_1a+z_2\left(\left(i_{max}^{[0,a)}-s+j \right)\lambda_n-a \right)-z_1\left(i_{max}^{[0,a)}-s)\right)\lambda_n}{j\lambda_n}\\ &=(z_1-z_2)\dfrac{a-\lambda_n\left(i_{max}^{[0,a)}-s\right)}{j\lambda_n}+z_2=:z_n,
\end{align*}
and since $\dfrac{a-\lambda_n\left(i_{max}^{[0,a)}-s\right)}{j\lambda_n}\rightarrow\dfrac{s}{j}\leq1$ for $n\to\infty$, it holds true that $z_n$ is a convex combination of $z_1$ and $z_2$, and therefore
\begin{align*}
&\sum_{j=2}^{K}\sum_{s=0}^{j-2}\lambda_nJ_j\left(\omega,i_{max}^{[0,a)}-s, \dfrac{u^{i_{max}^{[0,a)}-s+j}-u^{i_{max}^{[0,a)}-s}}{j\lambda_n}\right)=\sum_{j=2}^{K}\sum_{s=0}^{j-2}\lambda_nJ_j\left(\omega,i_{max}^{[0,a)}-s, z_n\right)\\\stackrel{(LJ2)}{\leq}&\sum_{j=2}^{K}\sum_{s=0}^{j-2}\lambda_nd\max\{\Psi\left(z_n\right),|z_n|\}\leq\lambda_ndC\frac{1}{2}(K+1)K\to0\quad\text{for}\ n\to\infty.
\end{align*}
Altogether, this shows the limsup inequality
\begin{align*}
\limsup_{n\to\infty}E_n(\omega,u_{n})\leq\int_{0}^{1}J_{\mathrm{hom}}(u'(x))\,\mathrm{d}x.
\end{align*}

\textit{3) $W^{1,1}$-functions.}

Now, we provide arguments to pass to functions $u\in W^{1,1}$:
For $u\in W^{1,1}(0,1)$, consider the piecewise affine interpolation $u_N$ of $u$ with grid points $t_N^j$, which means $u_N\in C(0,1)$ is affine on $[t^{j-1}_N,t^j_N)$ an it holds $u_N(t_j^N)=u(t_j^N)$ for all $j=0,...,N$. This is well defined because we can consider $u$ as its absolute continuous representative. Then, it holds
\begin{align}
\label{w11gl1}
\begin{split}
&E_{\mathrm{hom}}(u)=\int_{0}^{1}J_{\mathrm{hom}}(u')\,\mathrm{d}x=\sum_{i=1}^{N}\left(t_N^{i-1}-t_N^{i}\right)\dfrac{1}{t_N^{i-1}-t_N^{i}}\int_{t_N^{j-1}}^{t_N^{j}}J_{\mathrm{hom}}(u'(x))\,\mathrm{d}x\\
\stackrel{\text{Jensen}}{\geq}&\sum_{i=1}^{N}\left(t_N^{i-1}-t_N^{i}\right)J_{\mathrm{hom}}\left(\dfrac{1}{t_N^{i-1}-t_N^{i}}\int_{t_N^{j-1}}^{t_N^{j}}u'(x)\,\mathrm{d}x \right)\\
=&\sum_{i=1}^{N}\left(t_N^{i-1}-t_N^{i}\right)J_{\mathrm{hom}}\left(\dfrac{1}{t_N^{i-1}-t_N^{i}}\left(u(t_N^{i})-u(t_N^{i-1}) \right) \right)\\
=&\sum_{i=1}^{N}\left(t_N^{i-1}-t_N^{i}\right)J_{\mathrm{hom}}\left(\dfrac{1}{t_N^{i-1}-t_N^{i}}\left(u_N(t_N^{i})-u_N(t_N^{i-1}) \right) \right)=\int_{0}^{1}J_{\mathrm{hom}}(u_N')=E_{\mathrm{hom}}(u_N).
\end{split}
\end{align}

We know that the $\Gamma$-$\limsup$ is lower semicontinuous and Theorem \ref{thm:interpolationaffine} tells us that $u_N\rightharpoonup^*u$ in $BV(0,1)$. Since we know the $\Gamma$-$\limsup$ of piecewise affine functions from the previous steps, we have
\begin{align*}
&\Gamma\text{-}\limsup_{n\rightarrow\infty} E_n(\omega,u)\stackrel{\text{l.s.c}}{\leq}\liminf_{N\rightarrow\infty}\left\{\Gamma\text{-}\limsup_{n\rightarrow\infty} E_n(\omega,u_N) \right\}\\
&\hspace{10mm}\leq \limsup_{N\rightarrow\infty}\int_{0}^{1}J_{\mathrm{hom}}(u_N'(x))\,\mathrm{d}x\stackrel{\eqref{w11gl1}}{\leq}\limsup_{N\rightarrow\infty}\int_{0}^{1}J_{\mathrm{hom}}(u'(x))\,\mathrm{d}x=E_{\mathrm{hom}}(u),
\end{align*}
which gives us the limsup-inequality for $W^{1,1}(0,1)$. As argued in the beginning of the proof, this shows the limsup-inequality for the functional without boundary constraints.\\

\step 4 Convergence of minimum problems.

The convergence of minimum problems follows directly from the coercivity of $E_n^{\ell}(\omega,\cdot)$ and the fundamental Theorem of $\Gamma$-convergence (see e.g.~\cite[Thm.1.21]{Braides}). Since $J_{\mathrm{hom}}$ is decreasing, we get from the Jensen inequality and from $D^su\geq 0$
\begin{align*}
	\min_u E_{\mathrm{hom}}^{\ell}(u)\geq J_{\mathrm{hom}}\left(\int_{0}^{1}u'(x)\,\mathrm{d}x\right)\geq J_{\mathrm{hom}}(Du[0,1])=J_{\mathrm{hom}}(\ell).
\end{align*}
And the reverse inequality, we get from testing with $u(x)=\ell x$.

\end{proof}

\subsection{Properties of $J_{\rm hom}^L$, proofs of Proposition~\ref{Prop:existencejhom} and \ref{Prop:jhomlcontinuous} }
\label{sec:4-2-and4-3}

For later reference, we point out two further special properties of the approximating functions.
\begin{proposition}
\label{prop:ML}
Let the approximation be defined as above. Let $A=[a,b)$, $a,b\in\mathbb{R}$, be an interval, and $A_N:=NA\cap\mathbb{Z}$.
\begin{itemize}
    \item[(i)] There exists $L^*$ such that for all $L>L^*$ it holds true that
\begin{align}
\label{schrankeableitung}
m_j^L(\omega)\leq -M_L,
\end{align}
with a constant $M_L>0$ independent of $j$ and $\omega$. Further, we have that
\begin{align}
\label{schrankeableitungwachstum}
M_L\to\infty\quad\text{for}\ L\to\infty.
\end{align}
    \item[(ii)] It exists $\Omega_1\subset\Omega$ with $\mathbb{P}(\Omega_1)=1$ such that for all $\omega\in\Omega_1$, all $j=1,...,K$, it holds true that \begin{align}
\label{lipschitzhoelderestimate}
	\begin{split}
	\dfrac{1}{|A_N|}\sum_{j=1}^{K}\sum_{i\in A_N}\left|J_j^L\left(\omega,i,x\right)-J_j^L\left(\omega,i,y\right)\right|
	\leq C^{L,H,(N)}(\omega)\max\{|x-y|^\alpha,|x-y|\},
	\end{split}
\end{align}
for every $x,y\in\mathbb{R}$ and independent of the choice of $A$,  with $0<C^{L,H,(N)}(\omega)\to C^{L,H}$ for $N\to\infty$.
\end{itemize}
\end{proposition}
\begin{proof}
(i) By definition of the subdifferential, it holds true that
\begin{align*}
J_j(\omega,y)\geq J_j(\omega,x)+m_j^L(\omega)(y-x),
\end{align*}
for every $x,y\in(0,\frac{1}{d}]$. Setting $y=\frac{1}{d}$ and $y=z_L$, we get
\begin{align*}
m_j^L(\omega)\leq \dfrac{J_j(\omega,\frac{1}{d})-J_j(\omega,z_L)}{\frac{1}{d}-z_L}\leq\dfrac{d\max\left\{\Psi(\frac{1}{d}),|\frac{1}{d}| \right\}-\left(\frac{1}{d}\Psi(z_L)-d\right)}{\frac{1}{d}-z_L}.
\end{align*}
The denominator is always positive and $\Psi(z_L)\rightarrow\infty$ for $L\to\infty$. Note, that $m_j^L$ is always negative, by definition. The right hand side gets smaller and negative with $L\to\infty$. Therefore, there exists $L^*$ such that for all $L>L^*$ it holds true that
\begin{align*}
m_j^L(\omega)\leq -M_L,
\end{align*}
with a constant $M_L>0$ independent of $j$ and $\omega$. Further, by \eqref{growthpsi}, we have that
\begin{align*}
M_L\to\infty\quad\text{for}\ L\to\infty,
\end{align*}
which proves (i).\\

(ii) It holds true for every $x,y\in\mathbb{R}$ that
\begin{align*}
	\begin{split}
	&\dfrac{1}{|A_N|}\sum_{j=1}^{K}\sum_{i\in A_N}\left|J_j^L\left(\omega,i,x\right)-J_j^L\left(\omega,i,y\right)\right|\\
	\leq& 2\max\left\{KC_{\rm Lip}(z_L),\sum_{j=1}^{K}\dfrac{1}{|A_N|}\sum_{i\in A_N}C_j^{H}(\tau_i\omega)\right\}\max\{|x-y|^\alpha,|x-y|\}.
	\end{split}
\end{align*}
	This estimate can be derived as follows: recall that for a fixed $L$, the Lipschitz constant of $J_j^L(\omega,i,\cdot)$ on $(z_L,\delta)$ is bounded by $C_{\rm Lip}(z_L)$ due to Lemma \ref{lem:lipschitz}. By monotonicity and convexity of $J_j(\omega,i,\cdot)$, the Lipschitz constant of $J_j^L(\omega,i,\cdot)$ on $(-\infty,z_L)$ is also bounded by $C_{\rm Lip}(z_L)$, by construction of the approximating function. Further, $C_j^{H}(\tau_i\omega)$ is the H\"older constant of $J_j^L(\omega,i,\cdot)$ on $[\delta_j(\tau_i\omega),+\infty)$, by definition (see Proposition \ref{Prop:averages} and the related definitions). Now, we have to distinguish between three cases: (i) $x$ and $y$ are both greater than $\delta_j(\tau_i\omega)$, (ii) both are less than $\delta_j(\tau_i\omega)$ and (iii) one is less and one is greater than $\delta_j(\tau_i\omega)$. In the first case (i) the H\"older estimate holds, in the second one (ii) we can use the Lipschitz estimate and in the third one (iii) we can insert $\pm J_j^L(\omega, i,\delta_j(\tau_i\omega))$ and use the triangle inequality, which results in the factor $2$. Since the constants $C_{\rm Lip}(z_L)$ and $C_j^H$ are all positive, we still increase the estimate, if we replace the sums over $B_j$ by sums over $A_N$.

	 Due to (H1) and Proposition \ref{Prop:averages}, it exists $\Omega_1\subset\Omega$ with $\mathbb{P}(\Omega_1)=1$ such that for all $\omega\in\Omega_1$, all $j=1,...,K$, the sum on the right hand side is convergent. Therefore, we finally get 
\begin{align*}
	\begin{split}
	\dfrac{1}{|A_N|}\sum_{j=1}^{K}\sum_{i\in A_N}\left|J_j^L\left(\omega,i,x\right)-J_j^L\left(\omega,i,y\right)\right|
	\leq C^{L,H,(N)}(\omega)\max\{|x-y|^\alpha,|x-y|\},
	\end{split}
\end{align*}
for every $x,y\in\mathbb{R}$ and independent of the choice of $A$, with $C^{L,H,(N)}(\omega)\to C^{L,H}$ almost everywhere for $N\to\infty$. This proves (ii).
\end{proof}

\begin{proof}[Proof of Proposition~\ref{Prop:existencejhom}]
We will prove in the following the pointwise convergence of $J_{\mathrm{hom}}^{L,(N)}(\cdot,z,A)$ almost everywhere on $\Omega$ to a function $J_{\mathrm{hom}}(z)$ independent of $\omega$ and $A$. The upper bound from (LJ2) together with the dominated convergence theorem then yields \eqref{lim:JhomLN}.

\step 1 Fixed $z\in\mathbb{R}$ and intervals $A=[a,b)$ with $a,b\in\mathbb{Z}$.

	First, we prove the assumption for a fixed $z\in\mathbb{R}$. As $J_{\mathrm{hom}}^{L,(N)}(\omega,z,\cdot)$ is subadditive due to the zero boundary constraint and $J_{\mathrm{hom}}^{L,(N)}$ is stationary and ergodic, the Ergodic Theorem \ref{Thm:AkcogluKrengel} due to Akcoglu and Krengel can be applied. Therefore, there exists $\Omega_z\subset\Omega$ with $\mathbb{P}(\Omega_z)=1$ such that for every $\omega\in\Omega_z$ and for every $A=[a,b)$ with $a,b\in\mathbb{Z}$, the limit
	\begin{align*}
		\lim\limits_{N\rightarrow\infty} J_{\mathrm{hom}}^{L,(N)}(\omega,z,A)
	\end{align*}
	exists and is independent of $\omega$ and $A$. Note, that this holds true because of the countability of the intervals, since we only demand for $a,b\in\mathbb{Z}$. Otherwise, the property $\mathbb{P}(\Omega_z)=1$ cannot be ensured. More precisely, it holds true that $\Omega_z=\bigcap_{a,b\in\mathbb{Z}}\Omega_A$, with $\Omega_A\subset\Omega$ being the set on which the ergodic theorem holds true for a fixed $A$. Considering $A=[0,N)$, we get
	\begin{align*}
		J_{\mathrm{hom}}^{L}(z)=\lim\limits_{N\to\infty}J_{\mathrm{hom}}^{L,(N)}(\omega,z,A).
	\end{align*}

\step 2 Fixed $z\in\mathbb{R}$ and intervals $A=[a,b)$ with $a,b\in\mathbb{R}$.
	
	In order to pass to general intervals with $a,b\in\mathbb{R}$, we argue as in \cite[Proposition 1]{DalMasoModica1985}. For every $\delta>0$, there exists $T$ big enough and intervals $A_\delta^-$, $A_\delta^+$ with $a_{\delta}^-,b_{\delta}^-,a_{\delta}^+,b_{\delta}^+\in\mathbb{Z}$ such that it holds true
	\begin{align}
	\label{setrelations}
	A_\delta^-\subset TA\subset A_\delta^+,\quad \dfrac{|A_\delta^-|}{|TA|}\geq 1-\delta,\quad\dfrac{|TA|}{|A_{\delta}^+|}\geq 1-\delta.
	\end{align}
	From (LJ2), we get, for all intervals $B\subset A$ and $N$ big enough, the inequality
	\begin{align}
	\label{inequ}
		J_{\mathrm{hom}}^{L,(N)}(\omega,z,A)\leq J_{\mathrm{hom}}^{L,(N)}(\omega,z,B)+\dfrac{|N(A\setminus B)\cap\mathbb{Z}|}{|N(A)\cap\mathbb{Z}|}C_z,
	\end{align}
	with a constant $C_z$ depending on $z$, which can be seen as follows. Taking a minimizer $\phi$ of the minimum problem related to $B$, with notation from \ref{def:iminmax}, one has
	\begin{align*}
	&J_{\mathrm{hom}}^{L,(N)}(\omega,z,A)\leq \dfrac{1}{|NB\cap\mathbb{Z}|}\sum_{j=1}^{K}\sum_{i=i_{min}^B}^{i_{max}^B+1-j}J_j^L\left(\omega,i,z+\dfrac{\phi^{i+j}-\phi^i}{j}\right)\\
	&\hspace{22mm}+\dfrac{1}{|NA\cap\mathbb{Z}|}\sum_{j=1}^{K}\sum_{\substack{i=i_{min}^{A\setminus B}\\ i\in N(A\setminus B)\cap\mathbb{Z}}}^{i_{max}^{A\setminus B}+1-j}J_j^L(\omega,i,z)+\dfrac{1}{|NA\cap\mathbb{Z}|}\sum_{j=2}^{K}\sum_{i=i_{max}^{B}+2-j}^{i_{max}^{B}}J_j^L(\omega,i,z)\\
	&\stackrel{(LJ2)}{\leq}J_{\mathrm{hom}}^{L,(N)}(\omega,z,B)+\dfrac{1}{|NA\cap\mathbb{Z}|}d\max\{\Psi(z),|z|\}\sum_{j=1}^{K}\left(\sum_{\substack{i=i_{min}^{A\setminus B}\\ i\in N(A\setminus B)\cap\mathbb{Z}}}^{i_{max}^{A\setminus B}+1-j}1+\sum_{i=i_{max}^{B}+2-j}^{i_{max}^{B}}1\right)\\
	&\leq J_{\mathrm{hom}}^{L,(N)}(\omega,z,B)+\dfrac{1}{|NA\cap\mathbb{Z}|}d\max\{\Psi(z),|z|\}\left( K|N(A\setminus B)\cap\mathbb{Z}|+K^2\right),
	\end{align*}
	where \eqref{inequ} then holds true for $N$ big enough. Now, we get from Step 1
	\begin{align*}
	&J_{\mathrm{hom}}^{L}(z)=\lim\limits_{N\to\infty}J_{\mathrm{hom}}^{L,(N)}(\omega,z,A_\delta^+)\\
	\stackrel{\eqref{inequ}}{\leq}&\liminf_{N\to\infty}J_{\mathrm{hom}}^{L,(N)}(\omega,z,TA)+\liminf_{N\to\infty}\dfrac{|N(A_\delta^+\setminus TA)\cap\mathbb{Z}|}{|N(A_\delta^+)\cap\mathbb{Z}|}C_z\\
	=&\liminf_{N\to\infty}J_{\mathrm{hom}}^{L,(N)}(\omega,z,TA)+\dfrac{|(A_\delta^+\setminus TA)|}{|(A_\delta^+)|}C_z \stackrel{\eqref{setrelations}}{\leq}\limsup_{N\to\infty}J_{\mathrm{hom}}^{L,(N)}(\omega,z,TA)+\delta C_z\\
	\stackrel{\eqref{inequ}}{\leq}&\lim_{N\to\infty}J_{\mathrm{hom}}^{L,(N)}(\omega,z,A_\delta^-)+\left(\delta+\dfrac{|(TA\setminus A_\delta^-)|}{|(TA)|}\right)C_z\stackrel{\eqref{setrelations}}{=}J_{\mathrm{hom}}^{L}(z)+2C_z\delta.
	\end{align*}
	This shows
	\begin{align*}
		J_{\mathrm{hom}}^{L}(z)=\lim\limits_{N\to\infty}J_{\mathrm{hom}}^{L,(N)}(\omega,z,A),
	\end{align*}
	for $A=[a,b)$ with $a,b\in\mathbb{R}$, since $\delta$ can be chosen arbitrarily small. Note that for fixed $T>0$, $\lim\limits_{N\to\infty}J_{\mathrm{hom}}^{L,(N)}(\omega,z,TA)$ and $\lim\limits_{N\to\infty}J_{\mathrm{hom}}^{L,(N)}(\omega,z,A)$ are the same. \\

\step 3 $z\in\mathbb{R}$ and intervals $A=[a,b)$ with $a,b\in\mathbb{R}$.	
	
	With the definition of $\Omega_z$ from the previous steps, we define $\Omega_0:=\bigcap_{z\in\mathbb{Q}}\Omega_z$. It holds true that $\mathbb{P}(\Omega_0)=1$ and that we have for every $\omega\in\Omega_0$ 
	\begin{align}
	\label{jhomfuerQ}
		J_{\mathrm{hom}}^{L}(z)=\lim\limits_{N\to\infty}J_{\mathrm{hom}}^{L,(N)}(\omega,z,A),
	\end{align}
	for arbitrary $A$ and all $z\in\mathbb{Q}$. This was shown in the steps before.	
	
	Now, we derive the existence of the limit of $J_{\mathrm{hom}}^{L,(N)}(\omega,z,A)$ also for $z\in\mathbb{R}\setminus\mathbb{Q}$ and $\omega\in\Omega_0$. Note, that the ergodic theorem provides existence of that limit only for $\omega\in\Omega_z$ and not for $\omega\in\Omega_0$. For this, let $z\in\mathbb{R}\setminus\mathbb{Q}$ and $(z_k)_{k\in\mathbb{N}}\subset\mathbb{Q}$ be a sequence with $z_k\rightarrow z$. Strictly speaking, we also can assume $z\in\mathbb{R}$, but it is not necessary, because we already dealt with the case $z\in\mathbb{Q}$. By contrast, the assumption $(z_k)_{k\in\mathbb{N}}\subset\mathbb{Q}$ is essential, because with this we can use \eqref{jhomfuerQ} for $z_n$ in the following.
	
	With notation from \ref{def:iminmax}, we denote the minimizer related to the minimum problem of $J_{\mathrm{hom}}^{L,(N)}(\omega,z,A)$ by $\phi_{N,z}:(NA\cap\mathbb{Z})\cup\{i_{max}^A+1\}\to\mathbb{R}$ with $\phi_{N,z}^{i_{min}^A}=0=\phi_{N,z}^{i_{max}^A+1}$ (we give up the index $A$ for the minimizer for better readability), which means that it holds true that
	\begin{align}
	\label{beginchange}
	J_{\mathrm{hom}}^{L,(N)}(\omega,z,A)=\dfrac{1}{|NA\cap\mathbb{Z}|}\sum_{j=1}^{K}\sum_{i=i_{min}^A}^{i_{max}^A+1-j}J_j^L\left(\omega,i,z+\dfrac{\phi_{N,z}^{i+j}-\phi_{N,z}^{i}}{j}\right).
	\end{align}
Therefore, we have
	\begin{align*}
	\begin{split}
	&J_{\mathrm{hom}}^{L,(N)}(\omega,z,A)=\dfrac{1}{|NA\cap\mathbb{Z}|}\sum_{j=1}^{K}\sum_{i=i_{min}^A}^{i_{max}^A+1-j}J_j^L\left(\omega,i,z+\dfrac{\phi_{N,z}^{i+j}-\phi_{N,z}^{i}}{j}\right)\\
	=&\dfrac{1}{|NA\cap\mathbb{Z}|}\sum_{j=1}^{K}\sum_{i=i_{min}^A}^{i_{max}^A+1-j}J_j^L\left(\omega,i,z_k+\dfrac{\phi_{N,z}^{i+j}-\phi_{N,z}^{i}}{j}\right)\\&+\dfrac{1}{|NA\cap\mathbb{Z}|}\sum_{j=1}^{K}\sum_{i=i_{min}^A}^{i_{max}^A+1-j}\left(J_j^L\left(\omega,i,z+\dfrac{\phi_{N,z}^{i+j}-\phi_{N,z}^{i}}{j}\right)-J_j^L\left(\omega,i,z_k+\dfrac{\phi_{N,z}^{i+j}-\phi_{N,z}^{i}}{j}\right) \right)\\
	\geq& J_{\mathrm{hom}}^{L,(N)}(\omega,z_k,A)\\&-\dfrac{1}{|NA\cap\mathbb{Z}|}\sum_{j=1}^{K}\sum_{i=i_{min}^A}^{i_{max}^A+1-j}\left|J_j^L\left(\omega,i,z+\dfrac{\phi_{N,z}^{i+j}-\phi_{N,z}^{i}}{j}\right)-J_j^L\left(\omega,i,z_k+\dfrac{\phi_{N,z}^{i+j}-\phi_{N,z}^{i}}{j}\right)\right|.
	\end{split}
	\end{align*}
	Since $|z-z_k|\leq|z-z_k|^\alpha$ for $k$ large enough, we continue with this estimate by using \eqref{lipschitzhoelderestimate} and get
	\begin{align}
	\label{gl}
	\begin{split}
	&J_{\mathrm{hom}}^{L,(N)}(\omega,z,A)\geq J_{\mathrm{hom}}^{L,(N)}(\omega,z_k,A)-C^{L,H,(N)}(\omega)|z-z_k|^\alpha.
	\end{split}
	\end{align}
	We now calculate first the limit $\liminf_{N\rightarrow\infty}$, with $C^{L,H,(N)}(\omega)\to C^{L,H}$ from \eqref{lipschitzhoelderestimate},  and subsequently $\limsup_{k\rightarrow\infty}$ of \eqref{gl}. Since we assumed $(z_k)_{k\in\mathbb{N}}\subset\mathbb{Q}$, we get
	\begin{align*}
	\liminf\limits_{N\rightarrow\infty}J_{\mathrm{hom}}^{L,(N)}(\omega,z,A)\geq \limsup_{k\rightarrow\infty}J_{\mathrm{hom}}^L(z_n).
	\end{align*}
	Now, we can restart the whole calculation, from \eqref{beginchange} onwards, by changing the roles of $z$ and $z_k$. Hence, we first have to take the limit $\limsup_{N\rightarrow\infty}$ and subsequently $\liminf_{k\rightarrow\infty}$, and get analogously
	\begin{align*}
	\liminf_{k\rightarrow\infty}J_{\mathrm{hom}}^L(z_n)\geq\limsup\limits_{N\rightarrow\infty}J_{\mathrm{hom}}^{L,(N)}(\omega,z,A).
	\end{align*}
	Together, the two estimates yield
	\begin{align}
	\label{notyetcontinuous}
	\begin{split}
	J_{\mathrm{hom}}^L(z)&=\lim\limits_{N\rightarrow\infty}J_{\mathrm{hom}}^{L,(N)}(\omega,z,A)=\lim_{k\rightarrow\infty}J_{\mathrm{hom}}^L(z_k),\quad \text{for all}\ z\in\mathbb{R}\setminus\mathbb{Q}\ \text{and all}\ (z_k)_k\subset\mathbb{Q}.
	\end{split}
	\end{align}
	This shows that for $\omega\in\Omega'$ the limit of $J_{\mathrm{hom}}^{L,(N)}(\omega,z,A)$ exists and is independent of $\omega$ and $A$ for all $z\in\mathbb{R}\setminus\mathbb{Q}$.  Altogether, we have that the limit of $J_{\mathrm{hom}}^{L,(N)}(\omega,z,A)$ exists for every $z\in\mathbb{R}$, is independent of $\omega$ and $A$, and equals $J_{\mathrm{hom}}^L(z)$. This finally proves the assumption.
\end{proof}

\begin{proof}[Proof of Proposition~\ref{Prop:jhomlcontinuous}]

We prove the different properties separately in the next steps.

\step 1 Continuity.

	Let $(z_k)_{k\in\mathbb{N}}\subset\mathbb{R}$ be a sequence converging to $z\in\mathbb{R}$. Let $\phi_{N,z}$ be a minimizing sequence such that it holds $\phi_{N,z}^N=\phi_{N,z}^0=0$ and
	\begin{align}
	\label{beginchange2}
	\lim\limits_{N\rightarrow\infty}\dfrac{1}{N}\sum_{j=1}^{K}\sum_{i=0}^{N-j}J_j^L\left(\omega,i,z+\dfrac{\phi_{N,z}^{i+j}-\phi_{N,z}^{i}}{j}\right)=J_{\mathrm{hom}}^L(z)
	\end{align}
	for $\omega\in\Omega_0$ defined in Proposition \ref{Prop:existencejhom}.
	Then, it holds true that
	\begin{align*}
	&\dfrac{1}{N}\sum_{j=1}^{K}\sum_{i=0}^{N-j}J_j^L\left(\omega,i,z+\dfrac{\phi_{N,z}^{i+j}-\phi_{N,z}^{i}}{j}\right)=\dfrac{1}{N}\sum_{j=1}^{K}\sum_{i=0}^{N-j}J_j^L\left(\omega,i,z_k+\dfrac{\phi_{N,z}^{i+j}-\phi_{N,z}^{i}}{j}\right)\\&\quad\quad+\dfrac{1}{N}\sum_{j=1}^{K}\sum_{i=0}^{N-j}\left( J_j^L\left(\omega,i,z+\dfrac{\phi_{N,z}^{i+j}-\phi_{N,z}^{i}}{j}\right)-J_j^L\left(\omega,i,z_k+\dfrac{\phi_{N,z}^{i+j}-\phi_{N,z}^{i}}{j}\right) \right)\\
	\geq& J_{\mathrm{hom}}^{L,(N)}(\omega,z_k)-C^{L,H,(N)}(\omega)|z-z_k|^\alpha,
	\end{align*}
	where the last step is due to \eqref{lipschitzhoelderestimate} and since $|z-z_k|\leq|z-z_k|^\alpha$ for $k$ large enough. Recalling that $\sup (f-g)\geq \sup f -\sup g$, we continue by taking the limit $N\rightarrow\infty$, with $C^{L,H,(N)}(\omega)\to C^{L,H}$ from \eqref{lipschitzhoelderestimate}, and subsequently limsup $k\rightarrow\infty$. Proposition \ref{Prop:averages} provides boundedness of the sums and therefore we get
	\begin{align*}
	J_{\mathrm{hom}}^L(z)\geq\limsup_{k\rightarrow\infty} J_{\mathrm{hom}}^L(z_k).
	\end{align*}
	with the result of Proposition \ref{Prop:existencejhom}. Restarting the whole calculation, from \eqref{beginchange2} onwards, with changing roles of $z$ and $z_k$, we get analogously by by taking the limit $N\rightarrow\infty$ and subsequently $\liminf_{k\rightarrow\infty}$
	\begin{align*}
	J_{\mathrm{hom}}^L(z)\leq\liminf_{k\rightarrow\infty} J_{\mathrm{hom}}^L(z_k).
	\end{align*}
	Together, this shows $J_{\mathrm{hom}}^L(z)=\lim\limits_{k\rightarrow\infty}J_{\mathrm{hom}}^L(z_k)$ and therefore $J_{\mathrm{hom}}^L$ is continuous.\\
	
\step 2 Convexity.

	We need to show
	\begin{align*}
	J_{\mathrm{hom}}^L\left(tz_1+(1-t)z_2\right)\leq t J_{\mathrm{hom}}^L(z_1)+(1-t)J_{\mathrm{hom}}^L(z_2)
	\end{align*}
	for every $t\in[0,1]$ and every $z_1,z_2\in(0,+\infty)$. Otherwise, the inequality is trivial. Fix $t\in[0,1]$. We use in the following the notation from \ref{def:iminmax}. Let $\phi_{N,z_1}:N[0,t+\frac1N)\cap\mathbb{Z}\rightarrow\mathbb{R}$ be a minimizer related to the minimum problem of $J_{\mathrm{hom}}^{L,(N)}(\omega,z_1,[0,t))$, that is $\phi_{N,z_1}^s=0=\phi_{N,z_1}^{i_{max}^{[0,t)}+1-s}$ for $s=0,...,K-1$ and
	\begin{align*}
	J_{\mathrm{hom}}^{L,(N)}(\omega,z_1,[0,t))=\dfrac{1}{|N[0,t)\cap\mathbb{Z}|}\sum_{j=1}^{K}\sum_{i=0}^{i_{max}^{[0,t)}+1-j}J_j^L\left(\omega,i,z_1+\dfrac{\phi_{N,z_1}^{i+j}-\phi_{N,z_1}^{i}}{j}\right).
	\end{align*}
	Further, let $\phi_{N,z_2}:N[t,N]\cap\mathbb{Z}\rightarrow\mathbb{R}$ be a minimizer of the minimum problem of $J_{\mathrm{hom}}^{L,(N)}(\omega,z_2,[t,1))$, that is $\phi_{N,z_2}^{i_{min}^{[t,1)}+s}=0=\phi_{N,z_2}^{N-s}$ for $s=0,...,K-1$ and
	\begin{align*}
	J_{\mathrm{hom}}^{L,(N)}(\omega,z_2,[t,1))=\dfrac{1}{|N[t,1)\cap\mathbb{Z}|}\sum_{j=1}^{K}\sum_{i=i_{min}^{[t,1)}}^{N-j}J_j^L\left(\omega,i,z_2+\dfrac{\phi_{N,z_2}^{i+j}-\phi_{N,z_2}^{i}}{j}\right).
	\end{align*}
	This given, we define
	\begin{align*}
	\tilde{\phi}_{N}^{i}:=\begin{cases}
	\phi_{N,z_1}^{i}=0&\quad\text{for}\ 0\leq i\leq K-1,\\
	\phi_{N,z_1}^{i}+(i-K)(1-t)(z_1-z_2)&\quad\text{for}\ K \leq i\leq i_{max}^{[0,t)}+1-K,\\
	i_{max}^{[0,t)}(1-t)(z_1-z_2)&\quad\text{for}\ i_{max}^{[0,t)}+2-K \leq i\leq i_{max}^{[t,1)}+K,\\
	\phi_{N,z_2}^{i}+(N-i)(z_1-z_2)t&\quad\text{for}\ i_{min}^{[t,1)}+K \leq i\leq N-K ,\\
	\phi_{N,z_2}^{i}=0&\quad\text{for}\ N-K+1\leq i\leq N.
	\end{cases}
	\end{align*}
	Then, $\tilde{\phi}$ fulfils the constraints of the infimum problem of $J_{\mathrm{hom}}^{L,(N)}$ and therefore it holds true that
	\begin{align}
	\label{convex}
	\begin{split}
	&J_{\mathrm{hom}}^{L,(N)}\left(\omega,tz_1+(1-t)z_2\right)\leq\dfrac{1}{N}\sum_{j=1}^{K}\sum_{i=0}^{N-j}J_j^L\left(\omega,i,tz_1+(1-t)z_2+\dfrac{\tilde{\phi}_{N}^{i+j}-\tilde{\phi}_{N}^i}{j}    \right)\\
	=&\dfrac{1}{N}\sum_{j=1}^{K}\sum_{i=0}^{i_{max}^{[0,t)}+1-j}J_j^L\left(\omega,i,tz_1+(1-t)z_2+\dfrac{\tilde{\phi}_{N}^{i+j}-\tilde{\phi}_{N}^i}{j}\right)\\&+\dfrac{1}{N}\sum_{j=1}^{K}\sum_{i=i_{min}^{[t,1)}}^{N-j}J_j^L\left(\omega,i,tz_1+(1-t)z_2+\dfrac{\tilde{\phi}_{N}^{i+j}-\tilde{\phi}_{N}^i}{j}\right)\\&+\dfrac{1}{N}\sum_{j=2}^{K}\sum_{s=0}^{j-2}J_j^L\left(\omega,i_{max}^{[0,t)}-s,tz_1+(1-t)z_2+\dfrac{\tilde{\phi}_{N}^{i_{max}^{[0,t)}-s+j}-\tilde{\phi}_{N}^{i_{max}^{[0,t)}-s}}{j}\right)
	\end{split}
	\end{align}
We consider all three terms of \eqref{convex} individually and bring it together afterwards. For abbreviation, we use in the following $\overline{z}:=tz_1+(1-t)z_2$. We start with the first term of \eqref{convex} and therefore estimate
	\begin{align*}
	\begin{split}
	&\dfrac{1}{N}\sum_{j=1}^{K}\sum_{i=0}^{i_{max}^{[0,t)}+1-j}\left(J_j^L\left(\omega,i,\overline{z}+\dfrac{\tilde{\phi}_{N}^{i+j}-\tilde{\phi}_{N}^i}{j}\right)-J_j^L\left(\omega,i,z_1+\dfrac{\phi_{N,z_1}^{i+j}-\phi_{N,z_1}^i}{j}\right)\right)\\
	=&\dfrac{1}{N}\sum_{j=1}^{K}\sum_{i=0}^{K-1}\left(J_j^L\left(\omega,i,\overline{z}+\dfrac{\tilde{\phi}_{N}^{i+j}-\tilde{\phi}_{N}^i}{j}\right)-J_j^L\left(\omega,i,z_1+\dfrac{\phi_{N,z_1}^{i+j}-\phi_{N,z_1}^i}{j}\right)\right)\\&+\dfrac{1}{N}\sum_{j=1}^{K}\sum_{i=i_{max}^{[0,t)}+2-K-j}^{i_{max}^{[0,t)}+1-j}\left(J_j^L\left(\omega,i,\overline{z}+\dfrac{\tilde{\phi}_{N}^{i+j}-\tilde{\phi}_{N}^i}{j}\right)-J_j^L\left(\omega,i,z_1+\dfrac{\phi_{N,z_1}^{i+j}-\phi_{N,z_1}^i}{j}\right)\right).
	\end{split}
	\end{align*}
	By definition of $\tilde{\phi}_N$, we get boundedness of the differences
	\begin{align*}
	\left|\dfrac{\tilde{\phi}_{N}^{i+j}-\tilde{\phi}_{N}^i}{j}-\dfrac{\phi_{N,z_1}^{i+j}-\phi_{N,z_1}^i}{j}\right|\leq C
	\end{align*}
	in both cases $0\leq i\leq K-1$ and $i_{max}^{[0,t)}+2-K-j\leq i\leq i_{max}^{[0,t)}+1-j$. For $\epsilon>0$ and $I_{\epsilon}(x):=[x-\epsilon,x+\epsilon)\cap[0,1]$ for $N$ big enough, we then get by \eqref{lipschitzhoelderestimate}
\begin{align*}
\begin{split}
	&\dfrac{1}{N}\sum_{j=1}^{K}\sum_{i=0}^{i_{max}^{[0,t)}+1-j}J_j^L\left(\omega,i,\overline{z}+\dfrac{\tilde{\phi}_{N}^{i+j}-\tilde{\phi}_{N}^i}{j}\right)\\
	\leq&  \left(t+\dfrac{2}{N}\right) J_{\mathrm{hom}}^{L,(N)}(\omega,z_1,[0,t))+\dfrac{\left|NI_{\epsilon}(0)\cap\mathbb{Z}\right|}{N}\hat{C}(\omega)_+\dfrac{\left|NI_{\epsilon}(t)\cap\mathbb{Z}\right|}{N}\hat{C}(\omega).
	\end{split}
	\end{align*}
The second term of \eqref{convex} can be discussed analogously to the first one, with the analogue result
\begin{align*}
\begin{split}
	&\dfrac{1}{N}\sum_{j=1}^{K}\sum_{i=i_{min}^{[t,1)}}^{N-j}J_j^L\left(\omega,i,\overline{z}+\dfrac{\tilde{\phi}_{N}^{i+j}-\tilde{\phi}_{N}^i}{j}\right)\\
	&\leq \left(1-t+\dfrac{2}{N}\right) J_{\mathrm{hom}}^{L,(N)}(\omega,z_2,[t,1)) +\dfrac{\left|NI_{\epsilon}(t)\cap\mathbb{Z}\right|}{N}\hat{C}(\omega)+\dfrac{\left|NI_{\epsilon}(1)\cap\mathbb{Z}\right|}{N}\hat{C}(\omega).
	\end{split}
	\end{align*}
The third term of \eqref{convex} is
\begin{align*}
\dfrac{1}{N}\sum_{j=2}^{K}\sum_{s=0}^{j-2}J_j^L\left(\omega,i_{max}^{[0,t)}-s,\overline{z}+\dfrac{\tilde{\phi}_{N}^{i_{max}^{[0,t)}-s+j}-\tilde{\phi}_{N}^{i_{max}^{[0,t)}-s}}{j}\right).
\end{align*}
For the given values of $s$ and $j$, it holds true that $\tilde{\phi}_N^{i+j}-\tilde{\phi}_N^{i}=0$ because it is $i_{max}^{[0,t)}+2-K\leq i\leq i_{max}^{[0,t)}+K$. Therefore, we can estimate
\begin{align*}
\begin{split}
&\dfrac{1}{N}\sum_{j=2}^{K}\sum_{s=0}^{j-2}J_j^L\left(\omega,i_{max}^{[0,t)}-s,\overline{z}+\dfrac{\tilde{\phi}_{N}^{i_{max}^{[0,t)}-s+j}-\tilde{\phi}_{N}^{i_{max}^{[0,t)}-s}}{j}\right)\\
\leq& \dfrac{1}{N}\sum_{j=2}^{K}\sum_{s=0}^{j-2}d\max\left\{\Psi\left(\overline{z}\right),\left|\overline{z} \right| \right\}\leq \dfrac{1}{N}\dfrac{1}{2}(K+1)KC\to0\quad\text{for}\ N\to\infty.
\end{split}
\end{align*}
Putting together all previous estimates, we can calculate the limit $N\to\infty$ in \eqref{convex} and get with the convergence of the constant $\hat{C}(\omega)\to\hat{C}$ from \eqref{lipschitzhoelderestimate}
\begin{align*}
	&J_{\mathrm{hom}}^L\left(tz_1+(1-t)z_2\right)\leq tJ_{\mathrm{hom}}^L(z_1)+\epsilon \hat{C}+2\epsilon\hat{C} +(1-t)J_{\mathrm{hom}}^L(z_2)+ 2\epsilon\hat{C}+\epsilon\hat{C}+0,
	\end{align*}
	where Proposition \ref{Prop:existencejhom} yields the existence of $\Omega_0\subset\Omega$ with $\mathbb{P}(\Omega_0)=1$ such that the above calculated limit exists for all $\omega\in\Omega_0$ and all $z_1,z_2\in\mathbb{R}$. Finally, we can perform the limit $\epsilon\to0$ and get
	\begin{align*}
	J_{\mathrm{hom}}^L\left(tz_1+(1-t)z_2\right)\leq t J_{\mathrm{hom}}^L(z_1)+(1-t)J_{\mathrm{hom}}^L(z_2),
	\end{align*}
	which shows convexity.

\step 3 $\Gamma$-limit.	
	
	We first show the liminf-inequality.	Let $(z_N)_{N\in\mathbb{N}}$ be a sequence converging to $z$. Then, for every $N\in\mathbb{N}$ we denote a minimizer related to the minimum problem of $J_{\mathrm{hom}}^{L,(N)}(\omega,z_N,A)$ by $\phi_{N,z_N}:(\mathbb{Z}\cap NA)\rightarrow\mathbb{R}$, that is
	\begin{align*}
	&\dfrac{1}{|NA\cap\mathbb{Z}|}\sum_{j=1}^{K}\sum_{\substack{i\in\mathbb{Z}\cap NA\\ i+j-1\in NA}}J_j^L\left(\omega,i,z_N+\dfrac{\phi_{N,z_N}^{i+j}-\phi_{N,z_N}^{i}}{j}\right)=J_{\mathrm{hom}}^{L,(N)}(\omega,z_N,A).
	\end{align*}
	Now, we have
	\begin{align*}
	&J_{\mathrm{hom}}^{L,(N)}(\omega,z_N,A)=\dfrac{1}{|NA\cap\mathbb{Z}|}\sum_{j=1}^{K}\sum_{\substack{i\in\mathbb{Z}\cap NA\\ i+j-1\in NA}}J_j^L\left(\omega,i,z_N+\dfrac{\phi_{N,z_N}^{i+j}-\phi_{N,z_N}^{i}}{j}\right)\\
	=&\dfrac{1}{|NA\cap\mathbb{Z}|}\sum_{j=1}^{K}\sum_{\substack{i\in\mathbb{Z}\cap NA\\ i+j-1\in NA}}J_j^L\left(\omega,i,z+\dfrac{\phi_{N,z_N}^{i+j}-\phi_{N,z_N}^{i}}{j}\right)\\&+\dfrac{1}{|NA\cap\mathbb{Z}|}\sum_{j=1}^{K}\sum_{\substack{i\in\mathbb{Z}\cap NA\\ i+j-1\in NA}}\left( J_j^L\left(\omega,i,z_N+\dfrac{\phi_{N,z_N}^{i+j}-\phi_{N,z_N}^{i}}{j}\right)-J_j^L\left(\omega,i,z+\dfrac{\phi_{N,z_N}^{i+j}-\phi_{N,z_N}^{i}}{j}\right) \right)\\
	\geq& J_{\mathrm{hom}}^{L,(N)}(\omega,z,A)-C^{L,H,(N)}(\omega)|z-z_N|^\alpha,
	\end{align*}
	where the last step is due to \eqref{lipschitzhoelderestimate} and since $|z-z_k|\leq|z-z_k|^\alpha$ for $k$ large enough. With Proposition \ref{Prop:existencejhom} and \eqref{lipschitzhoelderestimate}, we get for $\omega\in\Omega_0$, by taking the limit $\liminf_{N\rightarrow\infty}$,
	\begin{align*}
	\liminf_{N\rightarrow\infty}J_{\mathrm{hom}}^{L,(N)}(\omega,z_N,A)\geq\liminf_{N\rightarrow\infty}J_{\mathrm{hom}}^{L,(N)}(\omega,z,A)=J_{\mathrm{hom}}^L(z),
	\end{align*}
	which shows the liminf-inequality.\\
	
	The limsup-inequality is trivial, since we can take for every $z\in\mathbb{R}$ the constant recovery sequence $z_N:=z$ and get
	\begin{align*}
	    \limsup_{N\to\infty}J_{\mathrm{hom}}^{L,(N)}(\omega,z_N,A)=\limsup_{N\to\infty}J_{\mathrm{hom}}^{L,(N)}(\omega,z,A)=J_{\mathrm{hom}}^{L}(z),
	\end{align*}
	due to Proposition \ref{Prop:existencejhom}. This shows the limsup-inequality and completes the proof of the $\Gamma$-limit.
\end{proof}

\subsection{Properties of $J_{\rm hom}$, proofs of Proposition~\ref{Prop:Jhom} and Proposition~\ref{Prop:Jhomzklein}}
\label{sec:propertiesJhom}
\begin{proof}[Proof of Proposition~\ref{Prop:Jhom}]
We prove the different assumptions separately in the following steps.

\step 1 Equation \eqref{eq:defJhom}

In Proposition \ref{Prop:Jhomzklein} we have shown the pointwise convergence of $J_{\mathrm{hom}}^{(N)}(\cdot,z,A)$ almost everywhere on $\Omega$ to a function $J_{\mathrm{hom}}(z)$ independent of $\omega$ and $A$. The upper bound from (LJ2) together with the dominated convergence theorem then yields \eqref{lim:JhomLN}.

\step 2 Convexity.

The pointwise limit of convex functions is convex. Hence, convexity of $J_{\rm hom}$ follows from Proposition~\ref{Prop:jhomlcontinuous} and Proposition~\ref{Prop:Jhomzklein}.

\step 3 Superlinear growth at $-\infty$, proof of \eqref{superlinwachstum}.

From the condition (LJ2) we have
\begin{align*}
&J_{\mathrm{hom}}^{(N)}(\omega,z,[0,1))=\inf_{\phi\in\mathcal A_{N,K}^0([0,1))}\left\{\dfrac{1}{N}\sum_{j=1}^{K}\sum_{i=0}^{N-j}J_j\left(\omega,i,z+\frac{\phi^{i+j}-\phi^i}j\right) \right\}\\
	\geq& \frac{1}{d}\inf_{\phi\in\mathcal A_{N,K}^0([0,1))}\left\{\sum_{j=1}^{K}\dfrac{1}{N}\sum_{i=0}^{N-j}\Psi\left(z+\frac{\phi^{i+j}-\phi^i}j\right) \right\}-Kd
	\geq \frac{1}{d}\sum_{j=1}^{K}\dfrac{N-j+1}{N}\Psi\left(z \right)-Kd,
\end{align*}
where we used in the last estimate Jensen's inequality and $\phi\in\mathcal A_{N,K}^0([0,1))$. Taking the limit $N\to\infty$, we obtain by Proposition \ref{Prop:Jhomzklein}
\begin{align}\label{hilfsgleichung}
J_{\mathrm{hom}}(z)=\lim\limits_{N\rightarrow\infty}	J_{\mathrm{hom}}^{(N)}(\omega,z)\geq \frac{1}{d}K\Psi(z)-Kd.
\end{align}
Clearly, \eqref{growthpsi} and \eqref{hilfsgleichung} imply \eqref{superlinwachstum}.

\step 4 Lower semicontinuity.

Due to convexity, $J_{\mathrm{hom}}(z)$ is continuous in its inner points, i.e. on $(0,+\infty)$. Further, we get from \eqref{growthpsi}
\begin{align}
\label{lowersemicont}
\lim\limits_{z\to 0^+}J_{\mathrm{hom}}(z)\stackrel{\eqref{hilfsgleichung}}{\geq}\lim\limits_{z\to 0^+}\left(\frac{1}{d} K\Psi(z)-Kd \right)\stackrel{\eqref{growthpsi}}{=}\infty.
\end{align}
This shows lower semicontinuity.

\step 5 Monotonicity.

First of all, $J_{\mathrm{hom}}$ is bounded from below, which can be seen by \eqref{hilfsgleichung} and from the fact that $\Psi(z)\geq 0$ for all $z\in\mathbb{R}$ by definition.

 From \eqref{lowersemicont} and together with convexity, (i) $J_{\mathrm{hom}}$ is either decreasing with $\lim\limits_{z\to+\infty}J_{\mathrm{hom}}(z)=C$ with $C\in\mathbb{R}$, or (ii) $J_{\mathrm{hom}}$ has a unique minimum. In the first case (i), it directly follows that $J_{\mathrm{hom}}$ is monotonically decreasing. The second case (ii) has to be considered separately.

Consider the case that $J_{\mathrm{hom}}$ has a unique minimum, which we call $J_{\mathrm{hom}}(\gamma)$, at the minimizer $z=\gamma$. To show the assertion that $J_{\mathrm{hom}}$ is monotonically decreasing, we need to show $J_{\mathrm{hom}}(z)=J_{\mathrm{hom}}(\gamma)$ for every $z>\gamma$. In fact, it is sufficient to show $J_{\mathrm{hom}}(z)\leq J_{\mathrm{hom}}(\gamma)$, because the reverse inequality is clear since $J_{\mathrm{hom}}(\gamma)$ is the unique minimizer.

For this, consider $z>\gamma$. Let $z_N:\{0,...,N-1\}\to\mathbb{R}$ be a minimizer related to the minimum problem of $J_{\mathrm{hom}}^{(N)}(\omega,\gamma)$, that is $z_N^s=z_N^{N-s-1}=\gamma$ for $s=0,...,K-2$, $\sum_{i=0}^{N-1}z_N^i=N\gamma$ and
\begin{align*}
J_{\mathrm{hom}}^{(N)}(\omega,\gamma)=\dfrac{1}{N}\sum_{j=1}^{K}\sum_{i=0}^{N-j}J_j\left(\omega,i,\frac{1}{j}\sum_{k=i}^{i+j-1}z^k \right).
\end{align*}
We set
	\begin{align*}
	\tilde{z}_N^i=\begin{cases}
	z&\quad\text{for}\ i=0,...,K-2\ \text{and}\ i=N-K+1,...,N-1,\\
	(z-\gamma)(N/2-K+1)+z_N^{K-1}&\quad\text{for}\ i=K-1,\\
	(z-\gamma)(N/2-K+1)+z_N^{K-1}&\quad\text{for}\ i=N-K,\\
	z_N^i&\quad\text{otherwise.}
	\end{cases}
	\end{align*}
	which fulfils the constraint $\sum_{i=0}^{N-1}\tilde{z}_N^i=Nz$ and $\tilde{z}_N^s=\tilde{z}_N^{N-s-1}=z$ for $s=0,...,K-2$. Then, it holds true that
	\begin{align}
	\label{monotonicallydecreasingequation}
	\begin{split}
	&J_{\mathrm{hom}}^{(N)}(\omega,z)\leq\dfrac{1}{N}\sum_{j=1}^{K}\sum_{i=0}^{N-j}J_j\left(\omega,i,\frac{1}{j}\sum_{k=i}^{i+j-1}\tilde{z}_N^k \right)\\
	=&J_{\mathrm{hom}}^{(N)}(\omega,\gamma)+\dfrac{1}{N}\sum_{j=1}^{K}\sum_{\substack{i\in\{0,...,K-1\}\cup\\\{N-K+1,...,N-1\}}}\left(J_j\left(\omega,i,\frac{1}{j}\sum_{k=i}^{i+j-1}\tilde{z}_N^k \right)-J_j\left(\omega,i,\frac{1}{j}\sum_{k=i}^{i+j-1}z_N^k \right)\right).
	\end{split}
	\end{align}
	We now argue that the remainder converges to $0$ for $N\to\infty$. The second part of the sum can be easily estimated by $-J_j(\omega,i, \frac{1}{j}\sum_{k=i}^{i+j-1}z_N^k)\leq d$, due to (LJ2). Since each sum contains at most $K$ elements, the prefactor $\frac{1}{N}$ shows the convergence to zero.
	
	The first part of the sum needs a finer argument. Due to $\sup_N J_{\mathrm{hom}}^{(N)}(\omega,\gamma)<\infty$, we have $z_N^i>0$ for every $i=0,...,N-1$. 
		With this, we consider the first part of the sum $J_j(\omega,i,\frac{1}{j}\sum_{k=i}^{i+j-1}\tilde{z}_N^k)$. Now it holds true that $\tilde{z}_N^k=z$ for $i\leq K-2$ and $i\geq N-K+1$, and $\tilde{z}_N^k\geq(z-\gamma)(N/2-K+1)$ for $i=K-1$ and $i=N-K$, and $\tilde{z}_N^k\geq 0$ otherwise. Therefore, we get $\sum_{k=i}^{i+j-1}\tilde{z}_N^k\geq z$ for $N$ large enough. Therefore, $J_j(\omega,i,\frac{1}{j}\sum_{k=i}^{i+j-1}\tilde{z}_N^k)$ is bounded, due to \eqref{LJ2abschaetzung} from (LJ2). Since both sums contain at most $K$ elements, the prefactor $\frac{1}{N}$ yields the convergence to $0$.
	
	Since the remainders in \eqref{monotonicallydecreasingequation} vanish for $N\to\infty$, we get, with Proposition \ref{Prop:Jhomzklein},
			\begin{align*}
				J_{\mathrm{hom}}(z)=\lim_{N\rightarrow\infty}J_{\mathrm{hom}}^{(N)}(\omega,z)\leq \lim\limits_{N\to\infty}J_{\mathrm{hom}}^{(N)}(\omega,\gamma)=  J_{\mathrm{hom}}(\gamma),
			\end{align*}
	which is the desired result and finally shows that $J_{\mathrm{hom}}(z)=J_{\mathrm{hom}}(\gamma)$ for all $z\geq \gamma$. Together with \eqref{lowersemicont}, this shows that $J_{\mathrm{hom}}$ is monotonically decreasing.

	\step 6 $\Gamma$-limit, proof of \eqref{eq:gammalimjhom}.
	
	For $z\in\mathbb{R}$, let $(z_N)$ be a sequence with $z_N\to z$. Then, the definition of the approximation and Proposition \ref{Prop:jhomlcontinuous} yield
	\begin{align*}
	    \liminf_{N\to\infty}J_{\mathrm{hom}}^{(N)}(\omega,z_N,A)\geq \liminf_{N\to\infty}J_{\mathrm{hom}}^{L,(N)}(\omega,z_N,A)\geq J_{\mathrm{hom}}^{L}(z).
	\end{align*}
	Further, taking the limit $L\to\infty$ we get with Proposition \ref{Prop:Jhomzklein} $\liminf_{N\to\infty}J_{\mathrm{hom}}^{(N)}(\omega,z_N,A)\geq  J_{\mathrm{hom}}(z)$, which proves the liminf-inequality.
	
	For $z\in\mathbb{Z}$, take the constant recovery sequence $(z_N)$ with $z_N:=z$. Then it holds true that
	\begin{align*}
	    \limsup_{N\to\infty}J_{\mathrm{hom}}^{(N)}(\omega,z_N,A)=\limsup_{N\to\infty}J_{\mathrm{hom}}^{(N)}(\omega,z,A)=J_{\mathrm{hom}}(z),
	\end{align*}
	which proves the limsup-inequality and completes the proof of the $\Gamma$-limit.
\end{proof}

\begin{proof}[Proof of Proposition~\ref{Prop:Jhomzklein}] 	
For $z\notin (0,+\infty)$, we have $\lim\limits_{N\rightarrow\infty}J_{\mathrm{hom}}^{(N)}(\omega,z,A)=\infty$ and $\lim\limits_{L\rightarrow\infty}J_{\mathrm{hom}}^L(z)=\infty$, because of (LJ1) and the definition of the approximation. Hence, the assertion is proven in this case. 
	
Fix $\omega\in\Omega_0$, $z\in(0,+\infty)$ and $A=[a,b)$, $a,b\in\mathbb R$. The definition of the approximation $J_j^L$ yields $J_{\mathrm{hom}}^{L,(N)}(\omega,z,A)\leq J_{\mathrm{hom}}^{(N)}(\omega,z,A)$ and thus by Proposition~\ref{Prop:existencejhom}
\begin{align}\label{Leins}
 J_{\mathrm{hom}}^L(z)\leq \liminf_{N\rightarrow\infty}J_{\mathrm{hom}}^{(N)}(\omega,z,A).
\end{align}
In Lemma~\ref{Lem:limsup} below, we show for every $z\in\mathbb R$
\begin{align}\label{Lzwei}
 \liminf_{L\to\infty}J_{\mathrm{hom}}^L(z)\geq \limsup_{N\rightarrow\infty}J_{\mathrm{hom}}^{(N)}(\omega,z,A).
\end{align}
The inequalities \eqref{Leins} and \eqref{Lzwei} yield
\begin{align*}
 \lim_{L\to\infty}J_{\mathrm{hom}}^L(z)=\lim\limits_{N\rightarrow\infty}J_{\mathrm{hom}}^{(N)}(\omega,z,A),
 \end{align*}
 which proves the proposition.
\end{proof}
The following lemma contains the still remaining proof of the limit \eqref{Lzwei}.
\begin{lemma}
\label{Lem:limsup}
Assume that Assumption \ref{Ass:stochasticLJ} is satisfied. For every $z\in(0,+\infty)$ and $\omega\in\Omega_0$ it holds
\begin{align}\label{Lziel}
\liminf_{L\to\infty}J_{\mathrm{hom}}^L(z)\geq \limsup_{N\rightarrow\infty}J_{\mathrm{hom}}^{(N)}(\omega,z,A).
\end{align}
\end{lemma}

\begin{proof}
For simplicity, we consider $A=[0,1)$, the proof for a general interval is essentially the same. First, note that the assumption $z\in(0,+\infty)$ implies finite values of the energy. To show \eqref{Lziel}, we start for a given $z$ with a minimizer related to the minimum problem of $J_{\mathrm{hom}}^{L,(N)}(\omega,z)$, which we call $\tilde{z}_{L,N}=(\tilde{z}_{L,N}^0,...,\tilde{z}_{L,N}^{N-1})$, that is
\begin{align*}
J_{\mathrm{hom}}^{L,(N)}(\omega,z,[0,1))=\dfrac{1}{N}\sum_{j=1}^{K}\sum_{i=0}^{N-j}J_j^L\left(\omega,i,\frac{1}{j}\sum_{k=i}^{i+j-1}\tilde{z}_{L,N}^{k}\right).
\end{align*}
For $L,N\in\mathbb N$, we define
\begin{align*}
 I_{L,N}:=\{i\,|\, \tilde{z}_{L,N}^{i}<z_L \},
\end{align*}
which is the set of all indices $i$ with $\tilde{z}_{L,N}^{i}$ in the region where $J_j$ and $J_j^L$ differ. 

\step 1 We assert that
\begin{equation}\label{lim:zlzniC}
    \lim_{L\to\infty}\lim_{N\to\infty}\frac1N\sum_{i\in I_{L,N}}\left(z_L-\tilde{z}_{L,N}^{i}\right)=0.
\end{equation}
By definition of $I_{L,N}$ every term in the sum in \eqref{lim:zlzniC} is non-negative. Suppose that for some $\epsilon>0$ it holds 
\begin{equation}\label{lim:LNznieps}
    \limsup_{L\to\infty}\limsup_{N\to\infty}\frac1N\sum_{i\in I_{L,N}}\left(z_L-\tilde{z}_{L,N}^{i}\right)\geq\epsilon.
\end{equation}
Using Proposition~\ref{Prop:existencejhom} and (LJ2), we obtain 
\begin{align*}
    &J_{\mathrm{ hom}}^L(z)=\lim_{N\to\infty}\frac1N\sum_{j=1}^K\sum_{i=0}^{N-j}J_j\left(\tau_i\omega,\frac1j\sum_{k=i}^{i+j-1}\tilde z_{L,N}^i\right)\\
    \geq&-Kd+\limsup_{N\to\infty}\frac1N\sum_{i\in I_{L,N}}m_1^L(\tau_i\omega)(\tilde z_{L,N}^i-z_L)
    \geq-Kd+M^L\limsup_{N\to\infty}\frac1N\sum_{i\in I_{L,N}}(z_L-\tilde z_{L,N}^i),
\end{align*}
where the last inequality is due to Proposition \ref{prop:ML} (i). Hence, a combination of \eqref{schrankeableitungwachstum} and the assumption in \eqref{lim:LNznieps} yields
\begin{equation*}
    \limsup_{L\to\infty}J_{\mathrm{ hom}}^L(z)=\infty.
\end{equation*}
This is absurd in view of the estimate
\begin{equation*}
 J_{\mathrm{ hom}}^L(z)\leq Kd\max\{\Psi(z),|z|\}<\infty
\end{equation*}
being valid for every $L\in\mathbb N$, and thus the claim is proven.

\step 2 Conclusion

We provide a new sequence of competitors $(\hat z_{L,N})$ for the minimization problem in $J_{\mathrm{ hom}}^{L,(N)}(\omega,z)$ satisfying $\hat z_{L,N}^i\geq z_L$ for all $i\in\{0,\dots,N-1\}$ and
\begin{equation}\label{est:Ljjls2}
\lim_{L\to\infty}\lim_{N\to\infty}\frac{1}{N} \sum_{j=1}^K\sum_{i=0}^{N-j}\left(J_j^L\left(\tau_i\omega,\frac1j\sum_{k=i}^{i+j-1}\tilde z_{L,N}^k\right)-J_j^L\left(\tau_i\omega,\frac1j\sum_{k=i}^{i+j-1}\hat z_{L,N}^k\right)\right)\geq0.
\end{equation}
Obviously \eqref{est:Ljjls2} and $\hat z_{L,N}^i\geq z_L$ for all $i\in\{0,\dots,N-1\}$ imply the claim \eqref{Lziel}. Since $z_L\to 0$ for $L\to\infty$, it holds true that $z_L<z$ for $L$ big enough. 

In what follows we suppose that there exists $i\in\{0,\dots,N-1\}$ such that $\hat z_{L,N}^i< z_L$ (the other case is trivial). The constraint $\sum_{i=0}^{N-1}(\tilde{z}_N^i-z)=0$, implies $I_{L,N,z}:=\{i\,|\,z<\tilde{z}_{L,N}^{i}\}\neq\emptyset$  and we obtain
\begin{align}\label{setbiggerz}
\begin{split}
		 0=\sum_{i=0}^{N-1}(\tilde{z}_{L,N}^{i}-z)&\leq\sum_{i\in I_{L,N,z}}(\tilde{z}_{L,N}^{i}-z)+\sum_{i\in I_{L,N}}(\tilde{z}_{L,N}^{i}-z).
\end{split}
\end{align}
Combining \eqref{setbiggerz} and the assumption $z_L<z$, we find $v_N^i$ for $i=0,...,N-1$ with $0\leq v_N^i\leq\max\{\tilde{z}_N^i-z,0\}$ and
\begin{align}\label{etadef}
	\sum_{i\in I_{L,N}}\left(z_L-\tilde{z}_{L,N}^{i}\right)=\sum_{i=0}^{N-1}v_N^i.
\end{align}
Notice that by construction $v_N^i=0$ whenever $i\notin I_{L,N,z}$. Next, we define $\hat{z}_{L,N}$ by
\begin{align*}
 \hat{z}_{L,N}^{i}=\begin{cases}
	z_L&\quad\text{for}\ i\in I_{L,N},\\
	\tilde{z}_{L,N}^{i}-v_N^i&\quad\text{for}\ i\notin I_{L,N}.
\end{cases}
\end{align*}
By definition it holds $\hat z_{L,N}^i\geq z_L$ for every $i$ and $\hat z_{L,N}$ is a competitor for the minimization problem in the definition of $J_{\hom}^{L,(N)}$. Indeed, $\hat z_{L,N}^i=\tilde z_{L,N}^i=z$ for all $i\in\{0,\dots,K-1\}\cup\{N-K+1,\dots,N-1\}$ and
\begin{align*}
	&\sum_{i=0}^{N-1}\hat{z}_{L,N}^i=\sum_{i\in I_L}z_L+\sum_{i\notin I_L}(\tilde{z}_{L,N}^i-v_N^i)=\sum_{i=0}^{N-1}\tilde z_{L,N}^i+\sum_{i\in I_L}(z_L-\tilde z_{L,N}^i)-\sum_{i=0}^{N-1}v_N^i\stackrel{\eqref{etadef}}{=} \sum_{i=0}^{N-1}\tilde z_{L,N}^i.
\end{align*}
Fix $\hat\rho=\hat\rho(b,d,\Psi)\in(0,\frac1d]$ such that
\begin{equation}\label{def:hatrho}
 \frac1d\Psi(z)-d\geq b\quad\mbox{for all $z\leq \hat \rho$,}
\end{equation}
where $b,d$ and $\Psi$ are the constants and the convex function from the definition of $\mathcal J$. Further, we define $\rho_z:=\min\{\frac{z}K, \tfrac1{d},\hat\rho\}$. We consider for all $L$ sufficiently large such that $z_L<\rho_z$ the expression
\begin{equation*}
{\mathrm{ Diff}}_{i,j}^{L,N}:=J_j^L\left(\tau_i\omega,\frac1j\sum_{k=i}^{i+j-1}\tilde z_{L,N}^k\right)-J_j^L\left(\tau_i\omega,\frac1j\sum_{k=i}^{i+j-1}\hat z_{L,N}^k\right).
\end{equation*}
To show \eqref{est:Ljjls2}, we distinguish three cases:
\begin{itemize}
 \item \textit{Case (i):} $\frac1j\sum_{k=i}^{i+j-1}\tilde z_{L,N}^k\leq \frac1j\sum_{k=i}^{i+j-1}\hat z_{L,N}^k\leq \delta_j(\tau_i\omega)$. Since $J_j^L(\tau_i\omega,\cdot)$ is monotone decreasing on $(0,\delta_j(\tau_i\omega)]$ (see (LJ2)) it follows ${\mathrm{ Diff}}_{ij}^{L,N}\geq0$.
 \item \textit{Case (ii):} $\frac1j\sum_{k=i}^{i+j-1}\hat z_{L,N}^k\geq \delta_j(\tau_i\omega)$. It is ${\mathrm{ Diff}}_{i,j}^{L,N}\geq\frac{1}{d}\Psi\left( \frac1j\sum_{k=i}^{i+j-1}\tilde z_{L,N}^k\right)-d-b$. By the definition of $\hat \rho$ (see \eqref{def:hatrho}), we have either ${\mathrm{ Diff}}_{ij}^{L,N}\geq0$ or $\frac1j\sum_{k=i}^{i+j-1}\tilde z_{L,N}^k\geq \hat\rho\geq z_L$.
 \item \textit{Case (iii):}  $\frac1j\sum_{k=i}^{i+j-1}\hat z_{L,N}^k\leq\delta_j(\tau_i\omega)$ and $\frac1j\sum_{k=i}^{i+j-1}\hat z_{L,N}^k\leq \frac1j\sum_{k=i}^{i+j-1}\tilde z_{L,N}^k$. By the definition of $\hat z$ there exists $\hat k\in \{i,\dots,i+j-1\}$ such that $\tilde z_{L,N}^{\hat k}\geq z$ and thus $\frac1j\sum_{k=i}^{i+j-1}\hat z_{L,N}^k\geq \frac1Kz$, since $\hat z_{L,N}^k\geq0$ due to the finite value of the energy.
 \end{itemize}
Those indices $i$ where ${\mathrm{ Diff}}_{ij}^{L,N}\geq0$ holds true, do not pose a problem regarding the proof of \eqref{est:Ljjls2}. In order to conclude the proof of \eqref{est:Ljjls2}, we have to further consider \textit{Case (iii)} and the part of \textit{Case (ii)} where $\frac1j\sum_{k=i}^{i+j-1}\tilde z_{L,N}^k\geq \hat\rho$. For short, we name the set of those remaining indices $I_{\mathrm{rem}}$. For this, we need a finer estimation and define sets of \textit{small} and \textit{big} shifts. Let $\mu>0$, then it is
	 	\begin{align}
	 	\label{newsets}
	 	\begin{split}
	 	I_{L,N;j}&:=\left\{i\in\{0,...,N-j\}\;:\;\{i,...,i+j-1\}\cap I_L\neq\emptyset \right\},\\
	 	I_{L,N;j}^s&:=\{i\in \{0,...,N-j\}\setminus I_{L,N;j}\ |\ v_N^k< \mu\ \text{for all}\ k=i,...,i+j-1 \},\\
	 	I_{L,N;j}^b&:=\{0,...,N-j\}\setminus \left(I_{L,N;j}\cup I_{L,N;j}^s\right),\\
	 	\tilde{I}_{L,N;j}^s&:=\{i\in I_{L,N;j}\ |\  |\tilde{z}_N^k-\hat{z}_N^k|<\mu\ \text{for all}\ k=i,...,i+j-1 \},\\
	 	\tilde{I}_{L,N;j}^b&:=I_{L,N;j}\setminus \tilde{I}_{L,N;j}^s,
	 	\end{split}
	 	\end{align}
	 	We claim that for every $\mu>0$ and for every $j=1,...,K$ it holds true that
	 	\begin{align}\label{claim}
	 	\lim_{L\to\infty}\lim_{N\to\infty}|I_{L,N;j}^b|/N=0\quad\text{and}\quad \lim_{L\to\infty}\lim_{N\to\infty}|\tilde{I}_{L,N;j}^b|/N=0.
	 	\end{align}
		Indeed, by definition, we have
	 	\begin{align*}
	 	0\leq\dfrac{1}{N}\sum_{i\in I_{L,N,z}}v_N^i=\dfrac{1}{N}\sum_{i\in I_{L,N;1}^s}v_N^i+\dfrac{1}{N}\sum_{i\in I_{L,N;1}^b}v_N^i\stackrel{\eqref{etadef}}{=}\dfrac{1}{N}\sum_{i\in I_{L,N}}\left(z_L-\tilde{z}_{L,N}^{i}\right),
	 	\end{align*}
	 	and $\lim_{L\to\infty}\lim_{N\to\infty}\frac1N\sum_{i\in I_{L,N}}\left(z_L-\tilde{z}_{L,N}^{i}\right)=0$ from \eqref{lim:zlzniC}. Thus, $\lim_{L\to\infty}\lim_{N\to\infty}\frac{1}{N}\sum_{i\in I_s^1}v_N^i=0$ as well as $\lim_{L\to\infty}\lim_{N\to\infty}\frac{1}{N}\sum_{i\in I_{L,N;1}^b}v_N^i=0$ directly follow, because of $v_n^i\geq0$. In particular, since we have
	 	\begin{align*}
	 	0\leq\dfrac{|I_{L,N;1}^b|}{N}\mu\leq\dfrac{1}{N}\sum_{i\in I_{L,N;1}^b}v_N^i,
	 	\end{align*}
	 	it follows that for every $\mu>0$ it holds true that $\lim_{L\to\infty}\lim_{N\to\infty}|I_{L,N;1}^b|/N=0$. Since $|I_{L,N;j}^b|\leq K|I_{L,N;1}^b|$, we also get $\lim_{L\to\infty}\lim_{N\to\infty}|I_{L,N;j}^b|/N=0$ for every $j=1,...,K$. In an analogous way, we have	
	 	\begin{align*}
	 		0\leq\dfrac{1}{N}\sum_{i\in \tilde{I}_{L,N;1}^s}\left(z_L-\tilde{z}_N^{i}\right)+\dfrac{1}{N}\sum_{i\in \tilde{I}_{L,N;1}^b}\left(z_L-\tilde{z}_N^{i}\right)=\dfrac{1}{N}\sum_{i\in I_{L,N}}\left(z_L-\tilde{z}_N^{i}\right),
	 	\end{align*}
	 	and $\lim_{L\to\infty}\lim_{N\to\infty}\frac1N\sum_{i\in I_{L,N}}\left(z_L-\tilde{z}_{L,N}^{i}\right)=0$ from \eqref{lim:zlzniC}. Therefore, we can directly deduce $\lim_{L\to\infty}\lim_{N\to\infty}\frac1N\sum_{i\in \tilde{I}_{L,N;1}^b}\left(z_L-\tilde{z}_{L,N}^{i}\right)=0$. Together with 	
	 	\begin{align*}
	 	0\leq\dfrac{|\tilde{I}_{L,N;1}^b|}{N}\mu\leq\dfrac{1}{N}\sum_{i\in \tilde{I}_{L,N;1}^b}\left(z_L-\tilde{z}_{L,N}^{i}\right),
	 	\end{align*}
	 	this yields $\lim_{L\to\infty}\lim_{N\to\infty}|\tilde{I}_{L,N;1}^b|/N=0$. Since $|\tilde{I}_{L,N;j}^b|\leq K|\tilde{I}_{L,N;1}^b|+K|I_{L,N;1}^b|$ we also get $\lim_{L\to\infty}\lim_{N\to\infty}|\tilde{I}_{L,N;j}^b|/N=0$, for every $j=1,...,K$. This concludes the proof of claim \eqref{claim}. \\ 

	Now, we consider \eqref{est:Ljjls2} for the remaining indices $I_{\mathrm{rem}}$, separately for the previously defined small and big shift sets. We start with the big ones. By definition of \textit{Case(ii) and (iii)}, it is $\hat{z}_N^i\geq \rho_z$ for all $i\in I_{\mathrm{rem}}$ and using (LJ2), we get
	\begin{align*}
		\mathrm{Diff}_{i,j}^{L,N}=J_j^L\left(\omega,i,\frac{1}{j}\sum_{k=i}^{i+j-1}\tilde{z}_N^{k}\right)-J_j^L\left(\omega,i,\frac{1}{j}\sum_{k=i}^{i+j-1}\hat{z}_N^{k}\right)\geq -d-d\max\left\{\Psi(\rho_z),|\rho_z|,b \right\}=:-C_z
	\end{align*}
	and therefore for \eqref{est:Ljjls2}
	\begin{align*}
			&\dfrac{1}{N}\sum_{j=1}^{K}\sum_{i\in \left(I_{L,N;j}^b\cup \tilde{I}_{L,N;j}^b\right)\cap I_{\mathrm{rem}}}\left(J_j^L\left(\omega,i,\frac{1}{j}\sum_{k=i}^{i+j-1}\tilde{z}_N^{k}\right)-J_j^L\left(\omega,i,\frac{1}{j}\sum_{k=i}^{i+j-1}\hat{z}_N^{k}\right)\right)\\
			&\geq  \dfrac{1}{N}\sum_{j=1}^{K}\sum_{i\in \left(I_{L,N;j}^b\cup \tilde{I}_{L,N;j}^b\right)\cap I_{\mathrm{rem}}}-C_z\geq- \dfrac{|I_{L,N;j}^b|+|\tilde{I}_{L,N;j}^b|}{N}K(-C_z).
	\end{align*}	
	Together with \eqref{claim} this shows \eqref{est:Ljjls2} for the big shift sets. The small shift sets yield by definition $\hat{z}_N^i\geq \rho_z$ and $\tilde{z}_N^i\geq \rho_z$ and allow for the following calculation: with \eqref{lipschitzhoelderestimate} and for $\mu<1$ we get
		\begin{align*}
		&\dfrac{1}{N}\sum_{j=1}^{K}\sum_{i\in \left(I_{L,N;j}^s\cup \tilde{I}_{L,N;j}^s\right)\cap I_{\mathrm{rem}}}\left(J_j^L\left(\omega,i,\frac{1}{j}\sum_{k=i}^{i+j-1}\tilde{z}_N^{k}\right)-J_j^L\left(\omega,i,\frac{1}{j}\sum_{k=i}^{i+j-1}\hat{z}_N^{k}\right)\right)\\
		\geq&-\dfrac{1}{N}\sum_{j=1}^{K}\sum_{i\in \left(I_{L,N;j}^s\cup \tilde{I}_{L,N;j}^s\right)\cap I_{\mathrm{rem}}}\left|J_j^L\left(\omega,i,\frac{1}{j}\sum_{k=i}^{i+j-1}\tilde{z}_N^{k}\right)-J_j^L\left(\omega,i,\frac{1}{j}\sum_{k=i}^{i+j-1}\hat{z}_N^{k}\right)\right|\\
		\geq& -\dfrac{1}{N}\sum_{j=1}^{K}\sum_{i\in \left(I_{L,N;j}^s\cup \tilde{I}_{L,N;j}^s\right)\cap I_{\mathrm{rem}}} \left[ J_j(\tau_i\omega) \right]_{C^{0,\alpha}(\rho_z,\infty)} \mu^\alpha\geq -\dfrac{1}{N}\sum_{j=1}^{K}\sum_{i=0}^{N-1} \left[ J_j(\tau_i\omega) \right]_{C^{0,\alpha}(\rho_z,\infty)} \mu^\alpha,
		\end{align*}
		because $|\tilde{z}_N^{k}-\hat{z}_N^{k}|=v_N^k<\mu$ or $|\tilde{z}_N^{k}-\hat{z}_N^{k}|=|\tilde{z}_N^{k}-z_L|<\mu$, by the definition of the small shift set. Here, $\left[ \cdot \right]_{C^{0,\alpha}(\rho_z,\infty)} $ is the Hölder coefficient. Now, (H1), Proposition \ref{Prop:averages} and Lemma \ref{lem:lipschitz} yield for fixed $\mu>0$
		\begin{align*}
		 \lim\limits_{L\to\infty}\lim\limits_{N\to\infty}\dfrac{1}{N}\sum_{j=1}^{K}\sum_{i=0}^{N-1} \left[ J_j(\tau_i\omega) \right]_{C^{0,\alpha}(\rho_z,\infty)} \mu^\alpha=\lim\limits_{L\to\infty} \mu^{\alpha} C(\rho_z)=\mu^{\alpha} C(\rho_z),
		\end{align*}
		with a constant $C(\rho_z)$ independent of $L$. As this holds for every $\mu>0$, we can afterwards take the limit $\mu\rightarrow0$, which shows \eqref{est:Ljjls2} for the small shift sets and concludes the proof.		
\end{proof}

\subsection{The case of nearest neighbor interactions, proof of Proposition~\ref{Prop:Jhomzgross}} \label{sec:NN}
\begin{proof}[Proof of Proposition~\ref{Prop:Jhomzgross}]
 For easier notation, we set $\mathbb{E}[J(\delta)]:=\mathbb{E}[J_1(\delta)]$, $\mathbb{E}[\delta]:=\mathbb{E}[\delta_1]$, $\delta(\omega):=\delta_1(\omega)$ and $J(\omega,i,z):=J_1(\omega,i,z)$. Further, we consider $A=[0,N)$, which is no restriction since we have shown in Proposition \ref{Prop:Jhomzklein} that the limit $J_{\mathrm{hom}}$ as well as $J_{\mathrm{hom}}^L$ are independent of $A$. The proof is divided into three steps.\\
		
	\textit{Step 1:} First, we determine the absolute minimum of $J_{\mathrm{hom}}^{(N)}(\omega,\cdot)$ for every $N\in\mathbb{N}$. Recalling the definition of $J_{\mathrm{hom}}^{(N)}$ as the infimum of the sum over individual functions $J(\omega,i,z_N^i)$ under the constraint $\sum_{i=0}^{N-1}z_N^i=Nz$, the absolute minimum is achieved when every single function $J(\omega,i,\cdot)$ takes on its own minimum. Therefore, the absolute minimum  is achieved at the minimum point $z_{\mathrm{min}}=\frac{1}{N}\sum_{i=0}^{N-1}\delta(\tau_i\omega)$. This reads
	\begin{align*}
	\min_{z\in\mathbb{R}}\left\{J_{\mathrm{hom}}^{(N)}(\omega,z)\right\}=J_{\mathrm{hom}}^{(N)}\left(\omega,z_{\mathrm{min}} \right)=J_{\mathrm{hom}}^{(N)}\left(\omega,\dfrac{1}{N}\sum_{i=0}^{N-1}\delta(\tau_i\omega) \right)=\dfrac{1}{N}\sum_{i=0}^{N-1}J(\omega,i,\delta(\tau_i\omega)).
	\end{align*}
	As this is the absolute minimum, we can conclude
	\begin{align*}
	J_{\mathrm{hom}}^{(N)}\left(\omega,z\right)\geq\dfrac{1}{N}\sum_{i=0}^{N-1}J(\omega,i,\delta(\tau_i\omega))\quad\text{for every}\ z\in\mathbb{R}.
	\end{align*}
	Taking the limit $\liminf_{N\rightarrow\infty}$, we then get with ergodicity and Proposition \ref{Prop:averages}
	\begin{align}
	\label{ziel1}
	\liminf_{N\rightarrow\infty}J_{\mathrm{hom}}^{(N)}(\omega,z)\geq\mathbb{E}[J(\delta)]\quad\text{for every}\ z\in\mathbb{R}.
	\end{align}
	
	\textit{Step 2:} We need to prove $J_{\mathrm{hom}}(\mathbb{E}[\delta])=\mathbb{E}[J(\delta)]$. By \eqref{ziel1}, it is only left to show
	\begin{align}
	\label{ziel2}
	\limsup_{N\rightarrow\infty}J_{\mathrm{hom}}^{(N)}(\omega,\mathbb{E}[\delta])\leq\mathbb{E}[J(\delta)],
	\end{align}
	to conclude the assertion. We start with
	\begin{align}
	\label{f1f2}
	\begin{split}
	J_{\mathrm{hom}}^{(N)}\left(\omega,\mathbb{E}[\delta]\right)&=J_{\mathrm{hom}}^{(N)}(\omega,\mathbb{E}[\delta])-\dfrac{1}{N}\sum_{i=0}^{N-1}J(\omega,i,\delta(\tau_i\omega))+\dfrac{1}{N}\sum_{i=0}^{N-1}J(\omega,i,\delta(\tau_i\omega))\\
	&=:f(\omega,N)+\dfrac{1}{N}\sum_{i=0}^{N-1}J(\omega,i,\delta(\tau_i\omega)),
	\end{split}
	\end{align}
	where $f(\omega,N):=J_{\mathrm{hom}}^{(N)}(\omega,\mathbb{E}[\delta])-\frac{1}{N}\sum_{i=0}^{N-1}J(\omega,i,\delta(\tau_i\omega))$. Then, we get
	\begin{align*}
	\left|f(\omega,N)\right|&=\left|J_{\mathrm{hom}}^{(N)}(\omega,\mathbb{E}[\delta])-\dfrac{1}{N}\sum_{i=0}^{N-1}J(\omega,i,\delta(\tau_i\omega))\right|\\
	&=\left|J_{\mathrm{hom}}^{(N)}\left(\omega,\underbrace{\lim\limits_{k\rightarrow\infty}\dfrac{1}{k}\sum_{i=0}^{k-1}\delta(\tau_i\omega)-\dfrac{1}{N}\sum_{i=0}^{N-1}\delta(\tau_i\omega)}_{=:\beta(\omega,N)}+\dfrac{1}{N}\sum_{i=0}^{N-1}\delta(\tau_i\omega) \right)-\dfrac{1}{N}\sum_{i=0}^{N-1}J(\omega,i,\delta(\tau_i\omega))\right|,
	\end{align*}
	with $\beta(\omega,N)\rightarrow0$ for $N\rightarrow\infty$, because of $\mathbb{E}[\delta]=\lim\limits_{N\to\infty}\frac{1}{N}\sum_{i=0}^{N-1}\delta(\tau_i\omega)$. By setting
	\begin{align*}
	z_N^i:=\delta(\tau_i\omega)+\beta(\omega,N)
	\end{align*}
	which fulfils the constraint $\sum_{i=0}^{N-1}z_N^i=N\left(\beta(\omega,N)+\dfrac{1}{N}\sum_{i=0}^{N-1}\delta(\tau_i\omega) \right)$ of the infimum problem, we can further estimate
	\begin{align*}
	\left|f(\omega,N)\right|&=\left|J_{\mathrm{hom}}^{(N)}\left(\omega,\beta(\omega,N)+\dfrac{1}{N}\sum_{i=0}^{N-1}\delta(\tau_i\omega) \right)-\dfrac{1}{N}\sum_{i=0}^{N-1}J(\omega,i,\delta(\tau_i\omega))\right|\\
	&\leq\left| \dfrac{1}{N}\sum_{i=0}^{N-1}J(\omega,i,\delta(\tau_i\omega)+\beta(\omega,N))-\dfrac{1}{N}\sum_{i=0}^{N-1}J(\omega,i,\delta(\tau_i\omega))\right|\\
	&\leq\dfrac{1}{N}\sum_{i=0}^{N-1}\Big|J(\omega,i,\delta(\tau_i\omega)+\beta(\omega,N))-J(\omega,i,\delta(\tau_i\omega)) \Big|.
	\end{align*}
	Due to the special choice of $z_N^i$, we have $\delta(\tau_i\omega)+\beta(\omega,N)\in(0,+\infty)$ for every $i\in0,...,N-1$ for $N$ large enough, because of two reasons. First, it is $\mathrm{dom}J(\omega,i,\cdot)=(0,+\infty)$ for all $J(\omega,i,\cdot)$ and $\delta(\tau_i\omega)\in(\frac{1}{d},d)\subset(0,+\infty)$ for all $i$, due to (LJ2). The second reason is that $\beta(\omega,N)\rightarrow0$ for $N\rightarrow\infty$. Therefore, it exists $0<\xi<\frac{1}{d}$ such that for $N$ big enough it holds true that $\delta(\tau_i\omega)+\beta(\omega,N)\in(\frac{1}{d}-\xi,+\infty)\subset(0,+\infty)$. Due to (LJ1), $J\in C^{1,\alpha}$ on $(\frac{1}{d}-\xi,\infty)$ and we can continue our estimate as follows:
	\begin{align*}
	\left|f(\omega,N)\right|&\leq |\beta(\omega,N)|^{\alpha}\dfrac{1}{N}\sum_{i=0}^{N-1} \left[J_j(\tau_i\omega,\cdot)\right]_{C^{0,\alpha}(\frac{1}{d}-\xi,+\infty)}\\& \leq |\beta(\omega,N)|^{\alpha} \max\left\{\dfrac{1}{N}\sum_{i=0}^{N-1}\left[J_j(\tau_i\omega,\cdot)\right]_{C^{0,\alpha}(\delta(\tau_i\omega),+\infty)},C_{\mathrm{Lip}}\left(\frac{1}{d}-\xi\right)\right\} \rightarrow 0\quad\text{for}\ N\rightarrow\infty,
	\end{align*}
	because of (H1), Proposition \ref{Prop:averages}, Lemma \ref{lem:lipschitz} and $\beta(\omega,N)\rightarrow0$ for $N\rightarrow\infty$.\\
	
	By the result $f(\omega,N)\to0$ and by $\frac{1}{N}\sum_{i=0}^{N-1}J(\omega,i,\delta(\tau_i\omega))\to \mathbb{E}[J(\delta)] $ due to Proposition \ref{Prop:averages}, we can calculate $\limsup_{N\to\infty}$ in \eqref{f1f2} and get
	\begin{align*}
	\limsup_{N\rightarrow\infty}J_{\mathrm{hom}}^{(N)}(\omega,\mathbb{E}[\delta])&\leq\lim\limits_{N\rightarrow\infty}f(\omega,N)+\lim\limits_{N\rightarrow\infty}\dfrac{1}{N}\sum_{i=0}^{N-1}J(\omega,i,\delta(\tau_i\omega))\rightarrow 0+\mathbb{E}[J(\delta)].
	\end{align*}
	This shows \eqref{ziel2} and, as said before, together with \eqref{ziel1} this yields that
	\begin{align*}
	J_{\mathrm{hom}}(\mathbb{E}[\delta])=\mathbb{E}[J(\delta)].
	\end{align*}
	
	\textit{Step 3:} 
	\begin{figure}[t]
	\centering
			\includegraphics[width=0.4\linewidth]{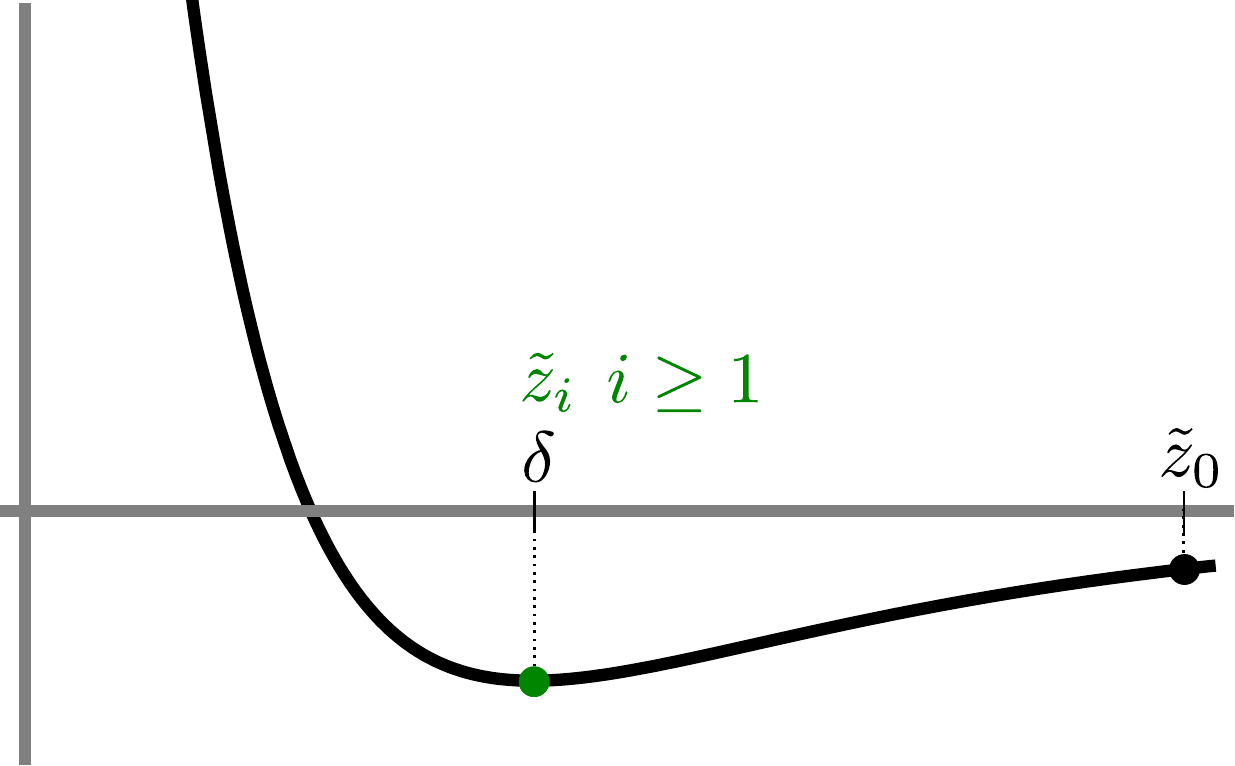}
			\caption{The new candidate $\tilde{z}$ for the minimizer.}
	\label{fig:JhomGleichEjDelta}
	\end{figure}
	
	We need to show $J_{\mathrm{hom}}(z)=\mathbb{E}[J(\delta)]$ for every $z>\mathbb{E}[\delta]$. For this, consider $z>\mathbb{E}[\delta]$. We set
	\begin{align*}
	\tilde{z}_N^i=\begin{cases}
	Nz-\sum_{i=1}^{N}\delta(\tau_i\omega)&\quad\text{for}\ i=0,\\
	\delta(\tau_i\omega)&\quad\text{for}\ i=1,...N-1,
	\end{cases}
	\end{align*}
	which fulfils the constraint $\sum_{i=0}^{N-1}\tilde{z}_N^i=Nz$, and is shown in Figure \ref{fig:JhomGleichEjDelta}. Then, it holds true that
	\begin{align*}
	J_{\mathrm{hom}}^{(N)}(\omega,z)&\leq\dfrac{1}{N-1}\sum_{i=0}^{N-1}J\left(\omega,i,\tilde{z}_N^i\right)=\dfrac{1}{N}\sum_{i=1}^{N-1}J(\omega,i,\delta(\tau_i\omega))+\dfrac{1}{N}J\left(\omega,0,Nz-\sum_{i=1}^{N-1}\delta(\tau_i\omega)\right)\\
	&=\dfrac{1}{N}\sum_{i=0}^{N-1}J(\omega,i,\delta(\tau_i\omega))-\dfrac{1}{N}J(\omega,0,\delta(\tau_0\omega))+\dfrac{1}{N}J\left(\omega,0,Nz-\sum_{i=1}^{N-1}\delta(\tau_i\omega)\right).
	\end{align*}
	With $\frac{1}{N}\sum_{i=0}^{N-1}\delta(\tau_i\omega)\to\mathbb{E}[\delta]$, it holds true that $Nz-\sum_{i=1}^{N-1}\delta(\tau_i\omega)=N\left(z-\frac{1}{N}\sum_{i=0}^{N-1}\delta(\tau_i\omega) \right)+\delta(\tau_0\omega)\geq NC\rightarrow\infty$ for $N\rightarrow\infty$ and therefore $J\left(\omega,0,Nz-\sum_{i=1}^{N-1}\delta(\tau_i\omega)\right)\rightarrow0$, due to (LJ1). With this, we get by taking $\limsup_{N\rightarrow\infty}$
	\begin{align*}
	\limsup_{N\rightarrow\infty}J_{\mathrm{hom}}^{(N)}(\omega,z)\leq\mathbb{E}[J(\delta)],
	\end{align*}
	which shows \eqref{ziel2}. Together with \eqref{ziel1} and Proposition \ref{Prop:Jhomzklein}, we get
	\begin{align*}
	J_{\mathrm{hom}}(z)=\lim\limits_{N\to\infty}J_{\mathrm{hom}}^{(N)}(\omega,z)=\mathbb{E}[J(\delta)],\quad\text{for all}\ z\geq \mathbb{E}[\delta].
	\end{align*}
	The proof for $J_{\mathrm{hom}}^L(z)$ is exactly the same.
\end{proof}

\appendix

\section{Appendix}

\begin{theorem}[Interpolation I] 
\label{thm:interpolationaffine}
	Let $u\in BV(0,1)$. For $\delta>0$ let $N\in\mathbb{N}$ and $(t_j)_{j=0,...,N}\subset[0,1]$ be such that $t_0=0$, $t_N=1$, $\delta<t_{j+1}-t_j<2\delta$, $t_j$ is not in the jump set of $u$. Let $v_N$ be the piecewise affine interpolation of $u$ with grid points $t_j$, which means $v_N\in C(0,1)$ is affine on $[t_{j-1},t_j)$ and it holds $v_N(t_j)=u(t_j)$ for all $j=0,...,N$. Then, it holds $v_N\rightharpoonup^*u$ in $BV(0,1)$ for $\delta\to0$.
\end{theorem}

\begin{theorem}[Interpolation II]
\label{Thm:interpolationConstant} 
	Let $(u_n)\subset \mathcal{A}_n$ be a sequence of piecewise affine functions weakly$^*$ converging to $u$ in $BV(0,1)$ with $\sup_n\lVert u_n'\rVert_{L^1(0,1)}<\infty$. Let $(\hat{u}_n)$ be the sequence of piecewise constant functions defined by $\hat{u}_n(i/n)=u_n(i/n)$ with $\hat{u}_n$ is constant on $[i,i+1)\frac 1n$, $i\in\{0,1,...,n-1\}$. Then, $(\hat{u}_n)$ converges weakly$^*$ to $u$ in $BV(0,1)$.
\end{theorem}

\begin{theorem}[Subadditive ergodic theorem, Akcoglu and Krengel, {\cite{AkcogluKrengel1981}}]
\label{Thm:AkcogluKrengel}
Let $F:\mathcal{I}\rightarrow L^1(\Omega)$ be a subadditive stochastic process and let $\{I_k\}_{k\in\mathbb{N}}$ be a regular family of sets in $\mathcal{I}$ with $\lim\limits_{k\rightarrow\infty}I_k=\mathbb{R}$. If $F$ is stationary w.r.t.~a measure preserving group action $\{\tau_z\}_{z\in\mathbb{Z}}$, that is 
\begin{align}
\forall I\in\mathcal{I},\ \forall z\in\mathbb{Z},\ F(I+z;\omega)=F(I;\tau_z\omega)\ \text{almost surely}, 
\end{align}
then there exists $\phi:\omega\rightarrow\mathbb{R}$ such that for $\mathbb{P}$-almost every $\omega$
\[\lim\limits_{k\rightarrow\infty}\dfrac{F(I_k;\omega)}{|I_k|}=\phi(\omega).  \]
Further, if $\{\tau_z\}_{z\in\mathbb{Z}}$ is ergodic, then $\phi$ is constant.
\end{theorem}

\begin{theorem}[Attouch Lemma,{\cite[Cor.~1.16]{Attouch1984}}]
\label{attouch}
Let $(a_{n,m})_{n\in\mathbb{N},m\in\mathbb{N}}$ be a doubly indexed sequence in $\overline{\mathbb{R}}$. Then, there exists a mapping $n\mapsto m(n)$, increasing to $+\infty$, such that
\begin{align*}
\limsup_{n\to\infty}a_{n,m(n)}\leq\limsup_{m\to\infty}\left( \limsup_{n\to\infty}a_{n,m}\right).
\end{align*}
\end{theorem}

\begin{theorem}[{\cite[Thm.~1.62]{Gelli}}]
\label{PropGelli}
Let $f:\mathbb{R}\rightarrow\mathbb{R}\cup \{+\infty\}$ be convex, lower semicontinuous, monotone decreasing with
\begin{align*}
\lim\limits_{z\rightarrow-\infty}\dfrac{f(z)}{|z|}=+\infty\quad\text{and}\quad \lim\limits_{z\rightarrow+\infty}f(z)=c\in\mathbb{R}.
\end{align*}
Let $F:BV(a,b)\rightarrow\mathbb{R}\cup\{+\infty\}$ be defined as
\begin{align*}
F(u):=\begin{cases}
\int_{a}^{b}f(u')\,\mathrm{d}x&\quad\text{if}\ u\in W^{1,1}(0,1),\\
+\infty&\quad\text{else.}
\end{cases}
\end{align*}
Let the functional $\mathcal{F}:BV(a,b)\rightarrow\mathbb{R}\cup\{+\infty\}$ be defined as
\begin{align*}
\mathcal{F}(u):=\begin{cases}
\int_{a}^{b}f(u')\,\mathrm{d}x&\quad\text{if}\ u\in BV(a,b),\ D^su\geq0,\\
+\infty&\quad\text{else.}
\end{cases}
\end{align*}
Let $\overline{F}$ denote the lower semicontinuous envelope of $F$ with respect to the weak$^*$ convergence in $BV(a,b)$. Then it holds $\mathcal{F}\equiv\overline{F}$.
\end{theorem}

\textbf{Acknowledgments.} LL gratefully acknowledges the kind hospitality of the Technische Universität Dresden during her research visits. AS would like to thank the Isaac Newton Institute for Mathematical Sciences for support and hospitality during the programme
``The Mathematical Design of New Materials''
when work on this paper was undertaken. This programme was supported by EPSRC grant number EP/R014604/1.

\end{document}